\documentclass[12pt]{article}
\usepackage{amsmath,amssymb}
\usepackage{amsthm}
\usepackage{epsfig}
\usepackage{graphics}
\usepackage{graphicx}
\usepackage{colordvi}
\usepackage{mathrsfs} 
\usepackage{stmaryrd}
\usepackage{geometry}
\usepackage{dsfont}
\usepackage{tikz}
\usetikzlibrary{automata}
\usepackage{url}
\usepackage
{hyperref}
\hypersetup{
colorlinks=true,
linkcolor=blue!75!black,
filecolor=blue!75!black,      
urlcolor=blue!75!black, 
citecolor=blue!75!black,
}
\usepackage{enumitem}
\usepackage[textsize=small]{todonotes}
\usepackage[margin=1cm, size=small]{caption}
\usepackage{multirow}
\allowdisplaybreaks[4]
\oddsidemargin
-.5in \evensidemargin -.5in\marginparwidth 0.4in
\usepackage{color,soul}
\usepackage{xcolor}

\usepackage{etoolbox}
\patchcmd{\thebibliography}{\section*}{\section}{}{}
\usepackage{empheq}
\usepackage[framemethod=tikz]{mdframed}
\usepackage{comment}
\specialcomment{notes}{\begingroup 
\begin{mdframed}[linecolor=red!75!black,linewidth=2pt, ]
}{\end{mdframed}\endgroup 
\vspace{.5cm}
}
\excludecomment{notes}
\usepackage{caption}
\setstcolor{red}

\definecolor{blue2}{cmyk}{.94,.11,0,0}
\definecolor{myblue}{rgb}{.8, .8, 1}
\newlength\mytemplen
\newsavebox\mytempbox

\newcommand{\g}{{\mathfrak g}}
\newcommand{\e}{{\mathrm e}}
\newcommand{\1}{\mathds 1}

\newcommand{\bP}{{\mathbf P}}

\newcommand{\vep}{{\varepsilon}}

\newcommand{\lra}{\longrightarrow}

\newcommand{\rest}{{\upharpoonright}}

\newcommand{\bs}{\boldsymbol}
\newcommand{\ms}{\mathscr}
\renewcommand{\P}{{\mathbb P}}
\newcommand{\E}{{\mathbb E}}

\renewcommand{\d}{{\mathrm d}}

\newcommand{\bE}{{\mathbf E}}

\newcommand{\Z}{{\Bbb Z}}
\newcommand{\R}{{\Bbb R}}

\newcommand{\Id}{{\rm Id}}

\newcommand{\Cov}{{\rm Cov}}
\newcommand{\lf}{\lfloor}
\newcommand{\rf}{\rfloor}

\renewcommand{\i}{{\mathtt i}}
\newcommand{\defeq}{{\stackrel{\rm def}{=}}}

\renewenvironment{proof}[1][\proofname]{\noindent {\bfseries #1.}\;}{\hfill\ensuremath{\blacksquare}\\}

\newcommand{\Log}{{\,\rm Log\,}}

\newcommand{\TV}{{\rm TV}}

\newcommand{\less}{\lesssim}
\newcommand{\more}{\gtrsim}
\renewcommand{\div}{{\rm div}}
\usepackage{currfile}
\usepackage{fancyhdr}
\newtheoremstyle{slantthm}{10pt}{10pt}{\slshape}{}{\bfseries}{}{.5em}{\thmname{#1}\thmnumber{ #2}\thmnote{ (#3)}.}
\newtheoremstyle{slantrmk}{10pt}{10pt}{\rmfamily}{}{\bfseries}{}{.5em}{\thmname{#1}\thmnumber{ #2}\thmnote{ (#3)}.}

\begin{document}
\theoremstyle{slantthm}
\newtheorem*{mthm}{Theorem}
\newtheorem{thm}{Theorem}[section]
\newtheorem{prop}[thm]{Proposition}
\newtheorem{lem}[thm]{Lemma}
\newtheorem{cor}[thm]{Corollary}
\newtheorem{defi}[thm]{Definition}
\newtheorem{disc}[thm]{Discussion}
\newtheorem{app}[thm]{Approximation hypothesis}
\newtheorem{hypothesis}[thm]{Hypothesis}
\newtheorem{conj}[thm]{Conjecture}

\theoremstyle{slantrmk}
\newtheorem{ass}[thm]{Assumption}
\newtheorem{rmk}[thm]{Remark}
\newtheorem{eg}[thm]{Example}
\newtheorem{que}[thm]{Question}
\numberwithin{equation}{section}
\numberwithin{figure}{section}
\newtheorem{quest}[thm]{Quest}
\newtheorem{prob}[thm]{Problem}
\newtheorem{ann}[thm]{Annotation}
\newtheorem{nota}[thm]{Notation}
\newtheorem{term}[thm]{Terminology}

\makeatletter
\newcommand\cb{%
    \@ifnextchar[
       {\@cb}%
       {\@cb[5pt]}}

\def\@cb[#1]{%
    \@ifnextchar[
       {\@@cb[#1]}%
       {\@@cb[#1][5pt]}}

\def\@@cb[#1][#2]#3{
    \sbox\mytempbox{#3}%
    \mytemplen\ht\mytempbox
    \advance\mytemplen #1\relax
    \ht\mytempbox\mytemplen
    \mytemplen\dp\mytempbox
    \advance\mytemplen #2\relax
    \dp\mytempbox\mytemplen
    \colorbox{gray!40!white}{\hspace{1em}\usebox{\mytempbox}\hspace{1em}}}
\makeatother

\makeatletter
\newcommand*{\mint}[1]{%
  \mint@l{#1}{}%
}
\newcommand*{\mint@l}[2]{%
  \@ifnextchar\limits{%
    \mint@l{#1}%
  }{%
    \@ifnextchar\nolimits{%
      \mint@l{#1}%
    }{%
      \@ifnextchar\displaylimits{%
        \mint@l{#1}%
      }{%
        \mint@s{#2}{#1}%
      }%
    }%
  }%
}
\newcommand*{\mint@s}[2]{%
  \@ifnextchar_{%
    \mint@sub{#1}{#2}%
  }{%
    \@ifnextchar^{%
      \mint@sup{#1}{#2}%
    }{%
      \mint@{#1}{#2}{}{}%
    }%
  }%
}
\def\mint@sub#1#2_#3{%
  \@ifnextchar^{%
    \mint@sub@sup{#1}{#2}{#3}%
  }{%
    \mint@{#1}{#2}{#3}{}%
  }%
}
\def\mint@sup#1#2^#3{%
  \@ifnextchar_{%
    \mint@sup@sub{#1}{#2}{#3}%
  }{%
    \mint@{#1}{#2}{}{#3}%
  }%
}
\def\mint@sub@sup#1#2#3^#4{%
  \mint@{#1}{#2}{#3}{#4}%
}
\def\mint@sup@sub#1#2#3_#4{%
  \mint@{#1}{#2}{#4}{#3}%
}
\newcommand*{\mint@}[4]{%
  \mathop{}%
  \mkern-\thinmuskip
  \mathchoice{%
    \mint@@{#1}{#2}{#3}{#4}%
        \displaystyle\textstyle\scriptstyle
  }{%
    \mint@@{#1}{#2}{#3}{#4}%
        \textstyle\scriptstyle\scriptstyle
  }{%
    \mint@@{#1}{#2}{#3}{#4}%
        \scriptstyle\scriptscriptstyle\scriptscriptstyle
  }{%
    \mint@@{#1}{#2}{#3}{#4}%
        \scriptscriptstyle\scriptscriptstyle\scriptscriptstyle
  }%
  \mkern-\thinmuskip
  \int#1%
  \ifx\\#3\\\else_{#3}\fi
  \ifx\\#4\\\els\e^{#4}\fi  
}
\newcommand*{\mint@@}[7]{%
  \begingroup
    \sbox0{$#5\int\m@th$}%
    \sbox2{$#5\int_{}\m@th$}%
    \dimen2=\wd0 %
    \let\mint@limits=#1\relax
    \ifx\mint@limits\relax
      \sbox4{$#5\int_{\kern1sp}^{\kern1sp}\m@th$}%
      \ifdim\wd4>\wd2 %
        \let\mint@limits=\nolimits
      \else
        \let\mint@limits=\limits
      \fi
    \fi
    \ifx\mint@limits\displaylimits
      \ifx#5\displaystyle
        \let\mint@limits=\limits
      \fi
    \fi
    \ifx\mint@limits\limits
      \sbox0{$#7#3\m@th$}%
      \sbox2{$#7#4\m@th$}%
      \ifdim\wd0>\dimen2 %
        \dimen2=\wd0 %
      \fi
      \ifdim\wd2>\dimen2 %
        \dimen2=\wd2 %
      \fi
    \fi
    \rlap{%
      $#5%
        \vcenter{%
          \hbox to\dimen2{%
            \hss
            $#6{#2}\m@th$%
            \hss
          }%
        }%
      $%
    }%
  \endgroup
}

\title{\vspace{-.5cm}\bf Convergence of the rescaled Whittaker stochastic differential equations and independent sums}

\author{Yu-Ting Chen\footnote{Department of Mathematics and Statistics, University of Victoria, British Columbia, Canada.}\;\footnote{Email: \url{chenyuting@uvic.ca}}}

\date{\today}

\maketitle

\vspace{-.5cm}
\abstract{We study some SDEs derived from the $q\to 1$ limit of a 2D surface growth model called the $q$-Whittaker process. The fluctuations are proven to exhibit Gaussian characteristics that ``come down from infinity'': After rescaling and re-centering,  convergence to the time-inverted stationary additive stochastic heat equation holds. The point of view in this paper is a probabilistic representation of the SDEs by independent sums. By this connection, the normal and Poisson approximations and the in-between slow decorrelation, all in particular integrated forms, explain the convergence of the re-centered covariance functions. With bounds and divergent constants from these approximations, the proof of the process-level convergence identifies additional divergent terms in the dynamics and considers cancellation arguments that treat the independent sums as discrete spin systems.  
\smallskip
 
\noindent \emph{Keywords:} Surface growth models; additive stochastic heat equation;
pure death processes; normal approximations; Poisson approximations.\medskip

\noindent \emph{Mathematics Subject Classification (2000):} 60J27, 60K35, 60F05
}

\setcounter{tocdepth}{2}
\renewcommand{\baselinestretch}{0.6}
\tableofcontents
\renewcommand{\baselinestretch}{1.0}\normalsize

\section{Introduction}\label{sec:intro}
The original model behind this paper is  the $q$-Whittaker process which goes back to Borodin and Corwin  \cite{BC:Mac}.
The process is a discrete interacting particle system modeling surface growth dynamics. In the limit of $q\to 1$, a recent work of Borodin, Corwin and Ferrari \cite{BCF} proves a convergence of the fluctuations by taking an iterated scaling limit. Our objective in this paper is the final stage of the limiting scheme. It obtains a pointwise limit of the (space-time) covariance functions of the following SDEs:
\begin{align}\label{SDE}
\d \xi_t(a)=t^{-1}(A_L\xi_t)(a) \d t+\d B_t(a),\quad a\in \mathcal T_L,\; 0<t<\infty,
\end{align}
where $L$ is an integer $\geq 3$ and 
$\mathcal T_L$ is the following upper triangular lattice:
\begin{align}\label{def:TNintro}
\mathcal T_L=\{a=(a_1,a_2)\in \Bbb N^2;1\leq a_1\leq a_2\leq L\}.
\end{align}
In \eqref{SDE}, $A_L$ can be identified as a generator matrix. The corresponding Markov chain is given by the invertible linear transformation $(m_1,m_2)\mapsto (m_1+1,m_1+m_2+1)$ of two independent linear pure death processes, such that the sum of the populations does not exceed $L-1$. The noise terms are given by a standard Brownian motion $\{B_t(a);a\in \mathcal T_L\}$. Due to the original discrete dynamics, we call these SDEs from \cite{BCF} the {\bf Whittaker SDEs}. 

Our goal in this paper is to prove a rescaled limit of the Whittaker SDEs at the process level. For the convergence of the covariance functions, we choose a method very different from the method in \cite{BCF} by turning to a probabilistic representation of solutions of the SDEs. Hence, the entire spectrum of limit theorems for sums of Bernoulli random variables enters and becomes the point of departure of the present proofs.

  \usetikzlibrary{decorations.markings}
\tikzset{->-/.style={decoration={
  markings,
  mark=at position #1 with {\arrow{>}}},postaction={decorate}}}
\begin{figure}[t]
\begin{center}
\begin{tikzpicture}[scale=3.5]
	\draw[step=.2cm,gray,very thin] (.2,0) grid (1.8,1.2);
        \foreach \a in {1,.8,.6,.4}
        \filldraw [blue] (1.2,0.2) circle (.8pt);
        \filldraw [red] (1.2,0.4) circle (.8pt);
        \filldraw [blue] (1.6,0.4) circle (.8pt);
        \filldraw [blue] (1.6,0.6) circle (.8pt);
        \filldraw [blue] (1.4,0.6) circle (.8pt);
        \filldraw [blue] (0.8,0.6) circle (.8pt);
        \filldraw [blue] (1.6,0.8) circle (.8pt);
        \filldraw [blue] (1.4,0.8) circle (.8pt);
        \filldraw [blue] (1.0,0.8) circle (.8pt);
        \filldraw [blue] (.6,0.8) circle (.8pt);
        \filldraw [blue] (.4,1) circle (.8pt);
        \filldraw [blue] (.8,1) circle (.8pt);
        \filldraw [blue] (1,1) circle (.8pt);
        \filldraw [blue] (1.4,1) circle (.8pt);
        \filldraw [blue] (1.6,1) circle (.8pt);
        \node at (1.2,.1) {\tiny $\lambda(1,1)$};
        \node at (1.65,.3) {\tiny $\lambda(1,2)$};
        \node at (1.2,.3) {\tiny $\lambda(2,2)$};
        \node at (1.65,.5) {\tiny $\lambda(1,3)$};
        \node at (1.4,.5) {\tiny$\lambda(2,3)$};
        \node at (0.8,.5) {\tiny $\lambda(3,3)$};
        \node at (1.65,.7) {\tiny $\lambda(1,4)$}; 
        \node at (1.4,.7) {\tiny$\lambda(2,4)$};
        \node at (1,.7) {\tiny$\lambda(3,4)$};
        \node at (.6,.7) {\tiny $\lambda(4,4)$};
        \node at (1.65,.9) {\tiny $\lambda(1,5)$};
        \node at (1.4,.9) {\tiny $\lambda(2,5)$};
        \node at (1,.9) {\tiny $\lambda(3,5)$};
        \node at (.8,.9) {\tiny $\lambda(4,5)$};
        \node at (.4,.9) {\tiny $\lambda(5,5)$};
        \draw[->] (0.2,0) -- (2,0);
                \draw[->] (0.2,0) -- (0.2,1.3);
        \node at (2.1,0) {$\Bbb Z$};
                \node at (.1,1.3) {$a_2$};
\end{tikzpicture}
\end{center}
\caption{The figure shows a configuration $\lambda=\{\lambda(a_1,a_2);(a_1,a_2)\in \mathcal T_5\}$ under the $q$-Whittaker process. Here, $\mathcal T_L$ is defined in \eqref{def:TNintro}. For example, the position of the $2$nd particle at level $3$ is $\lambda(2,3)\in \Bbb Z$. Any blue particle can jump to the right with pushing whenever chosen. The red particle is blocked by the particle at level $1$ to maintain the inequalities $\lambda(2,2)\leq \lambda(1,1)\leq \lambda(1,2)$ in the interlacing relation. } 
\label{fig:0}
\vspace{-.5cm}
\end{figure}
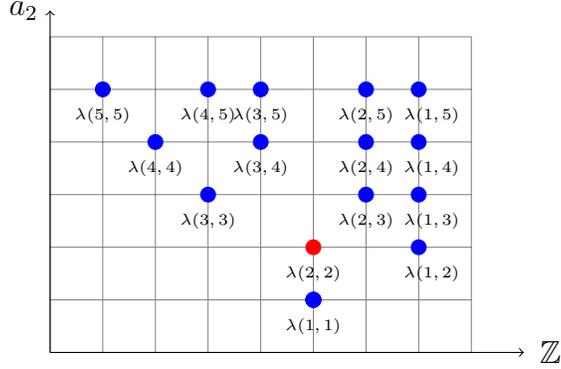

Let us begin by explaining connections between the Whittaker SDEs and surface growth dynamics in terms of the original discrete model. First, given an integer $L\geq 3$, the $q$-Whittaker process  is a stochastic dynamic model of particles living in $L$ rising levels. Each level is a copy of the one-dimensional integer lattice. The number of particles at each level is the same as the level number. The particles at each level are ordered from the right to the left. Accordingly, a configuration of the whole system takes values in $\Bbb Z^{\mathcal T_L}$ with respect to  the finite triangular lattice $\mathcal T_L$ defined above. The particle system evolves according to some special rates depending on $q$ and local configurations. The updates also respect an interlacing relation induced by $\mathcal T_L$ on the overall particle positions. See Figure~\ref{fig:0} for an example of the eligible configurations. More specifically, when a particle, say, at $a\in \mathcal T_L$ is chosen for a transition according to the rates, it attempts to jump to the right lattice point. This jump is successful if the chosen particle is not ``blocked'', in a sense that the jump does not break the assumed interlacing relation. In this case, the chosen particle also ``pushes'' to the right, by one lattice point, the entirety of the vertical string of particles above it in terms of positions and indices. In other words, these particles being pushed originally occupy the same integer site at the respective levels as the chosen particle and are indexed by $a+(0,1),\cdots,a+(0,\ell-1)$, such that $\ell\geq 1$ is maximal. See Section~\ref{sec:Wparticle} for the details of these mechanisms.

The $q$-Whittaker process has an important meaning of being a model of surface growth. Configurations of the interlacing particles can be mapped to discrete surfaces (via lozenge tilings). So transitions of particle configurations correspond to transitions of heights in the discrete surfaces \cite[Section~1.1]{BF}. From a broader perspective,  the pushing mechanism and  the blocking mechanism define the $q$-Whittaker process as a generalization of the totally asymmetric simple exclusion process (TASEP) that plays an important role in the 1D Kardar--Parisi--Zhang universality class \cite{Quastel}. Moreover, since the pushing mechanism induces a special tilt direction in the discrete surfaces, it is believed that the $q$-Whittaker process falls in the  \emph{anisotropic} Kardar--Parisi--Zhang universality class. In this direction, Wolf's conjecture \cite{Wolf}  from the physics literature expects that the large time fluctuations in this class are universally of the order $\sqrt{\ln t}$. Moreover, the probabilistic counterpart of this magnitude is the 2D Gaussian free field \cite[Section~1.4]{BF}. See \cite{BS,Toninelli_17} for more backgrounds.

For the $q$-Whittaker process, the proof of this physical relation  is first obtained in an insightful work  by Borodin and Ferrari \cite{BF}  for the case $q=0$. For the $q\to 1$ dynamics, an intermediate step of the proof in \cite{BCF} leads to the Whittaker SDEs. See \eqref{def:xit} for a summary of the limiting schemes. In particular, though introduced with a probabilistic interpretation by pure death processes above, in \cite{BCF}, the matrices $A_L$ in the drift coefficients of the SDEs appear with the explicit algebraic forms, and the methods are developed accordingly.

In the direction of universality, Wolf's conjecture is also proven in \cite{BCT} for the $q$-Whittaker \emph{driven} particle system \cite{CT}. This model is defined on periodic lattices in two dimensions. It gives another discrete surface growth dynamics that include a special tilt direction. For the $q\to 1$ fluctuations, an intermediate step of the proof in \cite{BCT} derives  some Ornstein--Uhlenbeck finite-dimensional  SDEs (called the Whittaker \emph{driven} SDEs in \cite{C:RW}) similar to the Whittaker SDEs.  The constant matrices in the drift coefficients of the Whittaker driven SDEs continue to incorporate some characteristics of the discrete dynamics. But they are structurally different from the matrices $A_L$. Not all of the off-diagonal entries are nonnegative. In this case, discrete Fourier transforms are the central tool for proving the convergence. By contrast, the method to be discussed below for the Whittaker SDEs considers different properties, mainly the probabilistic interpretation of $A_L$ and some limiting stationarity. Consequently, except for the use of a divergent characteristic direction for space and time to be specified below, the present proofs show very mild technical overlaps with those in \cite{C:RW}. On the other hand, another investigation toward the universality in Wolf's conjecture by unifying \cite{C:RW} and the present paper  may be possible, given the entrance of the central limit theorem to be explained below. It shall develop along a comparison of the proofs with the classical counterparts that use Fourier transforms in proving the central limit theorem. We do not pursue the details here, though.

The main result of this paper proves a rescaled limit of the Gaussian stochastic integral parts  $\{\zeta_t(a);a\in \mathcal T_L\}$ in the solutions of the Whittaker SDEs. It can be summarized informally as the theorem below. See Theorems~\ref{thm:covar} and \ref{thm:second} for the precise statements. For convenience, the proof works with an infinite-dimensional Gaussian process  from consistently extending these parts of all finite $L$ (see Proposition~\ref{prop:SDE} and the discussion below it). The limiting process is based on the 2D additive stochastic heat equation \cite{Walsh}:
\begin{align}\label{SHE:intro}
\frac{\partial X}{\partial t}(x,t)=\frac{\Delta }{2} X(x,t)+\dot{W}(x,t),
\end{align}
where $\dot{W}$ is space-time white noise. Note that various rescaled limits of the covariance functions of the Whittaker SDEs are already obtained in \cite{BCF}. But these results are with forms and methods all different from those introduced here. See Remark~\ref{rmk:skellam} (2$^\circ$). Connections to the stationary 2D additive stochastic heat equation are also pointed out in \cite{BCF}. 

\vspace{-.15cm}

\begin{mthm}[Informal version]
Let $\{\zeta_t(a);a\in \mathcal T_\infty\}$ denote the Gaussian stochastic integral part in the solution of the infinite-level Whittaker SDE, and let $X(x,t)$ denote the stationary solution of the 2D additive stochastic heat equation \eqref{SHE:intro}. Then as $N\to\infty$,
\begin{align}\label{rescale:intro}
\left(x\mapsto \zeta_{Nt}\Big(\left\lfloor Nt+Nt\cdot \frac{x_1}{N^{1/2}}\right\rfloor ,\left\lfloor Nt+Nt\cdot \frac{x_2}{N^{1/2}}\right\rfloor \Big)\right)_{t>0}\lra\big(x\mapsto X(x,t^{-1})\big)_{t>0}
\end{align}
holds in the following two modes of convergence: {\rm (1)} pointwise convergence of space-time covariance functions modulo constants, and {\rm (2)} convergence in distribution as c\`adl\`ag processes that take values in the space of tempered distributions on $\R^2$ modulo constants. 
\end{mthm}

\vspace{-.15cm}

Let us explain the theorem in more detail. First, the rescaling in \eqref{rescale:intro} is along the characteristic direction $(x, t)\mapsto  (t + tx, t)$ followed by the Edwards--Wilkinson scaling \cite{EW}: $(x,t,\zeta)\mapsto (N^{-1/2}x,Nt,N^{0}\zeta)$. For the limiting process, stationarity holds in the form that  the initial condition is given by the massless 2D log-correlated Gaussian free field (Section~\ref{sec:id}). Due to the time inversion ($t\mapsto t^{-1}$), the informal process-level picture thus shows that the limiting fluctuations ``come down from infinity'' along the clock of the additive stochastic heat equation. In particular, this result should be an answer to the inquiry in \cite[Section~1.2]{BCF}  whether,  for time varying over the entire half-line, the fluctuations of the $q$-Whittaker process can be related to the stationary additive stochastic heat equation. There in \cite{BCF}, the relation is established over the time interval $ [0,1]$.

For the proof, we first consider a rescaled limit of the covariance functions. We follow an initial step in \cite[Lemma~5.7]{BCF} but subsequently turn to a method that entirely circumvents the use of complex contour integrals and special functions in the original method of \cite{BCF}. That step from \cite{BCF} identifies $A_L$ as the generator matrix of a linear transformation of two i.i.d. Markov chains with linearly decreasing rates. The method of our choosing starts with the interpretation that the latter two Markov chains are linear pure death processes and model the population size of individuals with i.i.d. exponential lifetimes. Moreover, the chains can be represented as sums of independent Bernoulli variables that keep track of the numbers of survivors. The central limit theorem thus leads to the Gaussian characteristic of the limit. It induces a natural choice of the rescaling scheme, and its local version gives approximations by Gaussian densities. The additive stochastic heat equation arises since its solution  is structurally similar to the solutions of the Whittaker SDEs.

Given the crucial connection to independent sums, we are still faced with the basic question of whether the normal approximations are enough. The issue arises from the following representation of the covariance functions of the rescaled processes in \eqref{rescale:intro}: 
\begin{align}\label{covar:intro}
\int_0^{Ns}\prod_{j=1}^2\mathbf P\left(S_{\big\lfloor Ns+Ns\cdot \frac{x_j}{N^{1/2}}\big\rfloor}\left(\frac{r}{Ns}\right)=S'_{\big\lfloor Nt+Nt\cdot \frac{y_j}{N^{1/2}}\big\rfloor}\left(\frac{ r}{Nt}\right)\right)\d r,
\end{align}
where $x,y\in \R^2$, $0\leq s\leq t<\infty$, and $S_m(p)$ and $S'_{m'}(p')$ are independent binomial random variables such that $S_m(p)$ has parameters $(m,p)$, and $S'_{m'}(p')$ is similarly defined (Proposition~\ref{prop:covar}). In \eqref{covar:intro}, the \emph{parameters} of the Bernoulli random variables in the order-$N$ sums are \emph{integrated out} up to order $1$. With this manner, the convergence of the probabilities has to be proven for essentially all the possible values of the parameters simultaneously. Therefore, the above sketch of proof actually leaves out the Poisson approximations, and more importantly, the ``slow decorrelation'' of sums of Bernoulli variables from the Poisson limit to the Gaussian limit. The slow decorrelation raises most of the technicality in this stage. The key is to carry it out by calculating the asymptotics of an integral of probabilities as in \eqref{covar:intro}, the associated interval being changed to $1\ll r\ll N$ where an integrated Poisson approximation and an integrated local central limit theorem are \emph{joined}. The proof also needs to quantify these $\ll$-bounds of $r$. In particular, one ingredient to get the integrated Poisson approximation is a bound from the Stein--Chen method \cite{BE:Poisson,BH}.  Along this way, some logarithmically divergent constants, which also appear in \cite{BCF}, and the stationarity in the limiting equation arise from integrating the decorrelating probabilities. See Section~\ref{sec:conv}. Let us stress that we do not pursue the optimality of error bounds for these integrated approximations. Our interest is to investigate possible connections between the normal and Poisson approximations and the universality in Wolf's conjecture. This direction is in the same spirit as the comparison with the proof in  \cite{C:RW}  for the Whittaker drive SDEs mentioned above.

The process-level convergence of the rescaled processes considers the weak formulation by integrating out space. In this form, the main result proves the uniform H\"older continuity in time on compacts of the covariance functions. The method begins when proving the pointwise convergence of the covariance functions already discussed above, as we quantify appropriate bounds throughout. For the next step (Section~\ref{sec:tight}), the divergent constants from integrating decorrelating probabilities are  removed by re-centering. The main work then bounds the H\"older coefficients by removing some \emph{additional} divergent terms (Lemma~\ref{lem:JN2,3}). The issue here arises since it is not clear how these divergent terms can cancel each other without further transformations. The integral structures of these terms appear to be special and different from those in the decorrelation.  To find ways for the cancellations, the binomial integration by parts, and some recursive identities for the independent sums enter as the central tools. Essentially, by involving these tools, we treat the independent sums of Bernoulli random variables in the covariance functions as discrete spin systems. In summary, this part of the present paper shows a different application of the probabilistic presentation of the Whittaker SDEs. Moreover, as in \cite{C:RW}, it deals with dynamical singularities in the surface growth, along a divergent characteristic direction and not present in proving the convergence of covariance functions.

Finally, we remark that asymptotic expansions for the Poisson approximations \cite{Barbour} and the normal approximations \cite{Petrov} apply to the independent sums in this paper. Given this background, it may be of interest to investigate higher-order asymptotics of the fluctuations of the Whittaker SDEs at the process level. 
\medskip

\noindent {\bf Organization.} Section~\ref{sec:SDEs} discusses the Whittaker SDEs for more detailed connections to the $q$-Whittaker process, the solutions, and the key probabilistic representation giving \eqref{covar:intro}. Section~\ref{sec:conv} proves the convergence of the covariance function. In particular, the main technical conditions in the proofs of Sections~\ref{sec:conv} and~\ref{sec:tight} are imposed in Assumption~\ref{ass:sNtN} and Definition~\ref{def:cond}.  Some notations for binomial random variables used throughout this paper are defined in \eqref{def:Bkernel}. In Section~\ref{sec:ASHE}, we relate the limiting process to the additive stochastic heat equation. The proof of the tightness of the SDEs is in Section~\ref{sec:tight}. Finally, Section~\ref{sec:stat} collects some basic properties of the 2D stationary additive stochastic heat equation.  \smallskip

\noindent {\bf Convention for constants.} $C(T)\in (0,\infty)$ is a constant depending only on $T$ and can change from inequality to inequality. Other constants are defined analogously. We write $A\less B$ or $A\more B$ if $A\leq CB$ for a universal  constant $C\in (0,\infty)$. $A\asymp B$ means both $A\less B$ and $A\more B$.\smallskip

\noindent {\bf Acknowledgments.} The author would like to thank three referees for suggestions concerning  the presentation and comparison with related results and thank Andrew D. Barbour for answering questions on Poisson and normal approximations. Support from the Simons Foundation before the author's present position and from the Natural Science and Engineering Research Council of Canada is gratefully acknowledged.

\section{The Whittaker SDEs}\label{sec:SDEs}
In this section, we recall more details of the Whittaker SDEs derived in \cite[Proposition~5.5]{BCF}. Then we show that these SDEs can be solved explicitly by the driving Brownian motion and sums of i.i.d. Bernoulli random variables. Due to this connection, we turn to limit theorems of independent sums in the next section.

\subsection{Connections to the $q$-Whittaker process}\label{sec:Wparticle}
For fixed  $q\in [0,1)$ and integer $L\geq 3$, the {\bf $\bs q$-Whittaker process} $(\Lambda^q_t)_{t\geq 0}$ considers $L(L+1)/2$ many particles indexed by $\mathcal T_{L}$, where $\mathcal T_L$ is defined by \eqref{def:TNintro}. In any state $\lambda\in \Bbb Z^{\mathcal T_{L}}$ of the system, the following interlacing relation of particles holds: 
\begin{align}\label{def:interlacing}
\lambda(a_1+1,a_2)\leq \lambda(a_1,a_2-1)\leq \lambda(a_1,a_2),\;\forall\;a=(a_1,a_2)\in \mathcal T_{L},\;a_2\geq 2,\;a_2\geq a_1+1.
\end{align}

For $q>0$, the particle labelled by $a\in \mathcal T_L$ attempts to jump to the right ($\lambda(a)\mapsto \lambda(a)+1$) and, if successful, pushes other particles in the way described in Section~\ref{sec:intro}, with rate
\begin{align}\label{def:rate}
c_q(a,\lambda)=\frac{\left(1-q^{\lambda(a_1-1,a_2-1)-\lambda(a_1,a_2)}\right)\left(1-q^{\lambda(a_1,a_2)-\lambda(a_1+1,a_2)+1}\right)}{1-q^{\lambda(a_1,a_2)-\lambda(a_1,a_2-1)+1}}.
\end{align}
Here in \eqref{def:rate}, $1-q^{\lambda(a_1-1,a_2-1)-\lambda(a_1,a_2)}$ is understood to be $1$ if $(a_1-1,a_2-1)\notin \mathcal T_L$; a similar convention applies to  $1-q^{\lambda(a_1,a_2)-\lambda(a_1+1,a_2)+1}$ and $1-q^{\lambda(a_1,a_2)-\lambda(a_1,a_2-1)+1}$. That is, the following reduction applies when $a$ is a point on the left or right boundary of $\mathcal T_L$:
\begin{align}
c_q(a,\lambda)=\left\{
\begin{array}{ll}
1,&a_1=a_2=1;\\
\vspace{-.35cm}\\
1-q^{\lambda(a_2-1,a_2-1)-\lambda(a_2,a_2)},& a_1=a_2>1;\\
\vspace{-.35cm}\\
\displaystyle \frac{1-q^{\lambda(1,a_2)-\lambda(2,a_2)+1}}{1-q^{\lambda(1,a_2)-\lambda(1,a_2-1)+1}},&a_1=1,\;a_2>1.
\end{array}
\right.
\end{align}
One key feature of $c_q(a,\lambda)$ is that it is zero when $\lambda(a_1-1,a_2-1)=\lambda(a_1,a_2)$. The particle labelled by $a$ is thus ``blocked'' by the particle labelled by $(a_1-1,a_2-1)$, as illustrated in Figure~\ref{fig:0}. This condition ensures that the interlacing relation in \eqref{def:interlacing} is maintained. For $q=0$, the rates are extended by continuity so that $c_0(a,\lambda)=1$ whenever the particle labelled by $a$ under $\lambda$ is not blocked, and $c_0(a,\lambda)=0$ otherwise. (This case is not used in the present paper.) See \cite[Section~1.1]{BF} (for $q=0$), \cite[Definition~3.28]{BC:Mac}, and \cite[Section~2.5.2]{BCF}.

Given an integer $L\geq 3$, the following iterated limit in distribution
\begin{align}\label{def:xit}
\xi_t\;\defeq\,
\lim_{\vep_2\to 0+}\lim_{\vep_1\to 0+}(\vep_2\vep_1)^{1/2}\left(\Lambda^{\e^{-\vep_1}}_{(\vep_2\vep_1)^{-1}t}-\vep_1^{-1}\overline{\Lambda}_{(\vep_2\vep_1)^{-1}t}\right)
\end{align}
is proven under the initial conditions that $\Lambda_0^q(a)=0$ for all $a\in \mathcal T_L$. Here, the process $\overline{\Lambda}_t\;\defeq\,\lim_{\vep_0\to 0+}\vep_0\Lambda^{\e^{-\vep_0}}_{\vep_0^{-1}t}$ for re-centering is deterministic and so gives a law of large numbers for the particle system. Second, the limit in \eqref{def:xit} for $\vep_1\to 0+$ considers the $q$-Whittaker process with $q=\e^{-\vep_1}\to 1$. Finally, with  $\mathcal T_\infty\defeq\bigcup_{L\geq 3}\mathcal T_L$, the limiting fluctuation $(\xi_t)$ in \eqref{def:xit} obeys the SDE defined by \eqref{SDE} such that the matrix $A_L$ in the drift vector is
\begin{align}
\begin{split}
&A_L=A\rest (\mathcal T_L\times \mathcal T_L);\\
& \mbox{$\forall\;a,b\in \mathcal T_\infty$,}\; 
A(a,b)=\left\{
\begin{array}{ll}
a_1-1,&b=a-(1,1);\\
a_2-a_1, &b=a-(0,1);\\
0, &\mbox{for other }b\neq a;\\
-\sum_{b'\neq a}A(a,b'),&b=a.
\end{array}
\right.\label{def:AN}
\end{split}
\end{align}
See \cite[Sections~3 and 4]{BCF} for  the limits in \eqref{def:xit}. 

\usetikzlibrary{decorations.markings}
\tikzset{->-/.style={decoration={
  markings,
  mark=at position #1 with {\arrow{>}}},postaction={decorate}}}
\begin{figure}[t]
\begin{center}
\begin{tikzpicture}[scale=4]
	\draw[step=.2cm,gray,very thin] (0,0) grid (1,1);
        \draw (0,0) -- (1,0);
        \draw (0,0) -- (0,1);
        \foreach \a in {1,.8,.6,.4}
        \draw [ultra thick,->-=.5,color=gray](.2,\a)--(.2,\a-.2);
        \foreach \b in {1,.8,.6}
        \draw [ultra thick,->-=.5,color=gray](.4,\b)--(.4,\b-.2);
        \foreach \c in {1,.8}
        \draw [ultra thick,->-=.5,color=gray](.6,\c)--(.6,\c-.2);
        \foreach \d in {1}
        \draw [ultra thick,->-=.5,color=gray](.8,\d)--(.8,\d-.2);
        \foreach \x in {.4,.6,.8,1}
        \draw [ultra thick,->-=.5](\x,\x)--(\x-.2,\x-.2);
        \foreach \y in {.6,.8,1}
        \draw [ultra thick,->-=.5](\y-.2,\y)--(\y-.4,\y-.2);
        \foreach \z in {.8,1}
        \draw [ultra thick,->-=.5](\z-.4,\z)--(\z-.6,\z-.2);
        \foreach \zz in {1}
        \draw [ultra thick,->-=.5](\zz-.6,\zz)--(\zz-.8 ,\zz-.2);
        \foreach \x/\xtext in {0/0,.2/1,.4/2,.6/3,.8/4,1/5}
        \draw (\x,1pt) -- (\x,-1pt) node[anchor=north] {$\xtext$};
        \foreach \y/\ytext in {,.2/1,.4/2,.6/3,.8/4,1/5}
        \draw (1pt,\y) -- (-1pt,\y) node[anchor=east] {$\ytext$};
        \filldraw [red!80!black] (0.2,0.2) circle (.8pt);
        \foreach \u in {.2,.4,.6,.8,1}
        \filldraw [blue!50!white] (\u,1) circle (.8pt);
        \foreach \u in {.2,.4,.6,.8}
        \filldraw [blue!50!white] (\u,.8) circle (.8pt);
        \foreach \u in {.2,.4,.6}
        \filldraw [blue!50!white] (\u,.6) circle (.8pt);
        \foreach \u in {.2,.4}
        \filldraw [blue!50!white] (\u,.4) circle (.8pt);
\end{tikzpicture}
\end{center}
\caption{The possible trajectories of the Markov chain on $\mathcal T_5$ with generator $A_5$.} 
\label{fig:1}
\vspace{-.5cm}
\end{figure}
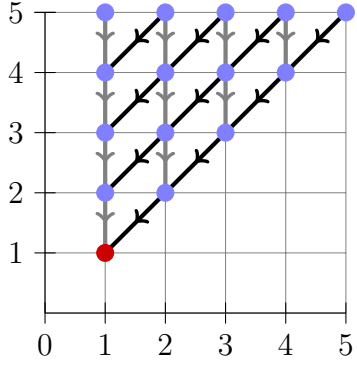

Note that $A$ has nonnegative off-diagonal entries by the definition of $\mathcal T_\infty$. Hence, $A$ is a generator matrix. The semigroup $(\e^{tA};t\geq 0)$ is Markovian. Moreover, given any integer $L\geq 3$, $A_{L}$ depends only on the lattice points in $\mathcal T_{L}$. It also holds that $A_{L}$ is a generator and $(\e^{tA_{L}};t\geq 0)$ is Markovian. See Figure~\ref{fig:1} for  an example. Moreover, $A_L$ shows the special jump rates \eqref{def:rate} in the limit and how the SDEs remember where the pushing and blocking come from in the discrete dynamics. For example, in \eqref{def:AN}, $b=a-(1,1)$ is the label of the particle that can block a jump of the particle labelled by $a$, and $b=a-(0,1)$ is the label of the particle that can push the particle labelled by $a$ or propagate a push.

\subsection{Explicit solutions}\label{sec:SDE}
The following proposition solves the Whittaker SDEs. 

\begin{prop}\label{prop:SDE}
Consider the Whittaker SDE defined in (\ref{SDE}) for an integer $L\geq 3$. \smallskip 

\noindent {\rm (1$^\circ$)} For any solution $(\xi_t;0<t<\infty)$, the following properties hold almost surely:
\begin{align}\label{def0:X0}
\xi_0\;\defeq\;\lim_{t\to 0+}\e^{(-\log t)A_{L}}\xi_t\quad\mbox{exists,}
\end{align}
and $\xi_t=\eta_t+\zeta_t$, where
\begin{align}\label{Xt:inftysoln}
\begin{split}
\eta_t(a)\;\defeq\;\e^{(\log t)A_L}\xi_0(a)\quad \&
\quad \zeta_t(a)\;\defeq\;\sum_{b\in \mathcal T_L}\int_{0}^t \e^{[\log(t/r)]A_L}(a,b)\d B_{r}(b),
\quad \forall\; a\in \mathcal T_L.
\end{split}
\end{align}

\noindent {\rm (2$^\circ$)} For any $\xi_0\in \R^{\mathcal T_{L}}$, $(\xi_t;0<t<\infty)$ defined by (\ref{Xt:inftysoln}) satisfies both  (\ref{SDE}) and (\ref{def0:X0}).
\end{prop}

\begin{proof}
The proofs of ($1^\circ$) and ($2^\circ$) both rely on the following almost-sure identity: 
\begin{align}\label{etA:SDE}
\e^{-(t+s)A_L}\xi_{\e^{t+s}}=\e^{-tA_L}\xi_{\e^t}+\int_{t}^{t+s} \e^{-r A_L }\d_r B_{\e^r},\quad \forall\;s\geq 0,\;t\in \R.
\end{align}
To see \eqref{etA:SDE}, first, note that for any solution $(\xi_t)$ of (\ref{SDE}) and for any fixed $t\in \R$,  by changing variables, $(\xi_{\e^{t+s}};s\geq 0)$ satisfies the following linear equation:
\begin{align*}
\xi_{\e^{t+s}}
=\xi_{\e^{t}}+\int_{0}^{ s}A_L\xi_{\e^{t+r}}\d r+B_{\e^{t+s}}-B_{\e^t}.
\end{align*}
Here, since $W_{r}=\int_0^{r}\e^{-v/2}\d_v (B_{\e^{t+v}}-B_{\e^t})$ is a standard Brownian motion by L\'evy's characterization of Brownian motion, we can write $B_{\e^{t+s}}-B_{\e^t}=\int_0^s\e^{r/2}\d_r W_r$. Hence, \eqref{etA:SDE} follows from a standard result of linear SDEs \cite[the last paragraph of p.354]{KS}.
\smallskip

\noindent {(1$^\circ$).}
First, we prove that the limit in (\ref{def0:X0}) exists almost surely. It suffices to show that the stochastic integral in (\ref{etA:SDE}) with $t+s=0$ converges almost surely as $t\to-\infty$. First, note that by the Riemann-sum approximations of stochastic integrals \cite[Section~3.2.B]{KS},
\begin{align}\label{improper}
\int_t^0\e^{-rA_L}\d_r B_{\e^r}=\int_0^{-t} \e^{rA_L}\d_r M_r,\quad \forall\; t\leq 0.
\end{align}
Here, by time reversal, $(M_s=B_1-B_{\e^{-s}};s\geq 0)$ is a continuous vector martingale with  $\d_s\langle M(b),M(b')\rangle_s=\delta_{b,b'}\e^{-s}\d s$. For all $a,a'\in \mathcal T_L$ and $s\geq 0$, the last stochastic integral satisfies
\begin{align}\label{improper:riemann}
\left\langle \left(\int_0^{\cdot} \e^{rA_L}\d_r M_r\right)(a),\left(\int_0^{\cdot} \e^{rA_L}\d_r M_r\right)(a')\right\rangle_{s}\!=\int_0^{s}[(\e^{rA_L})(\e^{rA_L})^\top](a,a')\e^{-r}\d r.
\end{align}
Since $\e^{rA_L}$ for $r\geq 0$ are stochastic matrices, the integral on the right-hand side of \eqref{improper:riemann} converges as $s\to\infty$. The existence of this limit and the martingale convergence theorem \cite[Problem~3.19 in Section~1.3]{KS} imply the almost-sure convergence of the improper vector stochastic integral $\int_{-\infty}^0 \e^{-rA_L}\d_r B_{\e^r}$ by \eqref{improper}, and hence, the limit in (\ref{def0:X0}) by \eqref{etA:SDE}.

Now, we pass $t\to -\infty$ in (\ref{etA:SDE}) and deduce that, for any $t_0\in \R$,
\begin{gather}
\begin{split}
\xi_{\e^{t_0}}=\e^{t_0A_L}\xi_0+\int_{-\infty}^{t_0}\e^{(t_0-r)A_L}\d_r B_{\e^r}=\e^{t_0A_L}\xi_0+\int_0^{\e^{t_0}}\e^{[\log (\e^{t_0}/r)]A_L}\d_r B_r.\label{etA:SDE1}
\end{split}
\end{gather}
The last equality is enough for (\ref{Xt:inftysoln}). We have proved ($1^\circ$).\smallskip

\noindent {(2$^\circ$).} Given $\xi_0\in \Bbb R^{\mathcal T_L}$, reversing the above arguments for (\ref{etA:SDE1}) and  (\ref{etA:SDE}) proves that $\xi_t=\eta_t+\zeta_t$ defined by (\ref{Xt:inftysoln}) satisfies both (\ref{SDE}) and (\ref{def0:X0}). The proof is complete.
\end{proof}

Since $A_{L}=A\rest (\mathcal T_{L}\times \mathcal T_{L})$ for all $L$, we can construct stochastic integral parts of solutions of the Whittaker SDEs on $\mathcal T_{L}$ for all $L\geq 3$ simultaneously from an infinite-dimensional standard Brownian motion $\{B_t(a);a\in \mathcal T_\infty\}$. This extension defines an infinite-dimensional Gaussian process $\{\zeta_t(a);a\in \mathcal T_\infty\}$ such that for $a\in\mathcal T_L$, $\zeta_t(a)$ equals the sum in \eqref{Xt:inftysoln}. In this case, $A_L$ and $\mathcal T_L$ in the sum can be replaced by $A$ and $\mathcal T_\infty$, respectively.

{\bf In the rest of this paper, we study the rescaled limits of $\{\zeta_t(a);a\in \mathcal T_\infty\}$ at the process level, not just the rescaled limits of their marginal distributions.}

\subsection{Representation by independent sums}
In this subsection, we discuss a probabilistic representation of the Markovian semigroup $(\e^{tA};t\geq 0)$ in terms of sums of i.i.d. Bernoulli random variables. This representation uses two sets of ingredients defined as follows. 

First, define a  sum function $\Sigma:\Bbb Z_+^2\to \mathcal T_\infty$  and a  difference function $\Delta:\mathcal T_\infty\to \Bbb Z_+^2$ by
\begin{align}
\Sigma:(m_1,m_2)\mapsto &\;(\Sigma_1(m_1,m_2),\Sigma_2 (m_1,m_2))\;\defeq\;( m_1 +1,m_1+ m_2+1),\label{def:sigma}\\
\label{def:delta}
\Delta:(a_1,a_2)\mapsto &\;(\Delta_1(a_1,a_2),\Delta_2 (a_1,a_2))\;\defeq\; ( a_1-1, a_2- a_1).
\end{align} 
The map $\Sigma$ is bijective with $\Delta$ being its inverse. (We abuse notation a bit since $\Delta$ has been used to denote the Laplacian in the context of SPDEs.) For the following result, see the proof of \cite[Lemma~5.7]{BCF}. 

\begin{lem}\label{lem:DA}
The Markov chain with generator $A$ can be represented as $\Sigma D$. Here, $D=(D^{(1)},D^{(2)})$,  and $D^{(1)}$ and $D^{(2)}$ are independent linear pure death chains on $\Bbb Z_+$ such that $k\to k-1$ with rate $k$, for any $k\in \Bbb Z_+$.
\end{lem}
\begin{proof}
A jump in $D^{(1)}$ changes $\Sigma D=(\Sigma_1D,\Sigma_2D)$ to $(\Sigma_1D-1,\Sigma_2D-1)$ with rate $D^{(1)}=\Sigma_1D-1$. Similarly, a jump in $D^{(2)}$ changes $\Sigma D=(\Sigma_1D,\Sigma_2D)$ to $(\Sigma_1D,\Sigma_2D-1)$ with rate $D^{(2)}=\Sigma_2D-\Sigma_1D$. These rates recover the entries of $A$.
\end{proof}

To introduce the second set of ingredients, we write $S_m(p)$ for a binomial random variable with parameters $m\in \Bbb Z_+$ and $p\in [0,1]$. Then take a sequence of i.i.d. exponential variables $\{\mathbf e_n\}$ under $\mathbf P$ and represent the binomial random variables explicitly as
\begin{align}\label{def:xiS}
S_m(\e^{-t})\;\defeq\; \sum_{n=1}^m \1_{(t,\infty)}(\mathbf e_n),\quad t\in [0,\infty],
\end{align}
with the convention that $\sum_{n=1}^0\equiv 0$. We let $\e^{-t}$ parametrize $S_m$ since  $\mathbf E[\1_{(t,\infty)}(\mathbf e_n)]=\e^{-t}$. The independent sum in (\ref{def:xiS}) is applied in the form that, as processes with c\`adl\`ag paths,
\begin{align}\label{eq:DS}
(D^{(j)}_t;t\geq 0)\stackrel{\rm (d)}{=} \big(S_{D^{(j)}_0}(\e^{-t});t\geq 0\big)
\end{align} 
for any deterministic initial condition $D^{(j)}_0\in \Bbb Z_+$. Recall that this probabilistic representation follows because memorylessness of exponential random variables supplies the Markov property and the property that $\mathbf e_n$'s are independent with $\mathbf E[\mathbf e_n]=1$ gives the linear death rates. See \cite[Section~6.2.1, pp.287--290]{PK}.

The probabilistic representation of $(\e^{tA})$ is now defined as follows: By Lemma~\ref{lem:DA}, the identity $\Delta\Sigma=\Id$ and \eqref{eq:DS}, we have
\begin{align}
 \forall\; a,b\in \mathcal T_\infty,\quad 
\e^{tA}(a,b)&=\mathbf P(\Sigma D_t=b|\Sigma D_0=a)\notag\\
&=\mathbf P( D_t=\Delta b| D_0=\Delta a)\notag\\
&=\mathbf P\left(S_{\Delta_1 a}(\e^{-t})=\Delta_1b\right)\mathbf P\left(S_{\Delta_2a}(\e^{-t})=\Delta_2b\right).\label{etA}
\end{align}
By \eqref{etA}, the pure death processes $D^j$ and the independent sums $S_m(\e^{-t})$ furnish a probabilistic representation of the solutions of the Whittaker SDEs. For this reason, we view these auxiliary random elements as being defined on a probability space, with probability $\mathbf P$ and expectation $\mathbf E$, separate from the probability space for the Whittaker SDEs. 

\begin{rmk}
Alternatively, the reader may choose to think of the Brownian motions $B(a)$'s collectively as  the ``random environment'' driving the Whittaker SDEs. In this case, all the random elements need to be defined on the same probability space.\hfill $\blacksquare$ 
\end{rmk}

We close this section with an immediate application of (\ref{etA}), which is the starting point of the next section. From now on, write $\zeta\Sigma (m_1,m_2)$ for the value of $\zeta:\mathcal T_\infty\to \Bbb R$ at $\Sigma (m_1,m_2)\in \Bbb Z^2$, $S'$ for an independent copy of the process $S$ defined by (\ref{def:xiS}), and ${\rm Cov}[U;V]=\E[UV]-\E[U]\E[V]$ for random variables $U$ and $V$.

\begin{prop}\label{prop:covar}
The mean-zero Gaussian process $\{\zeta_t(a);a\in \mathcal T_\infty\}$ defined below Proposition~\ref{prop:SDE} has a covariance  function satisfying the following probabilistic representation:
\begin{align}
\begin{split}
  {\rm Cov}\big[\zeta_s \Sigma (m_1,m_2);\zeta_t \Sigma ( m_1',m_2')\big]
=\int_0^s\prod_{j=1}^2\mathbf P\left(S_{m_j}\left(\frac{r}{s}\right)=S'_{m'_j}\left(\frac{ r}{t}\right)\right)\d r\label{PB}
\end{split}
\end{align}
for all $(m_1,m_2),(m'_1,m'_2)\in \Bbb Z_+^2$ and $0<s\leq t<\infty$.
\end{prop}
\begin{proof}
Recall that for any $L$ and $a\in \mathcal T_L$, $\zeta_t(a)=\sum_{b\in \mathcal T_\infty}\int_{0}^t \e^{[\log(t/r)]A}(a,b)\d B_{r}(b)$. See \eqref{Xt:inftysoln}. Hence, by It\^{o}'s isometry and the bijectivity of $\Sigma:\Bbb Z_+^2\to \mathcal T_\infty$, 
\begin{align}
&\quad   {\rm Cov}\big[\zeta_s \Sigma (m_1,m_2);\zeta_t \Sigma ( m_1',m_2')\big]\notag\\
&=\sum_{n_1,n_2\in \Bbb Z_+}\int_0^s \e^{[\log (s/r)]A}\big(\Sigma (m_1,m_2) ,  \Sigma (n_1,n_2)\big) \e^{[\log (t/r)]A}\big(\Sigma (m_1',m_2'),  \Sigma(n_1,n_2)\big)\d r\notag\\
&=\int_0^s\!\!\!\sum_{n_1,n_2\in \Bbb Z_+}\!\!\!\!\!\mathbf P\left(S_{ m_1}\left(\frac{r}{s}\right)=n_1\right)\mathbf P\left(S_{ m_2}\left(\frac{r}{s}\right)=n_2\right)\mathbf P\left(S_{ m_1'}\left(\frac{r}{t}\right)=n_1\right)\mathbf P\left(S_{ m_2'}\left(\frac{r}{t}\right)=n_2\right)\d r,\notag
\end{align}
where the last equality follows from (\ref{etA}) and the identity $\Sigma\Delta=\Id$. Summing over $n_1,n_2$ in the last equality proves \eqref{PB}.
\end{proof}

\section{Rescaled limit of the covariance function}\label{sec:conv}
In this section, we prove convergence of the covariance function in (\ref{PB}) and quantify the error bounds for the forthcoming applications to tightness of the fluctuation of the Whittaker SDEs.

First, the rescaling can be chosen from the central limit theorem if we consider the probabilities in (\ref{PB}). Write
\begin{align}\label{P:alter}
\mathbf P\left(S_{m_j}\left(\frac{r}{s}\right)=S'_{m'_j}\left(\frac{ r}{t}\right)\right)=\mathbf P\left(\overline{S}_{m_j}\left(\frac{r}{s}\right)-\overline{S}'_{m'_j}\left(\frac{ r}{t}\right)=-m_j\left(\frac{r}{s}\right)+m_j'\left(\frac{r}{t}\right)\right),
\end{align}
with the shorthand notation
\[
\overline{W}=W-\bE[W].
\]
Then a nontrivial limit of the random variable in \eqref{P:alter} follows if we set $m_j$ and $m_j'$ to be
\begin{align}\label{def:MjM'j}
M_j=M(x_j,s)\quad\&\quad M_j'=M(y_j,t),
\end{align} 
respectively,
where, for a fixed integer $N\geq 1$,
\begin{align}\label{def:Morigin}
M(u,r)=M_N(u,r)\,\defeq\,\left\lfloor Nr+Nr\cdot \frac{u}{N^{1/2}}\right\rfloor,\quad (u,r)\in \R\times\R_+.
\end{align}
Under this setup, the central limit theorem applies to the sequence
\[
\frac{1}{N^{1/2}}\left(S_{M_j}\left(\frac{r}{s}\right)-S'_{M'_j}\left(\frac{ r}{t}\right)\right),\quad N\geq 1,
\]
since, with 
\begin{align} 
\begin{split}
\mu_j(r;N)
&=\frac{M_j(\tfrac{r}{s})-M_j'(\tfrac{r}{t})}{N^{1/2}},\;\;
\sigma_j(r;N)^2=\frac{M_j}{N}\left(\frac{r}{s}\right)\left(1-\frac{r}{s}\right)+\frac{M_j'}{N}\left(\frac{r}{t}\right)\left(1-\frac{r}{t}\right),\label{def:mujsigmaj0}
\end{split}\\
 \mu_j(r)&=\mu(x_j-y_j,r)=(x_j-y_j)r, \;\;
\sigma_j(r)^2= r\left(2-\frac{r}{s}-\frac{ r}{t}\right),\label{def:mujsigmaj}
\end{align}
for $\sigma_j(r;N),\sigma_j(r)\geq 0$, we have
\begin{align*}
\mu_j(r;N)&=\mathbf E\left[\frac{1}{N^{1/2}}\left(S_{M_j}\left(\frac{r}{s}\right)-S'_{M'_j}\left(\frac{ r}{t}\right)\right)\right]\xrightarrow[N\to\infty]{} \mu_j(r),\\
\sigma_j(r;N)^2&={\rm Var}\left[\frac{1}{N^{1/2}}\left(S_{M_j}\left(\frac{r}{s}\right)-S'_{M'_j}\left(\frac{ r}{t}\right)\right)\right]\xrightarrow[N\to\infty]{}\sigma_j(r)^2.
\end{align*}

At the process level, we consider a rescaled version of $\zeta$ (Proposition~\ref{prop:SDE}) defined by
\begin{align}\label{def:zetaN1}
\begin{split}
&\zeta^N(x,s)\;\defeq\;\zeta_{Ns}\Sigma\left(\left\lf Ns+Ns\cdot \frac{x_1}{N^{1/2}}\right\rf,\left\lf Ns+Ns\cdot \frac{x_2}{N^{1/2}}\right\rf\right)
\end{split}
\end{align}
for all $x=(x_1,x_2)\in \R^2$ and $s\in \R_+$ such that $M_1,M_2\geq 0$ (recall \eqref{def:MjM'j}). {\bf We always assume this condition $M_1,M_2\geq 0$ on the space-time points when considering $\zeta^N$.} See also Remark~\ref{rmk:simple}.
Then by Proposition~\ref{prop:covar}, we have
\begin{align}
  {\rm Cov}\big[\zeta^N(x,s);\zeta^N(y,t)\big]&
=\int_0^{Ns}\prod_{j=1}^2\mathbf P\left(S_{M_j}\left(\frac{r}{Ns}\right)=S'_{M'_j}\left(\frac{ r}{Nt}\right)\right)\d r\label{obj1}\\
&=\int_0^{s}\prod_{j=1}^2N^{1/2}\mathbf P\left(S_{M_j}\left(\frac{r}{s}\right)=S'_{M'_j}\left(\frac{ r}{t}\right)\right)\d r\label{obj2}
\end{align}
for $x,y\in \R^2$ and $0\leq s\leq t<\infty$. Notice that the integral representation in \eqref{obj2} corresponds to the ``ideal case'' discussed above for \eqref{P:alter}. If $r,s,t$ are fixed such that $\sigma_j(r)\neq 0$, the above view for the probability in \eqref{P:alter}  applies to the integrand in the form of the local central limit theorem. But due to the integral nature of the covariance function, we cannot neglect the contribution of $r\approx 0$ as $N\to\infty$. This is where the local central limit theorem can break down. A similar issue arises if $r\approx s= t$ under \eqref{obj2}. Poisson approximations will apply over these two ranges of $r$. The integral representation in \eqref{obj1} is suitable for this purpose. Nevertheless, the central issue is the slow decorrelation.

\begin{rmk}[Edwards--Wilkinson scaling]\label{rmk:EW}
In terms of the above approximations, the covariance function in \eqref{obj1} has a natural generalization in other spatial dimensions $d$:
\[
N^{\frac{d-2}{2}}\int_0^{Ns}\prod_{j=1}^d\mathbf P\left(S_{M_j}\left(\frac{r}{Ns}\right)=S'_{M'_j}\left(\frac{ r}{Nt}\right)\right)\d r=\int_0^{s}\prod_{j=1}^dN^{1/2}\mathbf P\left(S_{M_j}\left(\frac{r}{s}\right)=S'_{M'_j}\left(\frac{ r}{t}\right)\right)\d r.
\] 
That is, introducing the factor $N^{\frac{d-2}{2}}$ enables the application of the local central limit theorem. Accordingly, one could consider generalizations of the Whittaker SDEs to other spatial dimensions, starting with multi-dimensional generalizations of the triangular lattices $\mathcal T_N$. We do not pursue these generalizations here. On the other hand, the foregoing display suggests that the rescaled process $\zeta^N$ defined in \eqref{def:zetaN1} can be seen as 
\[
\zeta^N(x,s)=N^{\frac{d-2}{4}}\zeta_{Ns}\Sigma\left(\left\lf Ns+Ns\cdot \frac{x_1}{N^{1/2}}\right\rf,\left\lf Ns+Ns\cdot \frac{x_2}{N^{1/2}}\right\rf,\cdots,\left\lf Ns+Ns\cdot \frac{x_d}{N^{1/2}}\right\rf\right),
\]
with $d=2$. From this aspect, $\zeta^N$ is subject to the Edwards--Wilkinson scaling exponents for interface growth models. 
\hfill $\blacksquare$
\end{rmk}

The following theorem summarizes the results of this section. Here and in what follows, we write $V(\lambda)$ and $V'(\lambda')$ for independent Poisson random variables with means $\lambda$ and $\lambda'$, respectively. Also, $(Q_t)$ stands for the probability semigroup of the two-dimensional standard Brownian motion.

\begin{thm}\label{thm:covar}
Let $\zeta^N$ be defined by \eqref{def:zetaN1}.
For all $0<s\leq t<\infty$ and $x,y\in \R^2$ such that either $s<t$ or $x\neq y$, it holds that
\begin{align}\label{covar:zeta}
\begin{split}
&\quad \lim_{N\to\infty} \Big(    {\rm Cov}\big[\zeta^N(x,s);\zeta^N(y,t)\big]
-\mathfrak C_N\Big)\\
&=\frac{1}{2\pi}\int_{\R^2}\int_{\R^2}Q_{s^{-1}}(y',x)\big(-\ln |y'-y''|\big)Q_{t^{-1}}(y'',y)\d y'\d y''\\
&\quad +\int_0^{t^{-1}} \int_{\R^2}Q_{s^{-1}-r}(z,x)Q_{t^{-1}-r}(z,y)\d z \d r,
\end{split}
\end{align}
where $\mathfrak C_N$ is given by 
\begin{align}
\begin{split}\label{def:CN}
\mathfrak C_N\,\defeq\,\mathfrak C_1+\frac{\ln N}{4\pi};\quad \mathfrak C_1\,\defeq\int_0^{\infty}\Big(\mathbf P\big(V(r)=V'(r)\big)^2+\frac{\e^{-\frac{1}{4r}}-2\1_{[1,\infty)}(r)}{4\pi r}\Big)\d r.
\end{split}
\end{align}
\end{thm}

Theorem~\ref{thm:covar}  combines the more detailed results, Theorems~\ref{thm:poisson} and~\ref{thm:normal}, to be proven in the rest of this section.

\begin{rmk}\label{rmk:skellam}
(1$^\circ$). For $\lambda,\lambda'\in (0,\infty)$, $V(\lambda)-V'(\lambda')$ is distributed as the {\bf Skellam distribution} \cite{Skellam}: with the modified Bessel function of the first kind denoted by $I_k$,
\[
\mathbf P\left(V(\lambda)-V'(\lambda')=k\right)=\e^{-(\lambda+\lambda')}\left(\frac{\lambda}{\lambda'}\right)^{k/2}I_k(2\sqrt{\lambda\lambda'}),\quad k\in \Bbb Z.
\]

\noindent {\rm (2$^\circ$).} We recall once again that various rescaled limits of the covariance functions are already obtained in \cite{BCF}. See \cite[Theorem~5.9, Proposition 5.28, and Proposition~5.29]{BCF} in particular. Whereas the first two results of this list are stated as complex contour integrals, \cite[Proposition~5.29]{BCF} appears to be similar to Theorem~\ref{thm:covar}. These two results can be compared as follows.
By a scaling of time and \eqref{PB}, that limit from \cite{BCF} can be restated as 
\begin{align}
\begin{split}
\lim_{N\to\infty}\left(\Cov[U_N;V_N]-\frac{\ln (N/d)}{\pi d\sqrt{a(1-a)}}\right)
=\int_{\R^2}\frac{Q_{(t-1)/t}(0,y)}{\pi d\sqrt{a(1-a)}} (-\ln |x-y|^2)\d y
\end{split}\label{BCF}
\end{align}
for fixed $d>0$, $a\in (0,1)$, $t>1$ and $x\in \R^2$, where we use the shorthand notation:
\begin{align*}
U_N&=\zeta_{Nt}\big(\lfloor(1-a)dNt\rfloor, \lfloor dNt\rfloor\big),\\
V_N&=\zeta_N\big(\lfloor 
(1-a)dN+\sqrt{(1-a)dN}\cdot x_1
\rfloor, \lfloor dN+\sqrt{(1-a)dN}\cdot x_1+\sqrt{adN}\cdot x_2 \rfloor 
\big).
\end{align*}
Compared to \eqref{def:zetaN1} (where the map $\Sigma$ is used), the lattice points defining $U_N$ do not include space and terms of the order $\mathcal O(N^{1/2})$, among several other differences. See also \cite[Remark 5.30, and Corollary~5.31]{BCF} for results related to  \cite[Proposition~5.29]{BCF}.
\hfill $\blacksquare$
\end{rmk}

Recall the notation $M(x_j,s)$ defined in \eqref{def:Morigin}.
In the rest of this paper, we mostly write
\begin{align}\label{def:Bkernel}
\begin{split}
\mathfrak b^{N,j}(x_j,y_j;r,u,v)&\;\defeq\; \mathbf P\left(S_{M(x_j,u)}\left(\frac{r}{Nu}\right)=S'_{M(y_j,v)}\left(\frac{r}{Nv}\right)\right), \quad \mathfrak b^N=\prod_{j=1}^2\mathfrak b^{N,j};\\
\mathfrak b_{N,j}(x_j,y_j;r,u,v)&\;\defeq\; N^{1/2}\mathbf P\left(S_{M(x_j,u)}\left(\frac{r}{u}\right)=S'_{M(y_j,v)}\left(\frac{r}{v}\right)\right), \quad \mathfrak b_N=\prod_{j=1}^2\mathfrak b_{N,j}.
\end{split}
\end{align}
(These probabilities involve both $S$ and $S'$, not just one binomial variable.)

 For the proofs of Theorems~\ref{thm:poisson} and~\ref{thm:normal}, we apply two schemes of integration which formalize the consideration below \eqref{obj2}: For $0<\ell_N<r_N<1$ and $0<\tau_N<1$, we subdivide $r\in [0,s]$ into the following three intervals: 
\begin{align}\label{def:range1}
r\in [0,s\ell_N], \; r\in [s\ell_N,sr_N],\;\mbox{ and }r\in [sr_N,s]\quad\mbox{ if $0\leq t-s\leq \tau_N$,}
\end{align}
or into the following two intervals:
\begin{align}\label{def:range2}
r\in [0,s\ell_N]\quad\mbox{and} \quad r\in [s\ell_N,s]\quad\mbox{ if $ t-s>\tau_N$.}
\end{align} 
We use the notation in \eqref{def:Bkernel}. Then under \eqref{def:range1}, we  work with the following decomposition: 
\begin{align}
\begin{split}\label{tauN1}
 {\rm Cov}\big[\zeta^N(x,s);\zeta^N(y,t)\big]
&=\int_0^{Ns\ell_N}\mathfrak b^N(x,y;r,s,t)\d r+\int_{s\ell_N}^{sr_N}\mathfrak b_N(x,y;r,s,t)\d r\\
&\quad +\int_{Nsr_N}^{Ns}\mathfrak b^N(x,y;r,s,t)\d r.
\end{split}
\end{align}
The decomposition corresponding to (\ref{def:range2}) is
\begin{align}
\begin{split}\label{tauN2}
  {\rm Cov}\big[\zeta^N(x,s);\zeta^N(y,t)\big]
&=\int_0^{Ns\ell_N}\mathfrak b^N(x,y;r,s,t)\d r
 +\int_{s\ell_N}^{s}\mathfrak b_N(x,y;r,s,t)\d r.
\end{split}
\end{align}

\begin{ass}\label{ass:sNtN}
Fix $\eta\in (0,1/2)$.
For all integers $N\geq 16$, set $\ell_N=1-r_N=\tau_N=N^{-(1/2+\eta)}$.\mbox{\quad}
\hfill $\blacksquare$ 
\end{ass}

Lastly, we introduce some conditions on $x_1,x_2,y_1,y_2,s,t,N$ for the forthcoming proofs.

\begin{defi}\label{def:cond}
Fix $0<T_0<1<T_1<\infty$ and let $\eta\in (0,1/2)$ be the constant fixed in Assumption~\ref{ass:sNtN}. The {\bf primary condition (over $\bs [\bs T_{\bs 0}\bs ,{\bs T}_{\bs 1}\bs ]$)} refers to the following condition:
\begin{align}\label{prim cond}
\left\{
\begin{array}{rl}
(x_1,x_2,y_1,y_2)&\hspace{-.3cm}\mbox{\rm :\;\;} x_1,x_2,y_1,y_2\in [-\tfrac{1}{2}N^{\eta},\tfrac{1}{2}N^{\eta}];\\
(s,t)&\hspace{-.3cm}\mbox{\rm :\;\;} T_0\leq s\leq t\leq T_1;  \\
N&\hspace{-.3cm}\mbox{\rm :\;\;} \Bbb N\ni N\geq 16,\;\lfloor \tfrac{1}{2}T_0N^{1/2-\eta}\rfloor \geq 1.
\end{array}
\right.
\end{align}
The {\bf secondary condition (over $\bs [\bs T_{\bs 0}\bs ,{\bs T}_{\bs 1}\bs ]$)} refers to the following condition:
\begin{align}
\begin{split}
\label{sec cond}
(x_1,x_2,y_1,y_2)\mbox{\rm :\;\;}
|x_1-y_1|\wedge |x_2-y_2|\geq \textstyle \frac{4}{T_0}N^{-1/2}.
\end{split}
\end{align}
\end{defi}

\begin{rmk}\label{rmk:simple}
Given $\eta\in (0,1/2)$, the primary condition has two simple implications: First, $\lfloor \tfrac{1}{2}T_0N\rfloor\leq M_j,M_j'\leq \lfloor \tfrac{3}{2}T_1N\rfloor  $, for $M_j,M_j'$ defined in \eqref{def:MjM'j}. These bounds follow from the choice of $x_j,y_j,s,t$.  Second, the lower bound of $M_j,M_j'$ and the choice of $N$ imply $M_j,M_j'\geq 1$. \hfill $\blacksquare$ \end{rmk} 
 
The secondary condition will be used only in the proofs of Proposition~\ref{prop:poisson2} {\rm (4$^\circ$)} and Theorem~\ref{thm:covar1}. (The proof of Proposition~\ref{prop:IN} uses a variation of this condition.) {\bf From now on, whenever either of the two conditions in Definition~\ref{def:cond} is in use, Assumption~\ref{ass:sNtN} is imposed automatically.}

\subsection{Integrated Poisson approximations}\label{sec:poisson}
In this subsection, we study the integrals 
\begin{align}
\int_0^{Ns\ell_N}\mathfrak b^N(x,y;r,s,t)\d r,\quad \int_{Ns r_N}^{Ns}\mathfrak b^N(x,y;r,s,t)\d r
 \label{Poisson:goal2222}
\end{align}
that appear in (\ref{tauN1}) and (\ref{tauN2}), where $\mathfrak b^N$ is defined in \eqref{def:Bkernel}. From now on, we begin to use the convention for constants specified at the end of Section~\ref{sec:intro}.

\begin{lem}\label{lem:poisson1}
Fix $0<T_0<1<T_1<\infty$ and assume the primary condition \eqref{prim cond}. Then for any $L,R\in (0,s)$, we have
\begin{align}
&\quad \int_0^{NL}\Bigg|\mathfrak b^N(x,y;r,s,t)
-\prod_{j=1}^2\mathbf P\left(V\left(\frac{M_jr}{Ns}\right)=V'\left(\frac{M_j' r}{Nt}\right)\right)\Bigg|\d r
\less NL^2\left(\frac{1}{s}+\frac{1}{t}\right)\label{Poisson1}
\end{align}
and
\begin{align}
&\quad \int_{NR}^{Ns}\Bigg|\mathfrak b^N(x,y;r,s,t)-\prod_{j=1}^2\mathbf P\Bigg(V\Big(\frac{M_j\big(Ns-r\big)}{Ns}\Big)=V'\Big(\frac{M_j'\big(Nt-r\big)}{Nt}\Big)+M_j-M_j'\Bigg)\Bigg|\d r\notag\\
\begin{split}\label{Poisson2}
&\less 
N(s-R)^2\left(\frac{1}{s}+\frac{1}{t}\right)+N(s-R)\left(\frac{t-s}{t}\right).
\end{split}
\end{align}
\end{lem}
\begin{proof}
We state some preliminary results first. Write $d_{\TV}$ for the total variance distance of probability measures defined on the same space. The central tool of this proof is the following bound for Poisson approximations from \cite[Theorem~1]{BH}: for independent Bernoulli random variables $\beta_n$ with $\bE[\beta_n]=p_n$, 
\begin{align}\label{PA:Stein}
d_{\rm TV}\Bigg(\bP\Bigg(\sum_{n=1}^m\beta_n\in \,\cdot\,\Bigg),\bP\Bigg( V\Bigg(\sum_{n=1}^m p_n\Bigg)\in \,\cdot\,\Bigg)\Bigg)\leq \left(\frac{1-\exp\{-\sum_{n=1}^m p_n\}}{\sum_{n=1}^m p_n} \right)\sum_{n=1}^m p_n^2.
\end{align}
See also \cite{BE:Poisson}. We only use the particular case that $p_n=p$ for all $n$, for which the bound is reduced to $(1-\e^{-mp})p$ and so can be bounded by $p$. Also, we recall the following standard result: for probability distributions $\mu_1,\mu_2,\nu_1,\nu_2$ on $\Z$,
\begin{align}\label{dTV:mu}
d_{\TV}(\mu_1\otimes \mu_2,\nu_1\otimes \nu_2)\leq d_{\TV}(\mu_1,\nu_1)+d_{\TV}(\mu_2,\nu_2)
\end{align}
(e.g. \cite[Lemma~3.6.5 on p.147]{Durrett}). See \cite[Proposition~2.3]{CK} for an improved bound.

We are ready to prove (\ref{Poisson1}). Recall that $d_{\rm TV}(\mu,\nu)=\tfrac{1}{2}\sup|\int h\d \mu-\int h\d \nu|$ for probability measures $\mu,\nu$ on $\Bbb Z$, where $h$ ranges over all functions such that $\|h\|_\infty\leq 1$ \cite[(7.2) on p.221]{CGS}. By (\ref{PA:Stein}) and (\ref{dTV:mu}), for all $m,m'\in \Bbb N$, $n\in \Bbb Z$, and $p,p'\in (0,1)$, it holds that
\begin{align}
\left|\mathbf P\big(S_{m}(p)=S'_{m'}(p')+n\big)-\mathbf P\big(V(mp\big)=V'(m'p')+n\big)\right|
&\leq \frac{p+p'}{2}.\label{SM:poisson}
\end{align}
The foregoing inequality and the discrete product rule
\begin{align}\label{XYAB}
XY-AB=(X-A)(Y-B)+(X-A)B+(Y-B)A
\end{align} 
imply that 
\begin{align*}
&\quad \int_0^{NL}\left|\prod_{j=1}^2\mathbf P\left(S_{M_j}\left(\frac{r}{Ns}\right)=S'_{M'_j}\left(\frac{ r}{Nt}\right)\right)-\prod_{j=1}^2\mathbf P\left(V\left(\frac{M_jr}{Ns}\right)=V'\left(\frac{M_j' r}{Nt}\right)\right)\right|\d r\\
&\less \int_0^{NL}
\frac{r}{N}\left(\frac{1}{s}+\frac{1}{t}\right)\d r
\end{align*}
since the $X,Y,A,B$ in this application of \eqref{XYAB} are all bounded by $1$. The required bound in (\ref{Poisson1}) follows.

The proof of (\ref{Poisson2}) is similar. If $X$ is binomial with parameters $(M,p)$, then $M-X$ is binomial with parameters $(M,(1-p))$. Hence, 
\begin{align}\label{bin:comp}
\bP\big(S_m(p)=S'_{m'}(p')+n\big)=\bP\big(S_m(1-p)=S_{m'}(1-p')+m-m'-n\big).
\end{align}
By \eqref{SM:poisson} and \eqref{XYAB}, the integral on the left-hand side of (\ref{Poisson2}) can be $\less$-bounded by
\begin{align*}
&\quad  \int_{NR}^{Ns}
\left(\frac{Ns-r}{Ns}+\frac{Nt-r}{Nt}\right)\d r
=\int_{0}^{N(s-R)}
\left(\frac{r}{Ns}+\frac{Nt-Ns+r}{Nt}\right)\d r.
\end{align*}
This is enough for the required bound in (\ref{Poisson2}). The proof is complete.
\end{proof}

{\bf In the sequel, the discrete product rule in (\ref{XYAB}) will be used repeatedly without being mentioned.} As an immediate result of Lemma~\ref{lem:poisson1}, we obtain the following integrated Poisson approximations.

\begin{prop}\label{prop:poisson1}
Fix $0<T_0<1<T_1<\infty$. Under the primary condition \eqref{prim cond},   
\begin{align*}
&\int_0^{Ns\ell_N}\Bigg|\mathfrak b^{N}(x,y;r,s,t)
 -\prod_{j=1}^2\mathbf P\left(V\left(\frac{M_jr}{Ns}\right)=V'\left(\frac{M_j' r}{Nt}\right)\right)\Bigg|\d r\less C(T_0,T_1)N\ell_N^2.
\end{align*}
If, in addition, we assume $0\leq t-s\leq \tau_N$, then it holds that 
\begin{align*}
&\int_{Nsr_N}^{Ns}\Bigg|\mathfrak b^{N}(x,y;r,s,t)
-\prod_{j=1}^2\mathbf P\Bigg(V\Big(\frac{M_j(Ns-r)}{Ns}\Big)=V'\Big(\frac{M_j'(Nt-r)}{Nt}\Big)+M_j-M_j'\Bigg)\Bigg|\d r\\
&\less C(T_0,T_1)N(1-r_N)^2.
\end{align*}
Since Assumption~\ref{ass:sNtN} is in force, these integrals converge to zero as $N\to\infty$.
\end{prop}

In the context of  Proposition~\ref{prop:poisson1} (which is under the primary condition, and hence, under Assumption~\ref{ass:sNtN}), its first inequality gives  an estimate of the first integral in \eqref{Poisson:goal2222}. To calculate the limit of this integral of Poisson probabilities, note that $Ns\ell_N\to \infty$ for $s>0$, and we need to pass the limit of the Poisson probabilities under the integral sign. For \eqref{Poisson:goal2222}, the additional assumption $0\leq t-s\leq\tau_N$ yields an estimate by the  integral of Poisson probabilities from the second inequality of Proposition~\ref{prop:poisson1}. This integral can be written as
\begin{align}
\int_{0}^{N(s-sr_N)}\prod_{j=1}^2\mathbf P\Bigg(V\Big(\frac{M_jr}{Ns}\Big)=V'\Big(\frac{M_j'(Nt-Ns)}{Nt}+\frac{M_j'r}{Nt}\Big)+M_j-M_j'\Bigg)\d r,\label{poisson:int2}
\end{align}
where $N(s-sr_N)\to\infty$ for $s>0$. In this case, $\liminf_N\inf_{(s,t)}|M_j-M_j'|/N^{1/2}>0$ for fixed $x_j\neq y_j$, where $(s,t)$ ranges over all the pairs satisfying the standing assumptions. Also, with $r$ in the range of integration, the parameters of $V$ and $V'$ in \eqref{poisson:int2} are $o(N^{1/2})$. Hence, by scaling, the probability indexed by $j$ in \eqref{poisson:int2}  is zero in the limit for each $j$.

We use the next three lemmas to pass limits for the first integral of Poisson probabilities from Proposition~\ref{prop:poisson1} and the one in \eqref{poisson:int2} in the manner mentioned above. The first and the last of these lemmas consider the property that due to the infinite divisibility, the Poisson distributions with large parameters are eligible for  normal approximations. For the local central limit theorem in the first lemma, see Remark~\ref{rmk:stein} for a discussion and also \cite{BRR:19}. From now on, write, for all $\sigma\in (0,\infty)$ and $x\in \R$,
\begin{align}\label{GaussFT}
\g(\sigma^2;x)\,\defeq\,\frac{1}{\sigma\sqrt{2\pi}}\exp\Big\{-\frac{x^2}{2\sigma^2}\Big\}=\frac{1}{2\pi }\int_\R \e^{-\i \theta x}\exp\Big\{-\frac{\sigma^2\theta^2}{2}\Big\} \d \theta.
\end{align}

\begin{lem}\label{lem:poisson2}
For all $\lambda,\lambda'\in (0,\infty)$,  it holds that  
\begin{align}
\begin{split}\label{lclt3:poisson}
&\quad \sup_{a\in \Bbb Z}\left|\mathbf P\big(V(\lambda )=V'(\lambda')+a\big)-\frac{1}{(\lambda+\lambda')^{1/2}}\mathfrak g\left(1; \frac{-\lambda+\lambda'+a}{(\lambda+\lambda')^{1/2}}\right)\right|
\less  \frac{1}{(\lambda+\lambda')}.
\end{split}
\end{align}
\end{lem}
\begin{proof}
Recall that by Fourier inversions,
\begin{align}
f(a)=\frac{1}{2\pi}\int_{-\pi}^\pi \e^{-\i \theta a}\sum_{b\in \Bbb Z} f(b)\e^{\i \theta b}\d \theta,\quad \forall\; a\in\Bbb Z.
\label{discreteFT}
\end{align}
Writing $\phi_{\lambda,\lambda'}(\theta)$ for $\E\exp \{\i \theta[\,\overline{V}(\lambda )-\overline{V}'(\lambda')]\}$, we obtain from \eqref{GaussFT} and \eqref{discreteFT} that
\begin{align}
  \begin{split}
\forall\;a\in \Bbb Z,\quad  & \quad (\lambda+\lambda')^{1/2}\mathbf P\big(V(\lambda )=V'(\lambda')+a\big)-\mathfrak g\left(1; \frac{-\lambda+\lambda'+a}{(\lambda+\lambda')^{1/2}}\right)\notag
  \end{split}\\
\begin{split}\notag
&= \frac{1}{2\pi  }\int_{|\theta|\leq \pi (\lambda+\lambda')^{1/2}}\e^{-\i \theta \frac{-\lambda+\lambda'+a}{(\lambda+\lambda')^{1/2}}}\left[ \phi_{\lambda,\lambda'}\left(\frac{\theta}{(\lambda+\lambda')^{1/2}}\right)-\exp\left\{-\frac{ \theta^2}{2}\right\}\right]\d\theta\\
&\quad +
\frac{1}{2\pi  }\int_{|\theta|\geq \pi (\lambda+\lambda')^{1/2}}\e^{-\i \theta \frac{-\lambda+\lambda'+a}{(\lambda+\lambda')^{1/2}}}\exp\left\{-\frac{\theta^2}{2}\right\}\d\theta\\
\end{split}\\
&={\rm I}_{\ref{integral:poisson}}+{\rm II}_{\ref{integral:poisson}}.\label{integral:poisson}
\end{align}

We show that the decomposition in (\ref{integral:poisson}) implies (\ref{lclt3:poisson}). To bound ${\rm I}_{\ref{integral:poisson}}$, consider 
\begin{align}
\phi_{\lambda,\lambda'}\left(\frac{\theta}{(\lambda+\lambda')^{1/2}}\right)\exp\left\{\frac{\theta^2}{8}\right\}
&=\exp\left\{ \lambda \left(\e^{\i \theta/(\lambda+\lambda')^{1/2}}-1-\frac{\i\theta}{(\lambda+\lambda')^{1/2}}\right)\right\}\notag\\
&\quad \times\exp\left\{\lambda' \left(\e^{-\i \theta/(\lambda+\lambda')^{1/2}}-1+\frac{\i\theta}{(\lambda+\lambda')^{1/2}}\right)+\frac{\theta^2}{8}\right\}\notag\\
\begin{split}
&=\exp\Bigg\{(\lambda+\lambda')\left[\cos\left(\frac{\theta}{(\lambda+\lambda')^{1/2}}\right)-1+\frac{\theta^2}{8(\lambda+\lambda')}\right]\Bigg\}\\
&\quad \times \exp\Bigg\{\i (\lambda-\lambda')\left[\sin\left(\frac{\theta}{(\lambda+\lambda')^{1/2}}\right)-\frac{\theta}{(\lambda+\lambda')^{1/2}}\right]\Bigg\}.\label{poisson:phi}
\end{split}
\end{align}
By the inequality 
\begin{align}\label{exp:MVT}
|\e^{z_1}-\e^{z_2}|\leq\max\{\e^{|z_1|},\e^{|z_2|}\}\cdot |z_1-z_2|,\quad \forall\;z_1,z_2\in \Bbb C,
\end{align}
and Taylor's theorem, (\ref{poisson:phi}) implies that, for all real $ |\theta|\leq \pi (\lambda+\lambda')^{1/2}$, 
\begin{align}\label{phir:bdd}
\left|\phi_{\lambda,\lambda'}\left(\frac{\theta}{(\lambda+\lambda')^{1/2}}\right)\exp\left\{\frac{\theta^2}{8}\right\}-\exp\left\{-\frac{3\theta^2}{8}\right\}\right|\less
(\lambda+\lambda')\cdot\frac{|\theta|^3}{(\lambda+\lambda')^{3/2}}.
\end{align}
(In more detail, we have used the inequality $\cos(y)-1+y^2/8\leq 0$, for all real $|y|\leq \pi$, when bounding the exponentials from the right-hand side of \eqref{exp:MVT}.) Hence,
\begin{align}\label{I1:poisson}
|{\rm I}_{\ref{integral:poisson}}|\less \int_{|\theta|\leq \pi (\lambda+\lambda')^{1/2}}\exp\left\{-\frac{\theta^2}{8}\right\}\frac{|\theta|^3}{(\lambda+\lambda')^{1/2}}\d \theta\less
\frac{1}{(\lambda+\lambda')^{1/2}}.
\end{align}

To bound $ {\rm II}_{\ref{integral:poisson}}$, we use the following simple inequality for any fixed $a\in (0,\infty)$: 
\begin{align}
x^a\int_x^\infty \exp\left\{-\frac{\theta^2}{2}\right\}\d \theta\leq C(a),\quad \forall\; x\geq 0.
\end{align}
With $a=1/2$ and $x=\pi (\lambda+\lambda')^{1/2}$, we get
\begin{align}
 |{\rm II}_{\ref{integral:poisson}}|
 &\less \frac{1}{(\lambda+\lambda')^{1/2}}.
 \label{I2:poisson}
\end{align}
The bound in (\ref{lclt3:poisson}) follows upon applying (\ref{I1:poisson}) and (\ref{I2:poisson}) to \eqref{integral:poisson}.
\end{proof}

\begin{lem}\label{lem:heat}
For all $0<a\leq b<\infty$ and $x,y\in \R$, it holds that 
\[
\left|\g(a;x)-\g (b;y)\right|\less \frac{|b-a|}{ a^{3/2}}+\frac{|x-y|}{ b}.
\]
\end{lem}
\begin{proof}
The required inequality follows if we apply the mean value theorem and the next two bounds to $[\g(a;x)-\g(b;x)]+[\g (b;x)-\g(b;y)]$:
\begin{align*}
\left|\frac{\d}{\d v}\frac{1}{\sqrt{2\pi v}}\exp\left\{-\frac{x^2}{2v}\right\}\right|\leq \frac{1}{\sqrt{2\pi}}\left(\frac{1}{2v^{3/2}}+\frac{1}{v^{3/2}}\e^{-1}\right),\;\; 
\left|\frac{\d}{\d u}\frac{1}{\sqrt{2\pi b}}\exp\left\{-\frac{u^2}{2b}\right\}\right|\leq \frac{\sqrt{2}}{\sqrt{2\pi b^2}},
\end{align*}
by $u\e^{-u}\leq \e^{-1}$ and $u\e^{-u^2}\leq 1$ (both are valid for all $u\geq 0$), respectively.
\end{proof}

For the next lemma, let $\lambda_{j}(r),\lambda_j'(r),\Lambda_j(r),\Lambda_j'(r)$ be increasing functions taking values in $(0,\infty)$ for all $r\in[1,\infty)$ and let $a_{j},A_j\in \Bbb Z$. Next, define an auxiliary function $g(r)=\prod_{j=1}^2 g_j(r)$ by $g_j(r)=\g(\lambda_j(r)+\lambda_j'(r);-\lambda_j(r)+\lambda_j'(r)+a_j )$ and $G(r)=\prod_{j=1}^2G_j(r)$ for similarly defined heat kernels $G_j(r)$ using $\Lambda_j,\Lambda'_j,A_j$ in place of $\lambda_j,\lambda_j',a_j$, respectively.

\begin{lem}\label{lem:poisson3}
Under the above setup, the following two inequalities hold for all $r\in [1,\infty)$:
\begin{align}
& \hspace{-2.3cm}\left|\prod_{j=1}^2\mathbf P\Big(V\big(\lambda_j(r) \big)=V'\left(\lambda'_{j}(r)\right)+a_{j}\Big)-g(r)  \right|\notag\\
&\less \frac{1}{\prod_{j=1}^2[\lambda_j(r)+\lambda'_j(r)]}+\sum_{\stackrel{\scriptstyle 1\leq i,j\leq 2}{i\neq j}} \frac{1}{[\lambda_i(r)+\lambda'_i(r)][\lambda_j(r)+\lambda'_j(r)]^{1/2}};\label{PIbdd:1}\\
\begin{split}
  \left|g(r)-G(r)\right| 
&\less \prod_{j=1}^2K_j\frac{|\lambda_{j}(r)-\Lambda_{j}(r)|+|\lambda'_{j}(r)-\Lambda'_{j}(r)|+|a_j-A_j|}{[\lambda_{j}(r)+\lambda'_{j}(r)]\wedge [\Lambda_{j}(r)+\Lambda'_{j}(r)]}\\
&\quad  +\sum_{\stackrel{\scriptstyle 1\leq i,j\leq 2}{i\neq j}}K_i\frac{|\lambda_{i}(r)-\Lambda_{i}(r)|+|\lambda'_{i}(r)-\Lambda'_{i}(r)|+|a_i-A_i|}{[\lambda_{i}(r)+\lambda'_{i}(r)]\wedge [\Lambda_{i}(r)+\Lambda'_{i}(r)]}\\
&\quad\quad   \times \frac{1}{[\lambda_{j}(r)+\lambda'_{j}(r)]^{1/2}\wedge [\Lambda_{j}(r)+\Lambda'_{j}(r)]^{1/2}},
\label{PIbdd:2}
\end{split}
\end{align}
where $K_j\defeq\{[\lambda_{j}(1)+\lambda'_{j}(1)]^{1/2}\wedge [\Lambda_{j}(1)+\Lambda'_{j}(1)]^{1/2}\}^{-1}+1$.
\end{lem}

\begin{proof}
By Lemma~\ref{lem:poisson2}, the left-hand side of \eqref{PIbdd:1}  is $\less$-bounded by
\begin{align}
&\quad \prod_{j=1}^2 \frac{1}{\lambda_{j}(r)+\lambda'_{j}(r)}
 +\sum_{\stackrel{\scriptstyle 1\leq i,j\leq 2}{i\neq j}}
\frac{1}{\lambda_{i}(r)+\lambda'_{i}(r)}g_{j}(r).\notag
\end{align}
Hence, (\ref{PIbdd:1}) follows from the foregoing bound and the definition of $g_j$.

The proof of \eqref{PIbdd:2} is similar. It is enough to note that by Lemma~\ref{lem:heat} and the assumed monotonicity of the functions in $r$, for all $r\in [1,\infty)$,
\begin{align*}
&\quad \left|\mathfrak g\left(\lambda_{j}(r)+\lambda'_{j}(r);-\lambda_{j}(r)+\lambda'_{j}(r)+a_{j}\right) -\mathfrak g\left(\Lambda_{j}(r)+\Lambda'_{j}(r);-\Lambda_{j}(r)+\Lambda'_{j}(r)+A_{j}\right) \right|\notag\\
&\less \frac{|\lambda_{j}(r)+\lambda'_{j}(r)-\Lambda_{j}(r)-\Lambda'_{j}(r)|}{[\lambda_{j}(r)+\lambda'_{j}(r)]^{3/2}\wedge [\Lambda_{j}(r)+\Lambda'_{j}(r)]^{3/2}}+\frac{|\lambda_{j}(r)-\lambda'_{j}(r)-a_{j}-\Lambda_{j}(r)+\Lambda'_{j}(r)+A_{j}|}{[\lambda_{j}(r)+\lambda'_{j}(r)]\wedge [\Lambda_{j}(r)+\Lambda'_{j}(r)]}\\
&\less K_j\frac{|\lambda_{j}(r)-\Lambda_{j}(r)|+|\lambda'_{j}(r)-\Lambda'_{j}(r)|+|a_j-A_j|}{[\lambda_{j}(r)+\lambda'_{j}(r)]\wedge [\Lambda_{j}(r)+\Lambda'_{j}(r)]},
\end{align*}
where $K_j$ is defined below \eqref{PIbdd:2}. The proof is complete. 
\end{proof}

The next proposition is the last step for the integrated Poisson approximations. Recall the discussion below Proposition~\ref{prop:poisson1}.

\begin{prop}\label{prop:poisson2}
Fix $0<T_0<1<T_1<\infty$ and let Assumption~\ref{ass:sNtN} be in force. \smallskip

\noindent \hypertarget{P2-1}{{\rm (1$^\circ$)}} For all $0< s\leq t<\infty$ and $x,y\in \R^2$, 
\begin{align*}
&\quad \lim_{N\to\infty}\int_0^{Ns\ell_N}\Bigg(\prod_{j=1}^2\mathbf P\left(V\left(\frac{M_jr}{Ns}\right)=V'\left(\frac{M_j' r}{Nt}\right)\right)-\frac{\1_{[1,\infty)}(r)}{4\pi r}\Bigg)\d r\\
&=\int_0^{\infty}\Bigg(\prod_{j=1}^2\mathbf P\big(V(r)=V'(r)\big)-\frac{\1_{[1,\infty)}(r)}{4\pi r}\Bigg)\d r.
\end{align*}

\noindent \hypertarget{P2-2}{{\rm (2$^\circ$)}} Under the primary condition over $[T_0,T_1]$, it holds that
\begin{align*}
&\quad \int_0^{Ns\ell_N}\left|\prod_{j=1}^2\mathbf P\left(V\left(\frac{M_jr}{Ns}\right)=V'\left(\frac{M_j' r}{Nt}\right)\right)-\frac{\1_{[1,\infty)}(r)}{4\pi r}\right|\d r\leq C(T_0,T_1)(1+|x|^2+|y|^2).
\end{align*} 

\noindent \hypertarget{P2-3}{{\rm (3$^\circ$)}} For all $s=t\in (0,\infty)$ and $x,y\in \R^2$ with $x\neq y$,
we have
\begin{align*}
&\quad \lim_{N\to\infty}\int_{Nsr_N}^{Ns}\prod_{j=1}^2\mathbf P\left(V\left(\frac{M_jr}{Ns}\right)=V'\left(\frac{M_j' r}{Nt}\right)\right)\d r=0.
\end{align*}

\noindent \hypertarget{P2-4}{{\rm (4$^\circ$)}}
Under  the primary and secondary conditions over $[T_0,T_1]$, and in addition, the assumption $0\leq t-s\leq \tau_N$, it holds that 
\begin{align*}
&\quad \int_{Nsr_N}^{Ns}\left|\prod_{j=1}^2\mathbf P\left(V\left(\frac{M_jr}{Ns}\right)=V'\left(\frac{M_j' r}{Nt}\right)\right)-\frac{\1_{[1,\infty)}(r)}{4\pi r}\right|\d r\\
&\leq C(T_0,T_1)\left(1+|x|^2+|y|^2+\big|\ln |x-y|\big|
\right).
\end{align*} 
\end{prop}

\begin{proof} We present the proofs of the first two statements and the last two separately. \smallskip

\noindent {(1$^\circ$) and (2$^\circ$).} We choose 
\[
 \lambda_{j}(r)=\frac{M_jr}{Ns},\quad \lambda_{j}'(r)=\frac{M'_jr}{Nt},\quad \Lambda_j(r)=\Lambda_j'(r)=r,\quad a_j=A_j=0
\]
in the setup of Lemma~\ref{lem:poisson3} and impose the primary condition.
Then by Remark~\ref{rmk:simple} and \eqref{PIbdd:1}, we get
\begin{align}\label{r3/2}
 \left|\prod_{j=1}^2\mathbf P\left(V\left(\frac{M_jr}{Ns}\right)=V'\left(\frac{M_j' r}{Nt}\right)\right)-g(r)\right|\leq \frac{C(T_0,T_1)}{r^{3/2}},\quad \forall\; r\in [1,\infty).
 \end{align}
The required limit in \hyperlink{P2-1}{{\rm (1$^\circ$)}} follows from the dominated convergence theorem.

Next, notice that $G(r)=\prod_{j=1}^2\g(2r;0)=1/(4\pi r)$. Hence,  by Remark~\ref{rmk:simple}, \eqref{PIbdd:2} and \eqref{r3/2}, we get the following bound for all $r\in [1,\infty)$: 
\begin{align}
&\quad \left|\prod_{j=1}^2\mathbf P\left(V\left(\frac{M_jr}{Ns}\right)=V'\left(\frac{M_j' r}{Nt}\right)\right)-\frac{1}{4\pi r}\right|\notag\\
&\leq C(T_0,T_1)\Bigg[\frac{1}{r^{3/2}}+\prod_{j=1}^2\Bigg( \left|\frac{M_j}{Ns}-1\right|+\left|\frac{M_j'}{Nt}-1\right|\Bigg) \notag\\
&\hspace{.5cm}+\frac{1}{r^{1/2}}\sum_{j=1}^2\Bigg( \left|\frac{M_j}{Ns}-1\right|+\left|\frac{M_j'}{Nt}-1\right|\Bigg)\Bigg]\notag\\
&\leq C(T_0,T_1)\Bigg(\frac{1}{r^{3/2}}+\frac{1}{N}(1+|x|+|y|)^2+\frac{1}{r^{1/2}N^{1/2}}(1+|x|+|y|)\Bigg),\label{PIbdd:app1}
\end{align}
where the last inequality follows since by the primary condition over $[T_0,T_1]$,
\begin{align}\label{MN-1}
|M_j/(Ns)-1|\leq C(T_0,T_1)(N^{-1}+N^{-1/2}|x_j|)
\end{align}
and a similar bound for  $|M_j'/(Nt)-1|$ holds. The bound in \hyperlink{P2-2}{{\rm (2$^\circ$)}} follows upon integrating both sides  of  \eqref{PIbdd:app1} over $r\in [1,Ns\ell_N]$. Here, we drop a multiplicative constant that tends to zero as $N\to\infty$ since this bound is considered for the tightness proof in Section~\ref{sec:tight}. \smallskip

\noindent {(3$^\circ$) and (4$^\circ$).}
Now, we work with the alternative expressions in \eqref{poisson:int2} for the integrals under consideration. Choose 
\begin{align}\label{3434setup}
\begin{split}
 \lambda_{j}(r)&=\frac{M_j}{Ns}r,\quad \lambda_{j}'(r)=\frac{M'_j(Nt-Ns)}{Nt}+\frac{M'_j}{Nt}r,\\
\Lambda_j(r)&=r,\quad \Lambda_j'(r)=\frac{(Nt+N^{1/2}ty_j)(Nt-Ns)}{Nt}+r,\quad a_j=A_j=M_j-M_j'
\end{split}
\end{align}
for the setup of Lemma~\ref{lem:poisson3} and impose both the primary and secondary conditions and the assumption that $0\leq t-s\leq \tau_N$. (We remove the floor function in $M_j'$ in choosing $\Lambda_j'$.) We proceed with the following steps. \smallskip

\noindent {\bf Step \hypertarget{P2-3-STEP1}{1}.}
First, by Remark~\ref{rmk:simple} and \eqref{PIbdd:1}, we have
\begin{align*}
&\left|\prod_{j=1}^2\mathbf P\Bigg(V\Big(\frac{M_jr}{Ns}\Big)=V'\Big(\frac{M_j'(Nt-Ns)}{Nt}+\frac{M_j'r}{Nt}\Big)+M_j-M_j'\Bigg)-g(r)\right|
\leq \frac{C(T_0,T_1)}{r^{3/2}}
\end{align*}
for all $r\geq 1$. Hence,
\[
\int_1^{N(s-sr_N)}\left|\prod_{j=1}^2\mathbf P\Bigg(V\Big(\frac{M_jr}{Ns}\Big)=V'\Big(\frac{M_j'(Nt-Ns)}{Nt}+\frac{M_j'r}{Nt}\Big)+M_j-M_j'\Bigg)-g(r)\right|\d r
\]
is bounded by $C(T_0,T_1)$, and in the case that $s=t$ and $x\neq y$, tends to zero by dominated convergence as $N\to\infty$ for the reason pointed out below Proposition~\ref{prop:poisson1}. \smallskip

\noindent {\bf Step 2.} We handle $\int_1^{N(s-sr_N)}|g(r)-G(r)|\d r$ by using \eqref{PIbdd:2}. To this end, note that under the setup in \eqref{3434setup}, 
\begin{align*}
&\quad |\lambda_{j}(r)-\Lambda_{j}(r)|+|\lambda'_{j}(r)-\Lambda'_{j}(r)|+|a_j-A_j|\\
&\leq r\left|\frac{M_j}{Ns}-1\right|+\frac{1}{Nt}\cdot N(t-s)+r\left|\frac{M_j'}{Nt}-1\right|\\
&\leq C(T_0,T_1)\left(\tau_N+rN^{-1}+rN^{-1/2}|x_j|+rN^{-1/2}|y_j|\right)\\
&\leq C(T_0,T_1)\left(rN^{-1/2}+rN^{-1}+rN^{-1/2}|x_j|+rN^{-1/2}|y_j|\right),\quad \forall\; r\geq 1,
\end{align*}
where the second inequality uses  \eqref{MN-1} and the assumption $0\leq t-s\leq \tau_N$, and the last one uses $N^{1/2}\tau_N\leq 1$ by Assumption~\ref{ass:sNtN}. Hence, \eqref{PIbdd:2} in the present case simplifies to
\begin{align*}
|g(r)-G(r)|& \leq  C(T_0,T_1)\prod_{j=1}^2 \big(N^{-1/2}+N^{-1/2}|x_j|+N^{-1/2}|y_j|\big)\\
&\quad  +C(T_0,T_1)\sum_{j=1}^2\frac{N^{-1/2}+N^{-1/2}|x_j|+N^{-1/2}|y_j|}{r^{1/2}}
\end{align*}
so that 
\begin{align}
&\quad  \int_1^{N(s-sr_N)}|g(r)-G(r)|\d r\notag\\
&\leq C(T_0,T_1)\Bigg((1-r_N)\prod_{j=1}^2(1+|x_j|+|y_j|)+(1-r_N)^{1/2}\sum_{j=1}^2 (1+|x_j|+|y_j|)\Bigg)\notag\\
&\leq C(T_0,T_1)(1-r_N)^{1/2}(1+|x|^2+|y|^2).\label{P2-3:eq1}
\end{align}

\noindent {\bf Step \hypertarget{P2-3-STEP3}{3}.} Recall the setup in \eqref{3434setup} and the assumption $0\leq t-s\leq \tau_N$. We consider 
\begin{align}
\begin{split}
 \int_1^{N(s-sr_N)}G(r)\d r
&= \int_{1}^{N(s-sr_N)}\prod_{j=1}^2\g\Bigg(2r+\frac{(Nt+N^{1/2}ty_j)(Nt-Ns)}{Nt};\\
&\hspace{2cm}\frac{(Nt+N^{1/2}ty_j)(Nt-Ns)}{Nt}+M_j-M_j'\Bigg)\d r.\label{G-gaussian}
\end{split}
\end{align}
This step is where we use the secondary condition \eqref{sec cond}.

To bound the Gaussian densities in \eqref{G-gaussian}, we first note that the variances therein satisfy the following bounds:
\begin{align}\label{Gaussian:var}
\begin{split}
2r+C(T_0,T_1)N(t-s)&\leq 2r+\frac{(Nt+N^{1/2}ty_j)(Nt-Ns)}{Nt}\\
&\leq 2r+C'(T_0,T_1)N(t-s),
\end{split}
\end{align}
where the primary condition \eqref{prim cond} is used. For the spatial variables in the Gaussian densities, we consider
\begin{align}
&\quad \frac{(Nt+N^{1/2}ty_j)(Nt-Ns)}{Nt}+M_j- M_j'\notag\\
&\geq  \frac{(Nt+N^{1/2}ty_j)(Nt-Ns)}{Nt}+\left(Ns+Ns\cdot \frac{x_j}{N^{1/2}}-1\right)-\left(Nt+Nt\cdot \frac{y_j}{N^{1/2}}\right)\notag\\
&=N^{1/2} y_j(t-s)+N^{1/2}s(x_j-y_j)-1-N^{1/2}(t-s)y_j\notag\\
&\geq -N^{1/2}\tau_N|y_j|+N^{1/2}s(x_j-y_j)-1-N^{1/2}\tau_N|y_j|\notag\\
&\geq  N^{1/2}s(x_j-y_j)-2,\label{use eta}
\end{align}
where the second inequality uses the assumption $0\leq t-s\leq\tau_N$ and the last inequality uses the primary condition \eqref{prim cond} and the definition of $\tau_N$ in Assumption~\ref{ass:sNtN}. Similarly,
\[
-\frac{(Nt+N^{1/2}ty_j)(Nt-Ns)}{Nt}-M_j+ M_j'\geq N^{1/2}t(y_j-x_j)-2.
\]
By \eqref{use eta} and the foregoing inequality, the secondary condition \eqref{sec cond} implies that 
\begin{align}\label{PO2}
\sum_{j=1}^2\left(\frac{(Nt+N^{1/2}ty_j)(Nt-Ns)}{Nt}+M_j-M_j'\right)^2\geq \frac{1}{2}NT_0^2|x-y|^2.
\end{align}

Recall that $1-r_N=\tau_N$. Applying \eqref{Gaussian:var} and \eqref{PO2} to \eqref{G-gaussian}, we get
\begin{align}
 \int_1^{N(s-sr_N)}G(r)\d r
&\leq \int_1^{NT_1\tau_N}\frac{C(T_0,T_1)}{r+N(t-s)}\exp\Bigg\{-\frac{C'(T_0,T_1)N|x-y|^2}{r+N(t-s)}\Bigg\}\d r\notag\\
&=\int_{1/N+(t-s)}^{T_1\tau_N+(t-s)}\frac{C(T_0,T_1)}{r}\exp\Bigg\{-\frac{C'(T_0,T_1)|x-y|^2}{r}\Bigg\}\d r\notag\\
&\leq \int_0^{(T_1+1)\tau_N/|x-y|^2}\frac{C(T_0,T_1)}{r}\exp\Bigg\{-\frac{C'(T_0,T_1)}{r}\Bigg\}\d r,\label{P2-3:eq1-1}
\end{align}
where we have changed variable in the last two lines and the last inequality uses $0\leq t-s\leq \tau_N$. Besides, since $N^{-1}\leq \tau_N\leq N^{-1/2}\leq C(T_0,T_1)|x-y|$ under the secondary condition \eqref{sec cond}, the last inequality implies that
\begin{align}\label{P2-3:eq2}
\int_1^{N(s-sr_N)}G(r)\d r\leq C(T_0,T_1)\big(1+\big|\ln |x-y|\big|\big).
\end{align}

\noindent {\bf Step 4.} Finally, \hyperlink{P2-3}{{\rm (3$^\circ$)}} follows from the second conclusion of Step~\hyperlink{P2-3-STEP1}{1}, \eqref{P2-3:eq1} and \eqref{P2-3:eq1-1} (since $\int_{0+} r^{-1}\e^{-r^{-1}}\d r<\infty $). As for \hyperlink{P2-4}{{\rm (4$^\circ$)}}, it is enough to prove the required inequality with  the term $\1_{[1,\infty)}(r)/(4\pi r)$ removed. Then we apply the other conclusion of Step~\hyperlink{P2-3-STEP1}{1}, \eqref{P2-3:eq1} and \eqref{P2-3:eq2} to the equivalent integral presented in \eqref{poisson:int2}. The proof is complete.
\end{proof}

The following theorem summarizes the asymptotic results proven in Propositions~\ref{prop:poisson1} and~\ref{prop:poisson2}.

\begin{thm}\label{thm:poisson} 
Let Assumption~\ref{ass:sNtN} be in force, and recall the notation $\mathfrak b^N$ in \eqref{def:Bkernel}. \smallskip

\noindent {\rm (1$^\circ$)} For all $0<s\leq t<\infty$ and $x,y\in \R^2$,
\begin{align*}
&\quad \lim_{N\to\infty}\int_0^{Ns\ell_N}\Bigg(\mathfrak b^N(x,y;r,s,t)
-\frac{\1_{[1,\infty)}(r)}{4\pi r}\Bigg)\d r
=\int_0^{\infty}\Bigg(\prod_{j=1}^2\mathbf P\big(V(r)=V'(r)\big)-\frac{\1_{[1,\infty)}(r)}{4\pi r}\Bigg)\d r.
\end{align*}

\noindent {\rm (2$^\circ$)} For all $s=t\in (0,\infty)$ and $x,y\in \R^2$ with $x\neq y$,
\begin{align*}
&\quad \lim_{N\to\infty}\int_{Nsr_N}^{Ns}\mathfrak b^N(x,y;r,s,t)\d r=0.
\end{align*}
\end{thm}

\subsection{Integrated normal approximations}\label{sec:normal}
We study the remaining integrals in \eqref{tauN1} and \eqref{tauN2}:
\begin{align}\label{iNP}
\int_{s\ell_N}^{sr_N}\mathfrak b_{N}(x,y;r,s,t)\d r
,\quad 
\int_{s\ell_N}^{s}\mathfrak b_{N}(x,y;r,s,t)\d r
\end{align}
by normal approximations. Here and in what follows, recall the notation in \eqref{def:Bkernel} for binomial probabilities. Let us begin with the following elementary result. For a Bernoulli random variable $\beta(p)$ with mean $p$,  we set $\psi_{p}(u)\,\defeq\,\bE\big[\e^{\i  u\overline{\beta}(p)}\big]=p\e^{\i  u(1-p)}+(1-p)\e^{-\i  u p}$, where $\overline{\beta}(p)=\beta(p)-p$.

\begin{lem}
It holds that 
\begin{align}
&|\psi_p(u)|^2
=1-2p(1-p)(1-\cos u),\quad \forall\; p\in (0,1),\;u\in \R;\label{xi:p1}\\
&\sup_{p\in (0,1)}\sup_{u:|u|\leq 1}|\psi_{p}(u)-1|<1.\label{xi:p2}
\end{align}
\end{lem}
\begin{proof}
For (\ref{xi:p2}), note that
$|\psi_{p}(u)-1|\leq p|\e^{\i u(1-p)}-1|+(1-p)|\e^{-\i up}-1|$, $u\mapsto |\e^{\i u}-1|^2=2-2\cos u$ is strictly increasing on $[0,\pi]$, and $|\e^{\i\pi/3}-1|=1$. \end{proof}

The next step is a counterpart of Lemma~\ref{lem:poisson1} in the context of normal approximations. Recall the notation in \eqref{GaussFT} for Gaussian densities. Note that the next lemma is not the binomial local central limit theorem (the de Moivre--Laplace theorem), and the forthcoming application needs the sum form  \eqref{def:para_lclt} of $\sigma^2$ for the bound in \eqref{ineq:quant_lclt}. See also Remark~\ref{rmk:stein}.

\begin{lem}\label{lem:normal1}
Given $q,q'\in (0,1)$ and integers $M,M',N\geq 1$, define
\begin{align}\label{def:para_lclt}
\mu=\frac{Mq-M'q'}{N^{1/2}}\quad \&\quad \sigma^2=\frac{M}{N}q(1-q)+\frac{M'}{N}q'(1-q'). 
\end{align}
Then
\begin{align}\label{ineq:quant_lclt}
\sup_{a\in \Bbb Z}\left|N^{1/2}\mathbf P\big(S_{ M}(q)=S'_{ M'}(q')+a\big) -\mathfrak g\left(\sigma^2;\frac{a}{N^{1/2}}-\mu\right)\right|\less \frac{1}{N^{1/2}\sigma^2}.
\end{align}
In particular, if the primary condition \eqref{prim cond} holds to ensure that $M_j,M_j'\geq 1$ (Remark~\ref{rmk:simple}) and $0<s\leq t<\infty$, then for all $r\in (0,s)$ and $1\leq j\leq 2$,
\begin{align}
\begin{split}
 \left| \mathfrak b_{N,j}(x_j,y_j;r,s,t)
  -\mathfrak g\big(\sigma_j(r;N)^2;\mu_j(r;N)\big)\right|
 &\less 
 \frac{1}{N^{1/2}\sigma_j(r;N)^{2}} 
 .\label{normal:bdd} \end{split}
\end{align}

\end{lem}
\begin{proof}
We reconsider the proof of Lemma~\ref{lem:poisson2} and proceed with the following identity:  By \eqref{GaussFT}, \eqref{discreteFT}, and the definition of $\mu$ in \eqref{def:para_lclt}, it holds that for all $a\in \Bbb Z$,
\begin{align}
 &\quad\;   N^{1/2}\mathbf P\big(S_{ M}(q)=S'_{ M'}(q')+a\big) -\mathfrak g\left(\sigma^2;\frac{a}{N^{1/2}}-\mu\right)\notag\\
 &=\frac{1}{2\pi  }\int_{|\theta|\leq N^{1/2} } \e^{-\i \theta\frac{a}{N^{1/2}}+\i \theta\mu }\left[\psi_{q}\left(\frac{\theta}{N^{1/2}}\right)^{M}\psi_{q'}\left(-\frac{\theta}{N^{1/2}}\right)^{M'}-\exp\left\{-\frac{\sigma^2\theta^2}{2}\right\}\right]\d\theta\notag\\
 & +\frac{1}{2\pi}\int_{N^{1/2} <|\theta|\leq N^{1/2} \pi}  \e^{-\i \theta\frac{a}{N^{1/2}}+\i \theta\mu }\left[\psi_{q}\left(\frac{\theta}{N^{1/2}}\right)^{M}\psi_{q'}\left(-\frac{\theta}{N^{1/2}}\right)^{M'}-\exp\left\{-\frac{\sigma^2\theta^2}{2}\right\}\right]\d\theta\notag\\
 &  + \frac{1}{2\pi  }\int_{|\theta|\geq N^{1/2} \pi}\e^{-\i \theta\frac{a}{N^{1/2}}+\i \theta\mu }\exp\left\{-\frac{\sigma^2\theta^2}{2}\right\}\d\theta\notag\\
  &={\rm I}_{\ref{integral}}+{\rm II}_{\ref{integral}}+{\rm III}_{\ref{integral}}.
  \label{integral}
\end{align}
The required bound in \eqref{ineq:quant_lclt} then follows upon applying (\ref{I2}) and (\ref{I3}), to be proven below, to \eqref{integral}.
\smallskip

\noindent {\bf Step \hypertarget{step-normal-1}{1}.}
For ${\rm I}_{\ref{integral}}$, we view 
\begin{align}
 &\quad \left|\psi_{q}\left(\frac{\theta}{N^{1/2}}\right)^{M}\psi_{q'}\left(-\frac{\theta}{N^{1/2}}\right)^{M'}\exp\left\{\frac{\sigma^2\theta^2}{5}\right\}-\exp\left\{-\frac{3\sigma^2\theta^2}{10}\right\}\right|,\, \forall\;|\theta|\leq N^{1/2},\label{z1z2bdd}
\end{align}
as $|\e^{z_1}-\e^{z_2}|$ and bound this difference in the way of \eqref{exp:MVT}. This use of \eqref{exp:MVT} is legitimate since we can take the logarithms of $\psi_q(\theta/N^{1/2})$ and $\psi_{q'}(-\theta/N^{1/2})$ by (\ref{xi:p2}).\smallskip

\noindent {\bf Step 1-1.}
To bound the term corresponding to $\max\{|\e^{z_1}|,|\e^{z_2}|\}$ in \eqref{exp:MVT}, we use
\eqref{xi:p1} and the inequality $1-v\leq \e^{-v}$ for all $v\geq 0$. They give
 \begin{align}
 \left|\psi_{q}\left(\frac{\theta}{N^{1/2}}\right)\right|^M\exp\left\{\frac{ Mq(1-q)\theta^2}{5N}\right\}
&\leq 
\exp\left\{2Mq\left(1-q\right)\left(\cos \frac{\theta}{N^{1/2}}-1+\frac{\theta^2}{10N}\right)\right\}\notag\\
&\leq \exp\left\{-\frac{Mq(1-q)\theta^2}{5N}\right\},\label{psi:est1}
 \end{align}
 where the last inequality holds whenever $|\theta|\leq N^{1/2}$ since $\cos v-1+v^2/5\leq 0$ for $|v|\leq 1$. Plainly, the same bound with $(q,M)$ replaced by $(q',M')$ on both sides holds. \smallskip

\noindent {\bf Step 1-2.}
Next, for \eqref{z1z2bdd}, we bound the term corresponding to $|z_1-z_2|$ in \eqref{exp:MVT}. We expand $\kappa(\i u)=\Log \psi_p(u)$ for $u\in \R$ around $0$, where $\kappa$ denotes the cumulant of $\overline{\beta}(p)$. We have $\kappa(0)$=0. Also, by definition, $\overline{\beta}(p)$ has a zero mean and variance $p(1-p)$, and so the first two derivatives of $u\mapsto \kappa(\i u)$ at $u=0$ are given by $0$ and $\i^2p(1-p)$, respectively.

The next bound is for the third-order derivative of $u\mapsto \kappa(\i u)$. It is chosen to incorporate the parameters of the Bernoulli random variables: Observe that $(\d/\d u)\kappa(\i u)=\mathbf E[\i \overline{\beta} (p)\e^{\i u \overline{\beta} (p)}]/\mathbf E[\e^{\i u \overline{\beta} (p)}]$. This expression implies that  every higher-order derivative is a ratio where the denominator is a power of $\psi_p(u)$ and the numerator is a sum of products of expectations of the form $\pm \mathbf E[(\i \overline{\beta} (p))^\ell \e^{\i u \overline{\beta} (p)}]$. Each product carries at least one such expectation with $\ell\geq 1$ and so can be bounded by $\mathbf E[|\overline{\beta}(p)|]=2p(1-p)$. Hence, by (\ref{xi:p2}),
\begin{align}\label{cum3:bdd}
\sup_{u:|u|\leq 1}\left|\frac{\d^3}{\d u^3}\kappa(\i u)\right|\less p(1-p),\quad \forall\;p\in (0,1).
\end{align}
Up to this point, we have proved that 
 \begin{align}\label{cumulant}
\left|\Log\psi_p(u)+\frac{p(1-p)}{2}u^2\right|\less p(1-p)|u|^3,\quad \forall\; |u|\leq 1,\; p\in (0,1).
 \end{align}
We remark that the factor $p(1-p)$ in the bound of \eqref{cumulant} will be useful.

Now, an application of the definition \eqref{def:para_lclt} of $\sigma^2$ and \eqref{cumulant} shows that the term $|z_1-z_2|$ in \eqref{exp:MVT} for bounding \eqref{z1z2bdd} satisfies 
 \begin{align}
  &\quad \Bigg|M\left[\Log \psi_{q}\left(\frac{\theta}{N^{1/2}}\right)+\frac{q(1-q)}{2}\frac{\theta^2}{N}\right]+M'\left[\Log \psi_{q'}\left(-\frac{\theta}{N^{1/2}}\right)
 +\frac{q'(1-q')}{2}\frac{\theta^2}{N}\right]\Bigg|\notag\\
 &\less \frac{|\theta|^3}{N^{1/2}}\left[\frac{M}{N}q(1-q)+\frac{M'}{N}q'(1-q')\right] = \frac{|\theta|^3}{N^{1/2}}\sigma^2,\quad \forall\; |\theta|\leq N^{1/2} ,\label{psi:est4}
 \end{align}
 where the last equality applies \eqref{def:para_lclt} again.\smallskip

\noindent {\bf Step 1-3.}
To finish the proof of Step \hyperlink{step-normal-1}{1}, we put together (\ref{psi:est1}), the analogous bound with $(q,M)$ replaced by $(q',M')$,  and (\ref{psi:est4}). Then we deduce that 
\begin{align*}
&\quad \left|\psi_{q}\left(\frac{\theta}{N^{1/2}}\right)^{M}\psi_{q'}\left(-\frac{\theta}{N^{1/2}}\right)^{M'}-\exp\left\{-\frac{\sigma^2\theta^2}{2}\right\}\right|\less \exp\left\{-\frac{\sigma^2\theta^2}{5}\right\}\frac{|\theta|^3}{N^{1/2}}\sigma^2 ,\; \forall\; |\theta|\leq N^{1/2}.
\end{align*} 
We arrive at the following bound:
\begin{align}
\begin{split}
|{\rm I}_{\ref{integral}}|\less &\frac{\sigma^2 }{N^{1/2}}\int_{0}^{N^{1/2} }\theta^3\exp\left\{-\frac{\sigma^2\theta^2}{5}\right\}\d\theta=\frac{1 }{N^{1/2}\sigma^2}\int_{0}^{N^{1/2} \sigma}\theta^3\exp\left\{-\frac{\theta^2}{5}\right\}\d\theta\less \frac{1}{N^{1/2}\sigma^2}.\label{I2}
\end{split}
\end{align}

\noindent {\bf Step 2.} Finally, we bound ${\rm II}_{\ref{integral}}+{\rm III}_{\ref{integral}}$. Notice that 
\[
|\psi_p(u)|^2\leq 1-4p(1-p)\cdot \frac{u^2}{\pi^2}\leq \exp\left\{-\frac{4p(1-p)u^2}{\pi^2}\right\},\quad \forall\;|u|\leq \pi,
\]
where the first inequality follows from \eqref{xi:p1} and the inequality $1-\cos u-2u^2/\pi^2\geq 0$ for all $|u|\leq \pi$, and the second inequality uses $1-v\leq \e^{-v}$ for all $v\geq 0$. It follows that 
\begin{align}
 |{\rm II}_{\ref{integral}}+{\rm III}_{\ref{integral}}|
 &\leq \frac{1}{2\pi }\int_{N^{1/2} <|\theta|\leq N^{1/2} \pi}
\exp\left\{-\frac{2\theta^2}{\pi^2}\left[\frac{M}{N}q(1-q)+\frac{M'}{N}q'(1-q')\right]\right\}
\d\theta\notag\\
&\quad +\frac{1}{2\pi }\int_{N^{1/2} <|\theta|\leq N^{1/2} \pi}\exp\left\{-\frac{\sigma^2\theta^2}{2}\right\}\d \theta+\frac{1}{2\pi  }\int_{|\theta|\geq N^{1/2} \pi} \exp\left\{-\frac{\sigma^2\theta^2}{2}\right\}\d\theta\notag\\
&\leq\int_{N^{1/2} }^\infty \exp\left\{-\frac{\sigma^2\theta^2}{5}\right\}\d \theta= \frac{1}{\sigma}\int_{N^{1/2}\sigma }^\infty \exp\left\{-\frac{\theta^2}{5}\right\}\d \theta\less \frac{1}{N^{1/2}\sigma^2}.
\label{I3}
\end{align}
In \eqref{I3}, the definition \eqref{def:para_lclt} of $\sigma^2$ is used. The proof is complete.
\end{proof}

\begin{rmk}\label{rmk:stein}
The proofs of Lemmas~\ref{lem:poisson2} and \ref{lem:normal1} extend the methods in \cite[Chapters XV and XVI]{Feller_v2} for proving the local central limit theorem of general lattice distributions. Here, we rely on the explicit forms of the characteristic functions to get the rates of convergence. Besides, in contrast to \eqref{integral:poisson}, we decompose the Fourier integral for $N^{1/2}\mathbf P\big(S_{ M}(q)=S'_{ M'}(q')+a\big) $ into two parts in \eqref{integral}. This validates \eqref{cumulant} since $\psi_{1/2}(\pi)=0$ by \eqref{xi:p1}. On the other hand, although it is not clear to us whether \eqref{integral:poisson} can be deduced from the known existing results, \cite[Theorem~9]{Petrov:1962} proves the rate of convergence for sums of i.i.d. general lattice-valued random variables (also by the Fourier analytic method). It may be possible to extend the proof in \cite{Petrov:1962} to \eqref{ineq:quant_lclt} by elaborating its use of the absolute third moments of the random variables in the Fourier integral decomposition. See also \cite[Section~1]{RR:local} for a general discussion for such rates. \mbox{} \hfill $\blacksquare$
\end{rmk}

The next proposition proves integrated normal approximations for the integrals in \eqref{iNP}. 
\begin{prop}\label{prop:normal1}
Fix $0<T_0<1<T_1<\infty$, and assume the primary condition \eqref{prim cond}. \smallskip

\noindent \hypertarget{normal1-1}{{\rm (1$^\circ$)}} It holds that 
\begin{align}
\begin{split}
&
\sup_{ s,t: T_0\leq s\leq t\leq T_1}
\int_{s\ell_N}^{sr_N} \Big|
\mathfrak b_{N}(x,y;r,s,t)
-\prod_{j=1}^2\mathfrak g\big(\sigma_j(r)^2;\mu_j(r)\big)\Big|\d r \leq  
\frac{C(T_0,T_1)}{ (N\ell_N)^{1/2}}.
\label{term:sumA}
\end{split}
\end{align}

\noindent \hypertarget{normal1-2}{{\rm (2$^\circ$)}} If we add the condition $t-s> \tau_N$ to the supremum in  (\ref{term:sumA})  and change the upper limit $sr_N$ of the integral on the left-hand side to $s$, then the same bound holds.\smallskip

\noindent \hypertarget{normal1-3}{{\rm (3$^\circ$)}}
The suprema in  {\rm (1$^\circ$)} and  {\rm (2$^\circ$)} tend to zero as $N\to\infty$ for all $x,y\in \R^2$. 
\end{prop}
\begin{proof}
Recall the quantities $\sigma_j(r;N)\geq 0$ and $\sigma_j(r)\geq 0$ defined in \eqref{def:mujsigmaj0} and \eqref{def:mujsigmaj}. We write $I_j(r;N)$ for the right-hand side of (\ref{normal:bdd}), $j\in \{1,2\}$, and define $\sigma(q)\geq 0$ by $\sigma(q)^2=\min\{\sigma_j(sq;N)^2,$ $\sigma_j(sq)^2;1\leq j\leq 2\}$. We stress that $\sigma(q)$ still depends on $s$ and $t$.

For the proof of the proposition, we consider the following four integrals for every choice of $(L,R)$ such that $0<L\leq1/2\leq R\leq 1$: 
\begin{align}
\begin{split}\label{lim2:I}
&\int_{sL}^{sR}I_1(r;N)I_2(r;N)\d r;\\
&\int_{sL}^{sR}I_j(r;N) \mathfrak g\big(\sigma_{j'}(r;N)^2;\mu_{j'}(r;N)\big)\d r,\quad 1\leq j,j'\leq 2,\;j\neq j';\\
&\int_{s L}^{sR}\Bigg|\prod_{j=1}^2 \g\big(\sigma_j(r;N)^2;\mu_j(r;N)\big)-\prod_{j=1}^2 \g\big(\sigma_j(r)^2;\mu_j(r)\big)\Bigg|\d r.
\end{split}
\end{align}
The proof of (1$^\circ$) below chooses $(L,R)$ to be $(\ell_N,1/2)$ and $(1/2,r_N)$. The sum of the corresponding eight integrals from \eqref{lim2:I} bounds the integral in \hyperlink{normal1-1}{\rm (1$^\circ$)}. For the proof of \hyperlink{normal1-2}{\rm (2$^\circ$)}, the integral  can be bounded in the same way by using $(\ell_N,1/2)$ and $(1/2,1)$.

Let us simplify the task of bounding the sum of the four integrals in \eqref{lim2:I} by making some observations. First, using the explicit form of $I_j(r;N)$ and changing variable $r$ to $q=r/s$, we can bound the sum of the first three integrals in (\ref{lim2:I})  by the sum of the following integrals up to a multiplicative constant $C(T_0,T_1)>0$: 
\begin{align}
\int_{L}^{R} \alpha(q) \alpha(q)\d q&= \frac{1}{N}\int_{L}^{R}\frac{\d q}{\sigma(q)^4},
\quad
\int_{L}^{R} \frac{\alpha(q)}{\sigma(q)}\d q= \frac{1}{N^{1/2}}
\int_{L}^{R} \frac{\d q }{\sigma(q)^3}.\label{I:sb}
\end{align}
where $\alpha(q)=1/(N^{1/2}\sigma(q)^2)$ and $\sigma(q)$ (depending on $s$ and $t$) is as defined at the beginning of the proof. Also, to bound the last integral in (\ref{lim2:I}), we use Lemma~\ref{lem:heat} and  the estimates $|\sigma_j(sq;N)^2-\sigma_j(sq)^2|\leq 2/N$ and $|\mu_j(sq;N)-\mu_j(sq)|\leq 2/N^{1/2}$, which gives
\begin{align}
&\quad \int_{L}^{R}\Bigg|\prod_{j=1}^2\mathfrak g\big(\sigma_j(sq;N)^2;\mu_j(sq;N)\big)- \prod_{j=1}^2\mathfrak g\big(\sigma_j(sq)^2;\mu_j(sq)\big)\Bigg|\d q\notag\\
&\less  \int_{L}^{R}\left(\frac{1}{N[\sigma(q)^2]^{3/2}}+\frac{1}{N^{1/2}[\sigma(q)^2]}\right)^2+\left(\frac{1}{N[\sigma(q)^2]^{3/2}}+\frac{1}{N^{1/2}[\sigma(q)^2]}\right)\cdot \frac{1}{[\sigma(q)^2]^{1/2}}\d q\notag\\
&\less \sum_{k=1}^4\int_L^R\frac{\d q}{N^{k/2}[\sigma(q)^2]^{k/2+1}}.\label{lim4}
\end{align}
By the preceding considerations,  we focus on the integrals in (\ref{I:sb}) and (\ref{lim4}), with $L$ and $R$ to be specified, in the rest of the proof. 
\smallskip

\noindent {(1$^\circ$).} We give the proof according to the decomposition $\int_{s\ell_N}^{sr_N}=\int_{s\ell_N}^{s/2}+\int_{s/2}^{sr_N}$ of the integral under consideration. We first consider the case $L=\ell_N$ and $R=1/2$ for the setup above. In view of the contribution of $(M_j/N)(r/s)(1-r/s)$ in $\sigma^2_j(r;N)$ and $\sigma_j^2(r)$ and the change of variable $q=r/s$, the following bound for $\sigma(q)^2$ holds:
\begin{align}\label{sigma:q1}
\sigma(q)^2\geq  C(T_0,T_1)q,\quad \forall\; q\in (0,1/2].
\end{align}
(The primary condition \eqref{prim cond} is used to get this bound and the analogues \eqref{sigma:q2} and \eqref{sigma:q3} below.)
Hence, by (\ref{I:sb}) and (\ref{lim4}),
\begin{align*}
&\int_{\ell_N}^{1/2} \alpha(q) \alpha(q)\d q\leq \frac{C(T_0,T_1)}{N\ell_N}\leq  \frac{C(T_0,T_1)}{ (N\ell_N)^{1/2}},\quad
  \int_{\ell_N}^{1/2} \frac{\alpha(q)}{\sigma(q)}\d q\leq  \frac{C(T_0,T_1)}{ (N\ell_N)^{1/2}},\\
&\int_{\ell_N}^{1/2}\left|\prod_{j=1}^2\mathfrak g\big(\sigma_j(sq;N)^2;\mu_j(sq;N)\big)- \prod_{j=1}^2\mathfrak g\big(\sigma_j(sq)^2;\mu_j(sq)\big)\right|\d q\leq  \frac{C(T_0,T_1)}{(N\ell_N)^{1/2}}.
\end{align*}
Note that these bounds are for integrals of functions exploding at $q=0$.

For the second case, we take $L=1/2$ and $R=r_N=1-\ell_N$. The lower bound for $\sigma(q)^2$ is now taken to be
\begin{align}\label{sigma:q2}
\sigma(q)^2\geq  C(T_0,T_1)(1-q),\quad \forall\;q\in [1/2,1).
\end{align}
In this case, the singularities in the integrals from (\ref{I:sb}) and (\ref{lim4}) are changed to $q=1$. However,  since $\ell_N=1-r_N$, a change of variable with $1-q$ replaced by $q$ shows that the bounds in the above case apply. Putting together the bounds in the two cases yields \hyperlink{normal1-1}{(1$^\circ$)}.\smallskip

\noindent {(2$^\circ$).} We use the decomposition $\int_{s\ell_N}^{s}=\int_{s\ell_N}^{s/2}+\int_{s/2}^{s}$ of the integral under consideration. The same argument for $\int_{s\ell_N}^{s/2}$ in the proof of  \hyperlink{normal1-1}{(1$^\circ$)} applies in this case.

We change the argument for the second case in Step~\hyperlink{normal1}{1} by taking $t-s> \tau_N$, $L=1/2$ and $R=1$. By considering $1-sq/t=[t-s+s(1-q)]/t$ in bounding $(M_j'/N)(r/t)(1-r/t)$, the lower bound in (\ref{sigma:q2}) is replaced by
\begin{align}\label{sigma:q3}
\sigma(q)^2\geq  C(T_0,T_1)[\tau_N+(1-q)] ,\quad \forall\; q\in [1/2,1).
\end{align}
Hence, with a translation of $(1-q)$ by $\tau_N$ in the domains of integration, we can still use the bounds for the second case in the proof of \hyperlink{normal1-1}{(1$^\circ$)}, except that $\ell_N$ is replaced by $\tau_N$.
Since $\tau_N=\ell_N$ by assumption, we have proved \hyperlink{normal1-1}{(2$^\circ$)}.
\smallskip

\noindent {(3$^\circ$).}
The required limits hold under Assumption~\ref{ass:sNtN}. The proof is complete.
\end{proof}

Finally, we pass limit under the integral sign by the following proposition.  Note that the first integral in \eqref{eq:normal2} converges absolutely by the inequality $1-\e^{-v}\leq v$ for all $v\geq 0$.

\begin{prop}\label{prop:normal2}
Let Assumption~\ref{ass:sNtN} be in force. \smallskip

\noindent \hypertarget{N2-1}{{\rm (1$^\circ$)}} Let $0<s\leq t<\infty$ and $x,y\in \R^2$ be such that either  $s<t$ or $x\neq y$. Then 
for $1/2<\widetilde{r}_N\leq 1$ with $\widetilde{r}_N\to 1$, it holds that 
\begin{align}
\begin{split}
&\quad \lim_{N\to\infty}\int_{s\ell_N}^{s\widetilde{r}_N}\Bigg(\prod_{j=1}^2\mathfrak g\big(\sigma_j(r)^2;\mu_j(r)\big)-\frac{\1_{[1,\infty)}(Nr)}{4\pi r}\Bigg)\d r\label{eq:normal-normal}\\
&\quad +\frac{\ln s}{4\pi}-\int_0^{\infty}\frac{\big(\e^{-\frac{1}{4v}}-\1_{[1,\infty)}(v)\big)}{4\pi v}\d v\\
\end{split}\\
\begin{split}\label{eq:normal2}
&=\frac{1}{2\pi}\int_{\R^2}\int_{\R^2}Q_{s^{-1}}(y',x)\big(-\ln |y'-y''|\big)Q_{t^{-1}}(y'',y)\d y'\d y''\\
&\quad +\int_0^{t^{-1}} \int_{\R^2} Q_{s^{-1}-r}(z,x)Q_{t^{-1}-r}(z,y)\d z\d r.
\end{split}
\end{align}

\noindent \hypertarget{N2-2}{{\rm (2$^\circ$)}} Fix $0<T_0<1<T_1<\infty$.  Under the primary condition \eqref{prim cond}, it holds that, 
\begin{align*}
\int_{s\ell_N}^s \Bigg|\prod_{j=1}^2\mathfrak g\big(\sigma_j(r)^2;\mu_j(r)\big)-\frac{\1_{[1,\infty)}(Nr)}{4\pi r}\Bigg|\d r\leq C(T_0,T_1)\big(1+|x|^2+|y|^2+\big|\ln |x-y|\big|\big).
\end{align*}
\end{prop}

The proof  uses the following property.

\begin{lem}\label{lem:Tglue}
For all $y',y''\in \R^2$ with $y'\neq y''$ and $T\in (0,\infty)$, it holds that
\begin{align}\label{limT:glue}
\begin{split}
\int_0^{T}Q_{2r}(y',y'')\d r&=\int_0^{\infty}\frac{\big(\e^{-\frac{1}{4v}}-\1_{[1,\infty)}(v)\big)}{4\pi v}\d v-\frac{1}{2\pi}\ln |y'-y''|\\
&\quad +\frac{\ln T}{4\pi} +\sum_{j=1}^3 \vep_j(y',y'';T)
\end{split} 
\end{align}
for error functions
\begin{align*}
\vep_1(y',y'';T)&=-\int_{\frac{T}{|y'-y''|^2}}^\infty\frac{\big(\e^{-\frac{1}{4v}}-\1_{[1,\infty)}(v)\big)}{4\pi v}\d v ,
\\
\vep_2(y',y'';T)&=-\1_{(0,1)}\left(\frac{T}{|y'-y''|^2}\right)\cdot \frac{1}{4\pi}\ln T,\\
\vep_3(y',y'';T)&=\1_{(0,1)}\left(\frac{T}{|y'-y''|^2}\right)\cdot \frac{1}{2\pi}\ln |y_1-y_2|
\end{align*}
satisfying the following limits for all $1\leq j\leq 3$, $x,y\in \R^2$, and $0< s\leq t<\infty$:
\begin{align}\label{Q:conv}
\begin{split}
&\lim_{T\to\infty}\int_{\R^2}\int_{\R^2}Q_{s^{-1}}(y',x)\big|\vep_j(y',y'';T)\big|Q_{t^{-1}}(y'',y)\d y'\d y''=0.
\end{split}
\end{align}
\end{lem}
\begin{proof}
To see \eqref{limT:glue}, note that by changing variables, 
\begin{align*}
\int_0^{T}Q_{2r}(y',y'')\d r&=\int_0^{\frac{T}{|y'-y''|^2}}\frac{\big(\e^{-\frac{1}{4v}}-\1_{[1,\infty)}(v)\big)}{4\pi v}\d v\\
&\quad +\1_{[1,\infty)}\left(\frac{T}{|y'-y''|^2}\right)\cdot \frac{1}{4\pi}\ln \left(\frac{T}{|y'-y''|^2}\right).
\end{align*}
Hence, $\vep_j$'s in \eqref{limT:glue} are introduced as error terms for large $T/|y'-y''|^2$.

Now, the convergence in \eqref{Q:conv} for $j=1$ holds by dominated convergence. For the other cases, let $B,B'$ be two independent copies of the two-dimensional standard Brownian motion. Then for $j=2$ and $T>1$, 
\begin{align}
0&\leq  \frac{\ln T}{4\pi } \int_{\R^2}\int_{\R^2}Q_{s^{-1}}(y',x)\1_{(0,1)}\left(\frac{T}{|y'-y''|^2}\right)Q_{t^{-1}}(y'',y)\d y'\d y''\notag\\
&= \frac{\ln T}{4\pi }\mathbf P\left(|x+B_{s^{-1}}-y-B'_{t^{-1}}|>T^{1/2}\right)\notag\\
&= \frac{\ln T}{4\pi }\mathbf P\left(|x-y+B_{s^{-1}+t^{-1}}|>T^{1/2}\right)\notag\\
&\leq \frac{\ln T}{4\pi }\cdot \frac{|x-y|+\mathbf E[|B_{s^{-1}+t^{-1}}|]}{T^{1/2}}\label{est:normal1}
\end{align}
by the Markov inequality. Passing $T\to\infty$ in the foregoing inequality, we see that the required limit is zero. For the required limit with $j=3$, the dominated convergence theorem applies since $|\vep_3(y',y'';T)|\leq \big|\ln |y'-y''|\big|$ for all $T>0$ and, for any $\alpha\in (0,1]$, 
\begin{align}
&\quad\int_{\R^2}\int_{\R^2}Q_{s^{-1}}(y',x)\big|\ln |y'-y''|\big|Q_{t^{-1}}(y'',y)\d y'\d y''\notag\\
&\less   \int_{\R^2}\int_{\R^2}Q_{s^{-1}}(y',x)\left(|y'-y''|+\frac{C(\alpha)}{ |y'-y''|^\alpha}\right)Q_{t^{-1}}(y'',y)\d y'\d y''\notag\\
\begin{split}
&\leq \mathbf E\left[|x+B_{s^{-1}}-y-B'_{t^{-1}}|\right]+C(\alpha)\int_{\R^2}Q_{s^{-1}}(y',x)\\
&\quad \times \left(\int_{\R^2}Q_{t^{-1}}(y'',y)\d y''+\|Q_{t^{-1}}(\cdot,y)\|_\infty\int_{|y'-y''|\leq 1}\frac{1}{|y'-y''|^\alpha}\d y''\right)\d y' <\infty.\label{est:normal2}
\end{split}
\end{align}
The proof is complete.
\end{proof}

\begin{proof}[\sc Proof of Proposition~\ref{prop:normal2}]
{(1$^\circ$)}
Recall \eqref{def:mujsigmaj} for $\sigma_j(r)$ and $\mu_j(r)$. Write
\begin{align}
 \int_{ s\ell_N}^s\prod_{j=1}^2\mathfrak g\big(\sigma_j(r)^2;\mu_j(r)\big)\d r
&=\int_{s \ell_N}^s\frac{1}{r^2}\prod_{j=1}^2 \g\left(\frac{2}{r}-\frac{1}{s}-\frac{1}{t};x_j-y_j\right)\d r\notag
\\
&=\int_0^{s^{-1}( \ell_N^{-1}-1)} Q_{2r'+s^{-1}-t^{-1}}(x,y)\d r'\label{Gaussian:int},
\end{align}
where we change variable by $r'=r^{-1}-s^{-1}$. The foregoing integral is finite whenever $s<t$ or $x\neq y$. Hence, the proof of (1$^\circ$) for $\widetilde{r}_N=1$  suffices by dominated convergence. We consider this case in the rest of the proof of (1$^\circ$).

Whenever $N$ is large enough such that $s^{-1}(\ell_N^{-1}-1)>t^{-1}$, by using the Chapman--Kolmogorov equation and changing variables as $r=r'-t^{-1}$ and $r=t^{-1}-r'$, we can write the last integral as
\begin{align}
&\quad \left(\int_{t^{-1}}^{s^{-1}( \ell_N^{-1}-1)}+\int_{0}^{t^{-1}}\right)  Q_{2r'+s^{-1}-t^{-1}}(x,y)\d r' \notag\\
\begin{split}
&=\int_{\R^2}\int_{\R^2}Q_{s^{-1}}(y',x)\left(\int_{0}^{s^{-1}( \ell_N^{-1}-1)-t^{-1}} Q_{2r}(y',y'')\d r\right)Q_{t^{-1}}(y'',y)\d y'\d y''\\
&\quad +\int_0^{t^{-1}} \int_{\R^2} Q_{s^{-1}-r}(z,x)Q_{t^{-1}-r}(z,y)\d z\d r.\label{kernel1}
\end{split}
\end{align}
The first term on the right-hand side of \eqref{kernel1} shows the integral on the left-hand side of \eqref{limT:glue} with $T=s^{-1}( \ell_N^{-1}-1)-t^{-1}$. Recall \eqref{Gaussian:int} and $N\ell_N\to\infty$ by Assumption~\ref{ass:sNtN}. By using $s^{-1}\ell_N^{-1}\sim s^{-1}( \ell_N^{-1}-1)-t^{-1}$ as $N\to\infty$ and then applying Lemma~\ref{lem:Tglue} to \eqref{kernel1},
\begin{align}
&\quad \lim_{N\to\infty} \int_{ s\ell_N}^s\Bigg(\prod_{j=1}^2\mathfrak g\big(\sigma_j(r)^2;\mu_j(r)\big)-\frac{\1_{[1,\infty)}(Nr)}{4\pi r}\Bigg)\d r\notag\\
&=\lim_{N\to\infty} \Bigg(\int_{ s\ell_N}^s\prod_{j=1}^2\mathfrak g\big(\sigma_j(r)^2;\mu_j(r)\big)\d r-\frac{\ln s}{4\pi }-\frac{\ln [s^{-1}( \ell_N^{-1}-1)-t^{-1}]}{4\pi}\Bigg)\notag\\
\begin{split}
&=-\frac{\ln s}{4\pi}  +\int_0^{\infty}\frac{\big(\e^{-\frac{1}{4v}}-\1_{[1,\infty)}(v)\big)}{4\pi v}\d v\\
&\quad +\frac{1}{2\pi}\int_{\R^2}\int_{\R^2}Q_{s^{-1}}(y',x)\big(-\ln |y'-y''|\big)Q_{t^{-1}}(y'',y)\d y'\d y''\\
&\quad +\int_0^{t^{-1}} \int_{\R^2} Q_{s^{-1}-r}(z,x)Q_{t^{-1}-r}(z,y)\d z\d r.\label{kernel:lim2}
\end{split}
\end{align}
We have obtained \eqref{eq:normal2} for the case $\widetilde{r}_N=1$ from \eqref{kernel:lim2}. \smallskip

\noindent {(2$^\circ$).} By the primary condition~\eqref{prim cond}, $NT_0\ell_N\geq 1$. Hence, the indicator function in the integral under consideration can only take the value $1$. We bound this integral according to $\int_{s\ell_N}^s=\int_{s\ell_N}^{s/2}+\int_{s/2}^s$. First, consider
\begin{align}
&\quad \int_{s\ell_N}^{s/2} \left|\prod_{j=1}^2\mathfrak g\big(\sigma_j(r)^2;\mu_j(r)\big)-\frac{1}{4\pi r}\right|\d r\notag\\
&\leq \int_{s\ell_N}^{s/2} \left|\frac{1}{2\pi  r(2-r/s-r/t)}\exp\left\{-\frac{|x-y|^2r^2}{2r(2-r/s-r/t)}\right\}-\frac{1}{2\pi r(2-r/s-r/t)}\right|\d r\notag\\
&\quad +C(T_0,T_1)\notag\\
&\leq C(T_0,T_1)(1+|x-y|^2)\leq C(T_0,T_1)(1+|x|^2+|y|^2),\label{N2-2-1}
\end{align}
where \eqref{N2-2-1} uses the inequality $1-\e^{-v}\leq v$, which is valid for all $v\geq 0$. Also,
\begin{align}
 \int_{s/2}^{s} \prod_{j=1}^2\mathfrak g\big(\sigma_j(r)^2;\mu_j(r)\big)\d r
&\leq \int_{s/2}^{s} \frac{1}{2\pi  r^2(2/r-1/s-1/t)}\exp\left\{-\frac{|x-y|^2}{2(2/r-1/s-1/t)}\right\}\d r\notag\\
&=\int_{\frac{s^{-1}-t^{-1}}{|x-y|^2}}^{\frac{3s^{-1}-t^{-1}}{|x-y|^2}} \frac{1}{4\pi v}\e^{-\frac{1}{2v}}\d v
\leq C(T_0,T_1)\big(1+\big|\ln |x-y|\big|\big).\label{N2-2-2}
\end{align}
The required inequality follows upon combining \eqref{N2-2-1} and \eqref{N2-2-2}. 
\end{proof}

The following theorem summarizes Proposition~\ref{prop:normal1} (3$^\circ$) and Proposition~\ref{prop:normal2}
(1$^\circ$).

\begin{thm}\label{thm:normal}
Let Assumption~\ref{ass:sNtN} be in force, and recall the notation $\mathfrak b_N$ in \eqref{def:Bkernel}.
\smallskip

\noindent {\rm (1$^\circ$)} For all $0<s< t<\infty$ and $x,y\in \R^2$,
\begin{align*}
&\quad \lim_{N\to\infty}\int_{s\ell_N}^{s}\Big(
\mathfrak b_N(x,y;r,s,t)
-\frac{\1_{[1,\infty)}(Nr)}{4\pi r}\Big)\d r +\frac{\ln s}{4\pi}-\int_0^{\infty}\frac{\big(\e^{-\frac{1}{4v}}-\1_{[1,\infty)}(v)\big)}{4\pi v}\d v\\
&= \frac{1}{2\pi}\int_{\R^2}\int_{\R^2}Q_{s^{-1}}(y',x)\big(-\ln |y'-y''|\big)Q_{t^{-1}}(y'',y)\d y'\d y''\\
&\quad +\int_0^{t^{-1}} \int_{\R^2} Q_{s^{-1}-r}(z,x)Q_{t^{-1}-r}(z,y)\d z\d r.
\end{align*}

\noindent {\rm (2$^\circ$)} For all $s=t\in (0,\infty)$ and $x,y\in \R^2$ with $x\neq y$, the same limit in (1$^\circ$)
 holds if we change the upper limits $s$ of the integrals on the left-hand side to $sr_N$. 
\end{thm}

\section{Convergence to the additive stochastic heat equation}\label{sec:ASHE}
In this section, we relate the limiting covariance function in Theorem~\ref{thm:covar}
to the covariance function of an additive stochastic heat equation. Whereas some of these connections are already pointed out in \cite{BCF}, we proceed with the weak formulation.

From now on, $\mathcal S(\R^2)$ denotes the space of real-valued Schwartz functions on $\R^2$ and $\mathcal S'(\R^2)$ denotes the space of bounded linear functionals on $\mathcal S(\R^2)$ over $\R$. By convention, $\mathcal S'(\R^2)$ is equipped with the weak topology \cite{RS}.

\subsection{Weak formulations}
With respect to the process $\zeta^N(x,s)$ in \eqref{def:zetaN1}, we define 
\begin{align}\label{def:zetaN2}
\zeta^N_s(\phi)=\int_{x\geq -N^{1/2}\mathbf 1}\phi(x)\zeta^N(x,s)\d x,\quad \phi\in \mathcal S(\R^2).
\end{align}
Here, $\mathbf 1=(1,1)$. The constraint $x\geq -N^{1/2}\mathbf 1$ is maximal for using the Whittaker SDEs since for any $s>0$, $M(x_j,s)\geq 0$ if and only if $x_j\geq -N^{1/2}$. (Recall \eqref{def:Morigin} for the notation $M(x_j,s)$.) In this subsection, we show some basic growth properties of the process $\zeta^N(x,s)$ and then translate Theorem~\ref{thm:covar} to a convergence result under the weak formulation.

By Proposition~\ref{prop:covar}, the metric induced by the covariance function of $\zeta_s\Sigma(m_1,m_2)$ can be represented as follows: for every $(m_1,m_2),(m_1',m_2')\in \Bbb Z_+^2$ and $0\leq s\leq t<\infty$,
\begin{align}
&\quad \E[|\zeta_s\Sigma(m_1,m_2)-\zeta_t\Sigma(m_1',m_2')|^{2}]\notag\\
\begin{split}\label{covar}
&=\int_s^t\prod_{j=1}^2\mathbf P\left(S_{m_j'}\left(\frac{r}{t}\right)=S'_{m_j'}\left(\frac{ r}{t}\right)\right)\d r\\
&\quad -\Bigg[\int_0^s\prod_{j=1}^2\mathbf P\left(S_{m_j}\left(\frac{r}{s}\right)=S'_{m_j'}\left(\frac{ r}{t}\right)\right)\d r-\int_0^s\prod_{j=1}^2\mathbf P\left(S_{m_j}\left(\frac{r}{s}\right)=S'_{m_j}\left(\frac{ r}{s}\right)\right)\d r\Bigg]\\
&\quad -\Bigg[\int_0^s\prod_{j=1}^2\mathbf P\left(S_{m_j}\left(\frac{r}{s}\right)=S'_{m_j'}\left(\frac{ r}{t}\right)\right)\d r-\int_0^s\prod_{j=1}^2\mathbf P\left(S_{m_j'}\left(\frac{r}{t}\right)=S'_{m_j'}\left(\frac{ r}{t}\right)\right)\d r\Bigg].
\end{split}
\end{align}
Here, the general identity in use is
\begin{align*}
&\quad \int_0^s f(r;s,s)\d s-2\int_0^s f(r;s,t)\d r+\int_0^t f(r;t,t)\d r\\
&=\int_s^t f(r;t;t)-\left[\int_0^s f(r;s,t)\d r-\int_0^s f(r;s,s)\d r\right]-\left[\int_0^s f(r;s,t)\d r-\int_0^sf(r;t,t)\d r\right].
\end{align*}

\begin{lem} \label{lem:bin}
{\rm (1$^\circ$)} 
Given $0<r< a$ and integers $m,n\geq 0$, it holds that 
\begin{align}\label{derivativeP}
\begin{split}
\frac{\partial}{\partial a}\mathbf P\left(S_m\left(\frac{r}{a}\right)=n\right)
&=\frac{1}{a}\left[(n+1)\mathbf P\left(S_m\left(\frac{r}{a}\right)=n+1\right)-n\mathbf P\left(S_m\left(\frac{r}{a}\right)=n\right)  \right].
\end{split}
\end{align}

\noindent {\rm (2$^\circ$)} Given $T\in (0,\infty)$, it holds that
\begin{align}\label{mod0}
\begin{split}
&\E[|\zeta_s\Sigma(m_1,m_2)-\zeta_t\Sigma(m_1,m_2)|^2]  \leq C(T) \big(\|(m_1,m_2)\|_\infty\vee 1\big)\times  |t-s|,
\end{split}
\end{align}
for all $0\leq s,t\leq T$ and $(m_1,m_2)\in \Bbb Z_+^2$, where $\|(m_1,m_2)\|_\infty=\max\{|m_1|,|m_2|\}$.
\end{lem}
\begin{proof}
To obtain \eqref{derivativeP}, we may assume that $0\leq n\leq m$ and then consider:
\begin{align*}
&\quad \frac{\partial}{\partial a}\mathbf P\left(S_m\left(\frac{r}{a}\right)=n\right)\\
&={m\choose n}n\left(\frac{r}{a}\right)^{n-1}\left(\frac{-r}{a^2}\right)\left(1-\frac{r}{a}\right)^{m-n}+{m\choose n}\left(\frac{r}{a}\right)^{n}(m-n)\left(1-\frac{r}{a}\right)^{m-n-1}\frac{r}{a^2}\\
&=-\frac{n}{a}\mathbf P\left(S_m\left(\frac{r}{a}\right)=n\right)+\frac{n+1}{a}\mathbf P\left(S_m\left(\frac{r}{a}\right)=n+1\right).
\end{align*}

Next, we show that \eqref{derivativeP} implies \eqref{mod0}. In the case $s=0$ or $m_1=m_2=0$, the required bound holds obviously, since then the second and third terms on the right-hand side of \eqref{covar} with $(m_1',m_2')=(m_1,m_2)$ are zero. For $0<s\leq t<\infty$ and nonzero $(m_1,m_2)\in \Bbb Z_+^2$, \eqref{derivativeP} implies the following bound for the second term in \eqref{covar} with $(m_1',m_2')=(m_1,m_2)$:
\begin{align*}
&\quad \int_0^s\Bigg|\prod_{j=1}^2\mathbf P\left(S_{m_j}\left(\frac{r}{s}\right)=S'_{m_j}\left(\frac{ r}{s}\right)\right)-\prod_{j=1}^2\mathbf P\left(S_{m_j}\left(\frac{r}{s}\right)=S'_{m_j}\left(\frac{ r}{t}\right)\right)\Bigg|\d r\\
&\less 
  \int_0^s \frac{|t-s|}{s} \sum_{j=1}^2\left(\mathbf E\left[S'_{m_j}\left(\frac{r}{s}\right)\right]+1\right) \d r\less \|(m_1,m_2)\|_\infty\times |t-s|.
\end{align*}
The third term in \eqref{covar} with $(m_1',m_2')=(m_1,m_2)$ can be bounded similarly. Hence, \eqref{mod0} holds whenever $0<s\leq t<\infty$ and $(m_1,m_2)$ is nonzero. We have proved \eqref{mod0}. 
\end{proof}

As an application, we obtain the a.s. polynomial growth of $\zeta$ in the following lemma.
 
\begin{lem}
\hypertarget{regular1}{{\rm (1$^\circ$)}} 
For all $\alpha\in (1,\infty)$, we can find $C(\alpha)\in (0,\infty)$ such that
\[
\E\left[\sup_{(m_1,m_2)\in \Bbb Z_+^2}\sup_{s\in [0,T]}\frac{|\zeta_s\Sigma(m_1,m_2)|^{2\alpha}}{1+\|(m_1,m_2)\|_\infty^{C(\alpha)}}\right]<\infty,\quad \forall\; T\in (0,\infty).
\]

\noindent  \hypertarget{regular2}{{\rm (2$^\circ$)}}  For each $N\geq 1$, the following statement holds with probability one:
the integral in \eqref{def:zetaN2} converges absolutely for all $s\in [0,\infty)$ and $\phi\in \mathcal S(\R^2)$, and $\zeta^N$
takes values in $D(\R_+,\mathcal S'(\R^2))$. 
\end{lem}
\begin{proof}
{(1$^\circ $).} We modify the proof of \cite[Proposition~4.1]{C:RW} as follows. For any integer $n\geq 1$, set $E_n=\{(m_1,m_2)\in \Bbb Z_+^2;2^{n-1}\leq \|(m_1,m_2)\|_\infty<2^n\}$. For any $\beta\in (0,\infty)$ and $\alpha\in (1,\infty)$,
\begin{align}
&\quad \E\left[\sup_{(m_1,m_2)\in \Bbb Z_+^2}\sup_{s\in [0,T]}\frac{|\zeta_s\Sigma(m_1,m_2)|^{2\alpha}}{1+\|(m_1,m_2)\|_\infty^{{\beta }}}\right]\notag\\
\begin{split}
&\leq \E\left[\sup_{s\in [0,T]}|\zeta_{s}\Sigma(0,0)|^{2\alpha}\right] +\sum_{n=1}^\infty  \frac{1}{1+2^{\beta(n-1)}}\sum_{(m_1,m_2)\in E_n}\E\left[\sup_{s\in [0,T]}|\zeta_{s}\Sigma(m_1,m_2)|^{2\alpha}\right].\label{CT010}
\end{split}
\end{align}
Since $\zeta$ is a Gaussian process, Lemma~\ref{lem:bin} (2$^\circ$) is enough for the application of Kolmogorov's theorem of continuity  \cite[(2.1)Theorem in Chapter~I]{RY}. Moreover, we can find $C'(\alpha,T),C(\alpha)$ such that the expected supremum in \eqref{CT010} indexed by $(m_1,m_2)$ is bounded by $C'(\alpha,T)\|(m_1,m_2)\|_\infty^{C(\alpha)}$. We get the required result upon setting $\beta=3+C(\alpha)$ in \eqref{CT010}.\smallskip

\noindent {(2$^\circ $).} For $x_j\geq -N^{1/2}$, $s\mapsto sN+sN\cdot x_j/N^{1/2}$ is nondecreasing on $[0,\infty)$ and so $s\mapsto M_j$ is c\`adl\`ag. For these $x_j$'s, we also have $0\leq M_j\leq Ns+Ns\cdot |x_j|/N^{1/2}$. Hence, by  (1$^\circ$), $\zeta_s^N(\phi)$ is absolutely convergent as an integral, for any $s$, and is c\`adl\`ag in $s$.  The weak topology of $\mathcal S'(\R^2)$ gives the required  path property of $\zeta^N$. 
\end{proof}

Now we extend Theorem~\ref{thm:covar} to a convergence under the weak formulation.

\begin{thm}\label{thm:covar1}
Let $\zeta^N$ be the $\mathcal S'(\R^d)$-valued processes defined by \eqref{def:zetaN2}.
Then it holds that 
\begin{align}\label{covar:phi1phi2......}
\begin{split}
&\quad \lim_{N\to\infty}\left[\Cov[\zeta^N_s(\phi_1);\zeta^N_t(\phi_2)]-\mathfrak C_N\left(\int \phi_1\right)\left(\int\phi_2\right)\right]\\
&=\frac{1}{2\pi}\int_{\R^2}\int_{\R^2}Q_{s^{-1}}\phi_1(y')\big(-\ln |y'-y''|\big)Q_{t^{-1}}\phi_2(y'')\d y'\d y''\\
&\quad +\int_0^{t^{-1}} \int_{\R^2} Q_{s^{-1}-r}\phi_1(z)Q_{t^{-1}-r}\phi_2(z)\d z\d r,\quad \forall\;\phi_1,\phi_2\in \mathcal S(\R^2),\;0<s\leq t<\infty,
\end{split}
\end{align}
where $\mathfrak C_N$ is defined in \eqref{def:CN}. 
\end{thm}
\begin{proof}
First, we state a preliminary bound by summarizing those in Propositions \ref{prop:poisson1}, \ref{prop:poisson2}, \ref{prop:normal1} and \ref{prop:normal2} via \eqref{tauN1} and \eqref{tauN2}: Fix $0<T_0<1<T_1<\infty$. Let  the primary condition \eqref{prim cond} be in force. If $0\leq t-s\leq \tau_N$, we also require the secondary condition \eqref{sec cond}. Then it holds that
\begin{align}
& \int_0^{Ns}\Big|\mathfrak b^N(x,y;r,s,t)
-\frac{\1_{[1,\infty)}(r)}{4\pi r}\Big|\d r 
\leq C(T_0,T_1)\left(1+|x|^2+|y|^2+\big|\ln |x-y|\big|\right) .\label{bdd:abs conv}
\end{align}

We show the main term of  the required limit. First, recall that Theorem~\ref{thm:covar} is obtained by summing the limits in Theorems~\ref{thm:poisson} and~\ref{thm:normal}. Next, we shift the integrand  of the integral in \eqref{obj1} by $\1_{[1,\infty)}(r)/(4\pi r)$ (as in Theorem~\ref{thm:poisson}) and define an integral $I_N(x,y;s,t)$ accordingly. Let $G_N$ denote the set of $(x,y)$ that satisfies the assumption for \eqref{bdd:abs conv}. By dominated convergence, we can pass the limit  $\lim_N\int_{G_N}I_N(x,y;s,t)\phi_1(x)\phi_2(y)\d x\d y$ under the integral sign. The limit is given by the right-hand side of \eqref{covar:phi1phi2......}.

It remains to show $\int_{G_N^\complement}I_N(x,y;s,t)\phi_1(x)\phi_2(y)\d x\d y\to 0$. We consider two cases. To handle the integral over $(x,y)$ that fails to satisfy the primary condition \eqref{prim cond}, note that for any $p>0$ and $\phi\in \mathcal S(\R^2)$, $\int_{ ([-N^\eta/2,N^\eta/2]^2)^\complement}|\phi(x)|\d x \leq C(\eta,p,\phi)N^{-p}$. Next, let $B_N$ denote the set of $(x,y)$ that fails the secondary condition \eqref{sec cond}. Recall that the secondary condition is only required in Proposition~\ref{prop:poisson2} (4$^\circ$) in getting \eqref{bdd:abs conv}. Then it is enough to note that for that integral in Proposition~\ref{prop:poisson2},
\[
\int_{B_N}\left(\int_{Nsr_N}^{Ns}
\d r\right)|\phi_1(x)\phi_2(y)| \d x\d y\leq C(T_0,T_1,\phi_1,\phi_2) N^{-1/2}\cdot N(1-r_N)=N^{-\eta}\xrightarrow[N\to\infty]{}0,
\]
where the choice of $1-r_N$ in Assumption~\ref{ass:sNtN} is used. This completes the proof.
\end{proof}

This proof shows that it can be reinforced a bit to get the following implication.

\begin{cor}\label{cor:uniformcovar}
For all $0<T_0<1<T_1<\infty$ and $\phi_1(x;r,s,t),\phi_2(x;r,s,t)$ such that 
\[
|\phi_j(x;r,s,t)|\leq C(T_0,T_1,\phi_j)/(1+|x|^n),\quad \forall\; n\in \Bbb N,\;j\in \{1,2\},
\]
it holds that  
\[
\sup_{\stackrel{\scriptstyle T_0\leq s\leq t\leq T_1}{N\geq 16}}\left|\int_{\R^2}\!\int_{\R^2} \!\int_0^{Ns}\!\!\! \left(\mathfrak b^N(x,y;r,s,t)-\frac{\1_{[1,\infty)}(r)}{4\pi r}\right)\d r\phi_1(x;r,s,t)\phi_2(y;r,s,t)\d x\d y\right|<\infty.
\]
\end{cor}

\subsection{Identification of the limit}\label{sec:id}
Given $X_0\in \mathcal S'(\R^2)$, the additive stochastic heat equation is defined by
\begin{align}\label{def:ASHE_weak}
X_t(\phi)=X_0(\phi)+\int_0^t X_s\left(\frac{\Delta}{2}\phi\right)\d s+\int_0^t\int_{\R^2}\phi(x)W(\d x,\d r),
\end{align}
where $W(\d x,\d r)$ is a space-time white noise. For \eqref{def:ASHE_weak} and in what follows, $\Delta$ refers to the Laplacian (rather than the operator defined in \eqref{def:delta}). The solution is
\begin{align}\label{main:SPDE}
X_{t}(\phi)= X_0[Q_{t}(\phi)]+\int_0^{t} \int_{\R^2}Q_{t-r}\phi(x)W(\d x,\d r),
\end{align}
where $Q_t=\e^{t\Delta/2}$ is the transition semigroup of the two-dimensional standard Brownian motion and $W$ is a space-time white noise. See \cite[pp.339--343 in Chapter~5]{Walsh}.

Theorem~\ref{thm:covar1} and \eqref{main:SPDE} suggest that the limiting process of the rescaled Whittaker SDEs is the solution $X$ of an additive stochastic heat equation: if $X_0$ is independent of the space-time white noise such that$\{X_0(\phi)\}$ is a family of centered Gaussian variables with
\begin{align}\label{covar:X0}
\E[X_0(\phi_1)X_0(\phi_2)]=\frac{1}{2\pi }
\int_{\R^2}\int_{\R^2}\phi_1(y')\phi_2(y'')\big(-\ln|y'-y''|\big)\d y'\d y'',
\end{align}
then for $0<s\leq t<\infty$,
\begin{align}\label{match}
\lim_{N\to\infty}\left[\Cov[\zeta^N_s(\phi_1);\zeta^N_t(\phi_2)]-\mathfrak C_N\left(\int \phi_1\right)\left(\int\phi_2\right)\right]=\Cov[X_{s^{-1}}(\phi_1);X_{t^{-1}}(\phi_2)]
\end{align}
and $(X_{t^{-1}};t>0)$ is a centered Gaussian process. To be precise,  recall that \eqref{covar:X0} well defines $X_0$ as a centered Gaussian random field indexed by $\mathcal S_0(\R^2)=\{\phi\in\mathcal S(\R^2);{\textstyle \int}\phi=0\}$. See \cite{LSSW} and the references therein for a reproducing kernel approach of the construction of $X_0$. Alternatively, $X_0$ can be defined as the stationary solution of the additive stochastic heat equation: $X_0(\phi)=\int_{-\infty}^0\int_{\R^2}Q_{-r}\phi(x)W(\d x,\d r)$, which converges a.s. See Section~\ref{sec:stat}.

Given the conditionally positive definiteness of $(x,y)\mapsto -\ln|x-y|$, the relation in \eqref{covar:X0} cannot apply on the full space $\mathcal S(\R^2)$.  To avoid unnecessary technical issues from this restriction of domain, first we fix $\psi\in \mathcal S(\R^2)$ such that $\int \psi=1$ and define a re-centering operator $R=R_\psi:\mathcal S(\R^2)\to \mathcal S_0(\R^2)$ by $R\phi=\phi-(\int \phi)\psi$. Note that $R$ is a projection onto $\mathcal S_0(\R^2)$: $R^2\phi=R\phi$. Then with the Gaussian random field $X_0$ specified above, we modify the choice of the limiting process in \eqref{match} to the following $\mathcal S'(\R^2)$-valued continuous process:
\begin{align}\label{def:lim}
X_t(\phi)=X_0\left[Q_{t} R\phi\right]+\int_0^{t} \int_{\R^2}Q_{t-r}R\phi(x)W(\d x,\d r),\quad \phi\in \mathcal S(\R^2).
\end{align} 
Here, $X_0\left[Q_{t} R\phi\right]$, and hence, the process $X$ are well-defined since the Lebesgue measure is an invariant measure of $(Q_t)$. Note that for $\phi\in \mathcal S_0(\R^2)$, this $X_t(\phi)$ still satisfies \eqref{def:ASHE_weak}. (For example,  \cite[$5^\circ$ in  the proof of Theorem~2.1 on page 430]{Shiga} allows for a straightforward extension beyond one dimension to this case.) By \eqref{match}, it holds that for all $\phi_1,\phi_2\in \mathcal S(\R^2)$ and $0<s\leq t<\infty$,
\begin{align}
\begin{split}\label{eq:main_conv1}
\lim_{N\to\infty}\Cov[\zeta^N_s\circ R(\phi_1);\zeta^N_t\circ R(\phi_2)]&=\Cov[X_{s^{-1}}(\phi_1);X_{t^{-1}}(\phi_2)].
\end{split}
\end{align}

\begin{thm}\label{thm:second}
As $N\to\infty$, the sequence of laws of $(\zeta_t^N\circ R;t>0)$ converges weakly to the law of $(X_{t^{-1}};t>0)$ as probability measures on $D((0,\infty),\mathcal S'(\R^2))$, where $X$ is defined by \eqref{def:lim}.
\end{thm}

For the proof of this theorem, the convergence of finite-dimensional marginals follows readily from \eqref{eq:main_conv1}. To obtain the convergence at the process level, Mitoma's theorem \cite{Mitoma} requires the tightness of  $\zeta^N(\phi)$ for all fixed $\phi\in \mathcal S_0(\R^2)$. The proof of this property is the subject of the next section.

\section{Tightness of the rescaled Whittaker SDEs}\label{sec:tight}
Our goal in this section is to prove that the family of laws of the continuous Gaussian processes $\zeta^N(\phi)$ defined by \eqref{def:zetaN2} is tight, for a fixed $\phi\in \mathcal S_0(\R^2)$. It is enough to show that $(s,t)\mapsto  \E[|\zeta_s^N(\phi)-\zeta_t^N(\phi)|^{2}]$ are uniformly  H\"older continuous on compacts. The major argument appears in the proof of Lemma~\ref{lem:JN2,3}, where several exact identities are introduced to cancel new divergent terms. Throughout this section, we continue to use the notations in \eqref{def:Bkernel} for binomial probabilities.

Consider the explicit expressions of these metrics by applying the rescaling under consideration. Recall \eqref{def:zetaN1}, \eqref{obj2} and \eqref{def:zetaN2}. By \eqref{covar}, we have
\begin{align}
\E[|\zeta_s^N(\phi)-\zeta_t^N(\phi)|^{2}]
&=\mathcal I_{N}(s,t)-\mathcal J_{N}(s,t)-\mathcal K_{N}(s,t),\label{covar:dec}
\end{align}
where
\begin{align}
\hypertarget{IN}{\mathcal I_{N}}(s,t)&=\int_s^t \d r \int_{x\geq  -N^{1/2}\mathbf 1}\d x\phi(x)\int_{y\geq -N^{1/2}\mathbf 1}\d y\phi(y)\mathfrak b_{N}(x,y;r,t,t),\notag\\
\hypertarget{JN}{\mathcal J_{N}}(s,t)&=\int_0^s \d r \int_{x\geq  -N^{1/2}\mathbf 1}\d x\phi(x)\int_{y\geq -N^{1/2}\mathbf 1}\d y\phi(y)[\mathfrak b_{N}(x,y;r,s,t)-\mathfrak b_{N}(x,y;r,s,s)],\notag\\
\hypertarget{KN}{\mathcal K_{N}}(s,t)&=\int_0^s \d r \int_{x\geq  -N^{1/2}\mathbf 1}\d x\phi(x)\int_{y\geq -N^{1/2}\mathbf 1}\d y\phi(y)[\mathfrak b_{N}(x,y;r,s,t)-\mathfrak b_{N}(x,y;r,t,t)]\notag.
\end{align}
Then the job is to show that each of the three terms on the right-hand side of \eqref{covar:dec} satisfies the following property for some $\alpha\in (0,1]$:
\begin{align}\label{IJK:bdd}
\begin{split}
&\sup_{N\geq 16}\sup_{s,t:T_0\leq s< t\leq T_1}|\mathcal L_N(s,t)|/(t-s)^\alpha<\infty,\quad\;\forall\; 0<T_0<1<T_1<\infty.
\end{split}
\end{align}
The results are presented as Propositions~\ref{prop:IN}, \ref{prop:JN}, and \ref{prop:KN} below. The use of uniform H\"older continuity is only for $\hyperlink{IN}{\mathcal I_{N}}$;  the other two functions are uniformly Lipschitz. 

\begin{nota}\label{nota:j}
Write $j_\ast $ for the coordinate in $\{1,2\}$ different from $j\in\{1,2\}$.  \hfill $\blacksquare$
\end{nota}

\begin{prop}\label{prop:IN}
\eqref{IJK:bdd} is satisfied by $\mathcal L_N=\hyperlink{IN}{\mathcal I_{N}}$ for $\alpha=1/2$. 
\end{prop}
\begin{proof}
We bound the integrand of $\hyperlink{IN}{\mathcal I_{N}}(s,t)$, for $r\in [s,t]$, according to $r/t\geq 1-(1/2)N^{-1/2}$ and the complementary case. The point here is that since $r$ is bounded away from zero, we do not need to deal with divergent re-centering constants as in Section~\ref{sec:conv}. 
\smallskip

\noindent \hypertarget{IN1}{{\bf Step 1.}}
First, we assume the primary condition \eqref{prim cond} and
\begin{align}\label{cond:gauss-bddbdd0}
r\in [s,t]\quad \mbox{such that}\quad r/t\geq 1-(1/2)N^{-1/2}.
\end{align}

Under \eqref{cond:gauss-bddbdd0}, we obtain from \eqref{SM:poisson} and \eqref{bin:comp} that
\begin{align}\label{IN-1}
\begin{split}
&\Bigg|\mathbf P\left(S_{m}\left(\frac{r}{t}\right)=S'_{m'}\left(\frac{r}{t}\right)\right)-\mathbf P\left(V\left(m\left(1-\frac{r}{t}\right)\right)=V'\left(m'\left(1-\frac{r}{t}\right)\right)+m-m'\right)\Bigg|
\leq \frac{1}{2N^{1/2}}
\end{split}
\end{align}
for $m=M(x_j,t)$, $m'=M(y_j,t)$, and $j\in \{1,2\}$. To bound the foregoing Poisson probabilities, notice that for all $x_j,y_j\in \R$ and $t\geq T_0$,  
\begin{align}\label{IN-2}
|M(x_j,t)-M(y_j,t)|\geq T_0N^{1/2}|x_j-y_j|-1.
\end{align}
Then it follows from the elementary inequalities $\mathbf P(V-V'\geq m)\leq \e^{-N^{-1/4}m}\mathbf E[\e^{N^{-1/4}(V-V')}]$ and $\e^{v}-1\leq v+v^2$ for all $v\in [-1,1]$ that the next two inequalities hold:  
\begin{align}
 &\quad N^{1/2}\mathbf P\left(V\left(M(x_j,t)\left(1-\frac{r}{t}\right)\right)=V'\left(M(y_j,t)\left(1-\frac{r}{t}\right)\right)+M(x_j,t)-M(y_j,t)\right)\notag\\
 &\leq \1_{\{M(x_j,t)-M(y_j,t)\geq 0\}}N^{1/2}\exp\Big\{-N^{-1/4}[M(x_j,t)-M(y_j,t)]^+\Big\}
\notag\\
 &\quad \times \exp\Big\{M(x_j,t)\Big(1-\frac{r}{t}\Big)(\e^{N^{-1/4}}-1)+M(y_j,t)\Big(1-\frac{r}{t}\Big)(\e^{-N^{-1/4}}-1)\Big\}
\notag\\
&\quad +\1_{\{M(x_j,t)-M(y_j,t)< 0\}}N^{1/2}\exp\Big\{-N^{-1/4}[M(x_j,t)-M(y_j,t)]^-\Big\}\notag\\
&\quad\quad  \times \exp\Big\{M(y_j,t)\Big(1-\frac{r}{t}\Big)(\e^{N^{-1/4}}-1)+M(x_j,t)\Big(1-\frac{r}{t}\Big)(\e^{-N^{-1/4}}-1)\Big\}\notag\\
&\leq N^{1/2}\exp\Big\{-N^{-1/4}|M(x_j,t)-M(y_j,t)|\Big\}\notag\\
&\quad \times \exp\Big\{\Big(1-\frac{r}{t}\Big)N^{-1/4}|M(x_j,t)-M(y_j,t)|+\Big(1-\frac{r}{t}\Big)N^{-1/2}[M(x_j,t)+M(y_j,t)]\Big\}\notag\\
&\leq N^{1/2}\exp\Big\{-\frac{1}{2}T_0N^{1/4}|x_j-y_j|+C(T_1)\Big\},\notag
\end{align}
where the last inequality uses the primary condition \eqref{prim cond}, \eqref{cond:gauss-bddbdd0} and \eqref{IN-2}. Combining \eqref{IN-1} and the last inequality proves that under the primary condition \eqref{prim cond} and \eqref{cond:gauss-bddbdd0},
\begin{align}\label{IN-5}
\mathfrak b_{N,j}(x_j,y_j;r,t,t)
\leq 1+C(T_1)N^{1/2}\e^{-\frac{1}{2}T_0N^{1/4}|x_j-y_j|}.
\end{align}

\noindent \hypertarget{IN2}{{\bf Step 2.}}
We assume the primary condition \eqref{prim cond} and the complementary case of \eqref{cond:gauss-bddbdd0}:
\begin{align}\label{cond:gauss-bddbdd}
r\in [s,t]\quad \mbox{such that}\quad 1-r/t> (1/2)N^{-1/2}.
\end{align}
Then $C(T_0,T_1)\leq  N^{1/2}\sigma_j(r;N)^2$ for all $j\in \{1,2\}$ (recall \eqref{def:mujsigmaj0}). Furthermore, if 
\begin{align}\label{sec cond-jj}
|x_j-y_j|\geq (4/T_0)N^{-1/2}\quad\mbox{for \emph{fixed} $j\in \{1,2\}$,}
\end{align}
then $|M(x_j,t)-M(y_j,t)|\geq C(T_0,T_1)N^{1/2}|x_j-y_j|$ by \eqref{IN-2}. Hence, under the primary condition \eqref{prim cond}, \eqref{cond:gauss-bddbdd} and \eqref{sec cond-jj}, Lemma~\ref{lem:normal1} gives
\begin{align}
\begin{split}
\mathfrak b_{N,j}(x_j,y_j;r,t,t)&\leq \mathfrak g\big(\sigma_j(r;N)^2;C(T_0,T_1)(x_j-y_j)\big)+C'(T_0,T_1).\label{IN-4-2}
\end{split}
\end{align}

\noindent \hypertarget{IN3}{{\bf Step 3.}}
Up to this point, we have a bound of $|\hyperlink{IN}{\mathcal I_{N}}(s,t)|$ given by the sum of the following six terms, up to a multiplicative constant $C(T_0,T_1,\phi)\in (0,\infty)$:
\begin{align*}
 \displaystyle 
 \tau_1&=\int_s^t \d r\int_{\R^2}\d x|\phi(x)|\int_{\R^2}\d y |\phi(y)|\1_{\{1-r/t\leq (1/2)N^{-1/2}\}}\prod_{j=1}^2\big[1+N^{1/2}\e^{-\frac{1}{2}T_0N^{1/4}|x_j-y_j|}\big],\\
\displaystyle
 \tau_2&=\int_s^t \d r\int_{\R^2}\d x|\phi(x)|\int_{\R^2}\d y |\phi(y)|\prod_{j=1}^2 \big[ \mathfrak g\big(\sigma_j(r;N)^2;C(T_0,T_1)(x_j-y_j)\big)+1\big],\quad \\
\displaystyle\tau_3&= \sum_{j=1}^2\int_s^t \d r\int_{\R^2}\d x|\phi(x)|\int_{\R^2}\d y |\phi(y)|\\
&\hspace{.5cm}\big[ \mathfrak g\big(\sigma_j(r;N)^2;C(T_0,T_1)(x_j-y_j)\big)+1\big]\times N^{1/2}\1_{\{|x_{j_\ast}-y_{j_\ast}|< (4/T_0)N^{-1/2}\}},\quad \\
 \displaystyle  \tau_4&=\int_s^t \d r\Bigg(\prod_{j=1}^2N^{-1/2}\Bigg)\cdot \Bigg(\prod_{j=1}^2N^{1/2}\Bigg),\hspace{3cm}\\
\displaystyle \tau_5&=
\int_{s}^t \d r\int_{([-\frac{1}{2}N^{\eta},\frac{1}{2}N^{\eta}]^2)^\complement}\d x|\phi(x)|\int_{\R^2}\d y |\phi(y)|\Bigg(\prod_{j=1}^2N^{1/2}\Bigg),\hspace{2cm} \\
\displaystyle \tau_6&= \int_{s}^t \d r\int_{\R^2}\d x |\phi(x)|\int_{([-\frac{1}{2}N^{\eta},\frac{1}{2}N^{\eta}]^2)^\complement}\d y|\phi(y)|\Bigg(\prod_{j=1}^2N^{1/2}\Bigg),\hspace{2cm}
\end{align*}
where the definition of $\tau_3$ uses Notation~\ref{nota:j}. Let us explain how these six terms arise from bounding $  |\hyperlink{IN}{\mathcal I_{N}}(s,t)|$. First, $\tau_1$ follows from \eqref{IN-5}.  We get $\tau_2$ from \eqref{IN-4-2} by enforcing \eqref{sec cond-jj} for all $j\in\{1,2\}$, whereas $\tau_3$ follows from \eqref{IN-4-2} under \eqref{sec cond-jj} for exactly one $j\in \{1,2\}$. Then we need $\tau_4$ when \eqref{sec cond-jj} fails for all $j\in \{1,2\}$. Now that $\tau_1,\tau_2,\tau_3,\tau_4$ are obtained by assuming the primary condition, $\tau_5,\tau_6$ handle the failure of the primary condition. 
 
We are ready to prove the required uniform H\"older continuity. It is immediate that
\begin{align}\label{IN-final2}
\sum_{j=2}^6\tau_j\leq C(T_0,T_1,\phi)|t-s|,
\end{align}
where $\tau_2,\tau_3,\tau_5,\tau_6$ are bounded by using the fast decay property of $\phi$ or the property that $\int_{\R}\d x_j\g(\sigma^2;x_j)=1$ for all $\sigma^2>0$. The argument in Step~\hyperlink{IN1}{1} is insufficient to yield such uniform Lipschitz continuity of $\tau_1$, but we can still prove its H\"older continuity as follows. First, by the Cauchy--Schwarz inequality with respect to $\int_s^t\d r$ and the bound $\int_s^t \d r\1_{\{1-r/t\leq (1/2)N^{-1/2}\}}\leq C(T_0,T_1)N^{-1/2}$,
\begin{align*}
&\quad \int_s^t \d r\int_{\R^2} \d x|\phi(x)|\int_{\R^2}\d y|\phi(y)|\1_{\{1-r/t\leq (1/2)N^{-1/2}\}}N^{1/2}\e^{-\frac{1}{2}T_0N^{1/4}|x_j-y_j|}\\
&\leq \left(C(T_0,T_1)N^{-1/2}\left(N^{1/2}\int_{\R^2}\d x|\phi(x)|\int_{\R^2}\d y|\phi(y)|\e^{-\frac{1}{2}T_0N^{1/4}|x_j-y_j|}\right)^{2}\right)^{1/2}(t-s)^{1/2}\\
&\leq C(T_0,T_1,\phi)(t-s)^{1/2},
\end{align*}
where the $\less$-inequality follows since $\sup_NN^{1/4}\int_{\R}\d x_j\e^{-\frac{1}{2}T_0N^{1/4}|x_j|}<\infty$. Hence, $\tau_1\less C(T_0,T_1,\phi)(t-s)^{1/2}$. This H\"older continuity and \eqref{IN-final2} are enough for the proposition. 
\end{proof}

The main theme of this section is to bound $\hyperlink{JN}{\mathcal J_{N}}$. We start with an interpolation to represent this term: 
\begin{align}
\begin{split}
\hyperlink{JN}{\mathcal J_{N}}(s,t)
&=\int_s^t\d v\int_0^s \d r  \int_{x\geq  -N^{1/2}\mathbf 1}\d x\phi(x)\frac{\partial}{\partial v}\int_{y\geq -N^{1/2}\mathbf 1}\d y\phi(y)
\mathfrak b_{N}(x,y;r,s,v)
.\label{interpolate}
\end{split}
\end{align}
For the foregoing derivative, we change variables with $N^{1/2}y'=vN^{1/2}(y+N^{1/2}\mathbf 1)$:
\begin{align}
&\quad\; \frac{\partial }{\partial v}\int_{y\geq -N^{1/2}\mathbf 1}\d y\phi(y)  \prod_{j=1}^2 N^{1/2}\mathbf P\left(S_{M(x_j,s)}\left(\frac{r}{s}\right)=S'_{M(y_j,v)}\left(\frac{r}{v}\right)\right),\notag\\
&= \frac{\partial }{\partial v}\frac{1}{v^2}\int_{y'\geq \mathbf 0}\d y'\phi\left(\frac{y'}{v}-N^{1/2}\mathbf 1\right)
 \prod_{j=1}^2 N^{1/2}\mathbf P\left(S_{M(x_j,s)}\left(\frac{r}{s}\right)=S'_{\lfloor N^{1/2}y'_j\rfloor}\left(\frac{r}{v}\right)\right).\label{interpolate1}
\end{align}
To proceed, we differential under the integral sign and then change the variable $y'$ back to $y$. Hence, by \eqref{interpolate} and \eqref{interpolate1}, we get
\begin{align}
\hyperlink{JN1}{\mathcal J_{N}}(s,t)&=-2\hypertarget{JN1}{\mathcal J_{N,1}}(s,t)-\hypertarget{JN2}{\mathcal J_{N,2}}(s,t)+\hypertarget{JN3}{\mathcal J_{N,3}}(s,t).\label{main:JN}
\end{align}
Here, since $\d y'=\d yv^2$, 
\begin{align*}
\left\{
\begin{array}{ll}
\displaystyle \mathcal J_{N,k}(s,t)= \int_s^t\d v \int_0^s \d r\int_{x\geq -N^{1/2}\mathbf 1}\d x\phi(x) \int_{y\geq -N^{1/2}\mathbf 1}\d yv^2 \times 
\mathfrak a_{N,k} \\
\vspace{-.2cm}\\
\displaystyle \quad \quad \mbox{for}\quad \mathfrak a_{N,1}=\frac{\phi(y)}{v^3}\mathfrak b_{N}(x,y;r,s,v), \quad \mathfrak a_{N,2}=\frac{v(y+N^{1/2}\mathbf 1)\cdot \nabla\phi(y)}{v^4}\mathfrak b_{N}(x,y;r,s,v),\\
\vspace{-.2cm}\\
\displaystyle 
 \quad\quad \mathfrak a_{N,3}=\frac{\phi(y)}{v^2} \frac{\partial}{\partial v}\prod_{j=1}^2 N^{1/2}\mathbf P\left(S_{M(x_j,s)}\left(\frac{r}{s}\right)=S'_{m'_j}\left(\frac{r}{v}\right)\right)\Big|_{m_1'=M(y_1,v),m_2'=M(y_2,v)}.
\end{array}
\right.
\end{align*}
We handle $\hyperlink{JN1}{\mathcal J_{N,1}}$ and $-\hyperlink{JN2}{\mathcal J_{N,2}}+\hyperlink{JN3}{\mathcal J_{N,3}}$ separately. 

\begin{lem}\label{lem:JN1}
\eqref{IJK:bdd} is satisfied by $\mathcal L_N=\hyperlink{JN1}{\mathcal J_{N,1}}$  for $\alpha=1$. 
\end{lem}
\begin{proof}
We undo the change of variables below \eqref{interpolate} to rewrite  $\hyperlink{JN1}{\mathcal J_{N,1}}(s,t)$ as
\begin{align*}
&\int_s^t\frac{\d v}{v} \int_{x\geq -N^{1/2}\mathbf 1}\d x\phi(x)\int_{y\geq -N^{1/2}\mathbf 1}\d y\phi(y)\int_0^{Ns} 
\d r\Big(\mathfrak b^N(x,y;r,s,v)
-\frac{\1_{[1,\infty)}(r)}{4\pi r}\Big)\\
&\quad +\int_s^t \frac{\d v}{v}\int_{x\geq -N^{1/2}\mathbf 1}\d x\phi(x)\int_{y\geq -N^{1/2}\mathbf 1}\d y\phi(y)\frac{\ln (Ns)}{4\pi}\1_{\{Ns\geq 1\}}.
\end{align*}
By Corollary~\ref{cor:uniformcovar}, the first term in the foregoing equality is bounded by $C(T_0,T_1,\phi)|t-s|$. The second term can be bounded in the same way since the assumption $\phi\in \mathcal S_0(\R^2)$ enables the cancellation of $\frac{\ln (Ns)}{4\pi}$ up to an error term that can be subdued by using the fast decay of $\phi$. We have proved the proposition. 
\end{proof}

For the remaining terms in \eqref{main:JN} for $\hyperlink{JN}{\mathcal J_{N}}$, we first state some elementary results.

\begin{lem}\label{lem:shift}
Let $F:\R\to \R$ be bounded, $p\in (0,1)$, and $m\in \Bbb Z_+$. \smallskip

\noindent \hypertarget{shift1}{{\rm (1$^\circ$)}} The independent sums $S_m=S_m(p)$'s satisfy
\begin{align}
&p\big(\mathbf E\left[F(S_m)\right]-\mathbf E\left[F(S_m+1)\right]\big)=\mathbf E\left[F(S_m)\right]-\mathbf E\left[F(S_{m+1})\right]
,\label{bin:obs0}\\
&\mathbf E[F(S_m+2)]
=\frac{1}{p^2}\mathbf E[F(S_{m+2})]-\frac{2(1-p)}{p^2}\mathbf E[F(S_{m+1})]+\frac{(1-p)^2}{p^2}\mathbf E[F(S_{m})].\label{bin:obs2}
\end{align}

\noindent  \hypertarget{shift2}{{\rm (2$^\circ$)}} {(Binomial integration by parts).} 
$\mathbf E\left[S_mF(S_m)\right]=\mathbf E[S_m]\mathbf E\left[F(S_{m-1}+1)\right]$.\smallskip

\noindent  \hypertarget{shift3}{{\rm (3$^\circ$)}} For all $\phi\in \mathcal S(\R^2)$, $L\in [-\infty,\infty]$, $v\in (0,\infty)$, $\ell\in \Bbb Z$ and $j\in \{1,2\}$, it holds that
\[
\int_{L}^\infty \d x_j\phi(x)F\big(M(x_j,v)+\ell\big)
=\int_{ L+ \ell/(vN^{1/2})}^\infty \d x_j\phi\big(x- \ell e_j/(vN^{1/2})\big)F\big(M(x_j,v)\big), 
\]
where $\{e_1,e_2\}$ is the standard basis of $\R^2$. 
\end{lem} 
\begin{proof}
Conditioning $S_{m+1}$ on $S_m$ yields \eqref{bin:obs0}. To get \eqref{bin:obs2}, note that  
\begin{align}
\mathbf E[F(S_m+1)]
=\frac{1}{p}\mathbf E[F(S_{m+1})]-\frac{1-p}{p}\mathbf E[F(S_m)]\label{bin:obs1}
\end{align}
by rearranging \eqref{bin:obs0}, and then iterate \eqref{bin:obs1}. Next, the identity in \hyperlink{shift2}{\rm (2$^\circ$)} is equivalent to $\sum_{j=1}^m {m\choose j}p^j(1-p)^{m-j}jF(j)=mp\sum_{j=0}^{m-1} {m-1\choose j}p^j(1-p)^{m-1-j}F(j+1)$. For  \hyperlink{shift3}{\rm (3$^\circ$)}, observe that $M(x_j,a)+\ell=M(x_j+ \ell/(aN^{1/2}),a)$ by \eqref{def:Morigin}, and then change variables.
\end{proof}
 
\begin{lem}\label{lem:JN2,3}
\eqref{IJK:bdd} is satisfied by $\mathcal L_N=-\hyperlink{JN2}{\mathcal J_{N,2}}+\hyperlink{JN3}{\mathcal J_{N,3}}$  for $\alpha=1$. 
\end{lem}
\begin{proof}
We divide the proof into a few steps. Step~\hyperlink{J1}{1} is for $\hyperlink{JN2}{\mathcal J_{N,2}}$, whereas we need Steps~\hyperlink{J2}{2}, \hyperlink{J21}{2-1}--\hyperlink{J22}{2-2} to handle $\hyperlink{JN3}{\mathcal J_{N,3}}$. A summary is given in Step \hyperlink{J3}{3} to complete the proof. 
\smallskip

\noindent {\bf Step \hypertarget{J1}{1}.}
According to $(y+N^{1/2}\mathbf 1)\cdot \nabla \phi(y)=y\cdot \nabla \phi(y)+N^{1/2}\div \phi(y)$, we can write 
\begin{align}
\hyperlink{JN2}{\mathcal J_{N,2}}(s,t)
&=\hypertarget{JN21}{\mathcal J_{N,2,1}}(s,t)+\hypertarget{JN22}{\mathcal J_{N,2,2}}(s,t). \label{main1}
\end{align}
Here, 
\begin{align}
\left\{
\begin{array}{ll}
\displaystyle \hyperlink{JN21}{\mathcal J_{N,2,j}}(s,t)= \int_s^t\frac{\d v}{v}\int_0^s \d r \int_{x\geq -N^{1/2}\mathbf 1}\d x\phi(x)\int_{y\geq -N^{1/2}\mathbf 1}\d y \mathfrak a_{N,2,j}\notag\\
\vspace{-.2cm}\\
\displaystyle \quad \mbox{\rm for}\quad \mathfrak a_{N,2,1}=y\cdot \nabla\phi(y)\mathfrak b_{N}(x,y;r,s,v),\quad \mathfrak a_{N,2,2}=N^{1/2}\div \phi(y)\mathfrak b_{N}(x,y;r,s,v).\notag
\end{array}
\right.
\end{align}
Since $y\mapsto y\cdot \nabla \phi(y)\in \mathcal S_0(\R^2)$ by integration by parts, the proof in Lemma~\ref{lem:JN1} shows that \eqref{IJK:bdd} for $\mathcal L_N=\hyperlink{JN21}{\mathcal J_{N,2,1}}$ and $\alpha=1$ holds. We handle $\hyperlink{JN22}{\mathcal J_{N,2,2}}(s,t)$ in Step~\hyperlink{J3}{3}.

In the rest of this proof, we write $\mathcal R_{N,j}=\mathcal R_{N,j}(x_{j_\ast},y_{j_\ast};r,s,v)$ for a function such that 
\begin{align*}
&\sup_{N\in \Bbb N}\sup_{s,v\in [T_0,T_1]}\Bigg|\int_0^s
\d r \int_{-N^{1/2}}^\infty \!\!\d x_{j_\ast}\int_{-N^{1/2}}^\infty \!\!\d y_{j_\ast} \mathcal R_{N,j}(x_{j_\ast},y_{j_\ast};r,s,v) \mathfrak b_{N,j_\ast}(x_{j_\ast},y_{j_\ast};r,s,v)\Bigg|<\infty.
\end{align*}
(Recall Notation~\ref{nota:j} for the index $j_\ast$.) The ``remainder'' $\mathcal R_{N,j}$ may change from term to term unless otherwise specified.\smallskip

\noindent {\bf Step \hypertarget{J2}{2}.}
We consider $\hyperlink{JN3}{\mathcal J_{N,3}}$ in Steps~\hyperlink{J2}{2}--\hyperlink{J2-1-2}{4}. First, we compute the derivative in its definition. By Lemma~\ref{lem:bin} (1$^\circ$) and then Lemma~\ref{lem:shift} (2$^\circ$) with $S_m(p)=S_{{M(x_j,s)} }\left(\frac{r}{s}\right)$, 
\begin{align}
&\quad\; \frac{\partial}{\partial v}N^{1/2}\mathbf P\left(S_{{M(x_j,s)} }\left(\frac{r}{s}\right)=S'_{m_j'}\left(\frac{r}{v}\right)\right)\Big|_{m'_j=M(y_j,v)}
\notag\\
&=\frac{N^{1/2}}{v}\mathbf E\left[\left(S_{{M(x_j,s)} }\left(\frac{r}{s}\right)+1\right)\1_{\big\{S_{{M(x_j,s)} }\left(\frac{r}{s}\right)+1=S'_{M(y_j,v)}\left(\frac{r}{v}\right)\big \}}\right]\notag\\
&\quad -\frac{N^{1/2}}{v}\mathbf E\left[
S_{{M(x_j,s)} }\left(\frac{r}{s}\right)\1_{\big\{
S_{{M(x_j,s)} }\left(\frac{r}{s}\right)=
S'_{M(y_j,v)}\left(\frac{r}{v}\right)\big \}}\right]\notag\\
&=\frac{N^{1/2}M(x_j,s)r}{vs}\mathbf P\left(S_{{M(x_j,s)-1} }\left(\frac{r}{s}\right)+2=S'_{M(y_j,v)}\left(\frac{r}{v}\right)\right)\notag\\
&\quad -\frac{N^{1/2}M(x_j,s)r}{vs}\mathbf P\left(S_{{M(x_j,s)-1} }\left(\frac{r}{s}\right)+1=S'_{M(y_j,v)}\left(\frac{r}{v}\right)\right)\notag\\
&\quad +\frac{N^{1/2}}{v}\mathbf P\left(S_{{M(x_j,s)} }\left(\frac{r}{s}\right)+1=S'_{M(y_j,v)}\left(\frac{r}{v}\right)\right)\notag\\
&=\hypertarget{JN3j1}{\mathfrak a_{N,3,j,1}}(x_j,y_j;r,s,v)-\hypertarget{JN3j2}{\mathfrak a_{N,3,j,2}}(x_j,y_j;r,s,v),\label{JN3-partialder}
\end{align}
where the last equality applies \eqref{bin:obs0} for $F(n)=\mathbf P(S_{{M(x_j,s)-1} }(\frac{r}{s})+2=n)$ and $S_m(p)=S'_{M(y_j,v)}(\frac{r}{v})$ so that
\begin{align*}
\hypertarget{JN3j1}{\mathfrak a_{N,3,j,1}}(x_j,y_j;r,s,v)&=N\times \frac{M(x_j,s)}{Ns}\times N^{1/2}\mathbf P\left(S_{{M(x_j,s)-1} }\left(\frac{r}{s}\right)+2=S'_{M(y_j,v)}\left(\frac{r}{v}\right)\right)\\
&\quad -N\times \frac{M(x_j,s)}{Ns}\times N^{1/2}\mathbf P\left(S_{{M(x_j,s)-1} }\left(\frac{r}{s}\right)+2=S'_{M(y_j,v)+1}\left(\frac{r}{v}\right)\right),\\
\hypertarget{JN3j2}{\mathfrak a_{N,3,j,2}}(x_j,y_j;r,s,v)&=\frac{1}{v}\times N^{1/2}\mathbf P\left(S_{{M(x_j,s)} }\left(\frac{r}{s}\right)+1=S'_{M(y_j,v)}\left(\frac{r}{v}\right)\right).
\end{align*}
Note that $\mathfrak a_{N,3,j,k}$ are written out this way to make clear the property that $\lim_NM(x_j,s)/(Ns)= 1$. By the definition of $\hyperlink{JN3}{\mathcal J_{N,3}}$ and \eqref{JN3-partialder}, we have shown that 
\begin{align}\label{JN3-intintint}
\begin{split}
\hyperlink{JN3}{\mathcal J_{N,3}(s,t)}&=\sum_{j=1}^2 \int_s^t\d v \int_0^s \d r\int_{ -N^{1/2}}^\infty \d x_{j_\ast}\int_{ -N^{1/2}}^\infty \d y_{j_\ast} \\
&\hspace{-1.5cm}\left(\sum_{k=1}^2\int_{-N^{1/2}}^\infty \d x_{j}\phi(x)\int_{-N^{1/2}}^\infty \d y_j\phi(y)\mathfrak a_{N,3,j,k} (x_j,y_j;r,s,v)\right) \mathfrak b_{N,j_\ast}(x_{j_\ast},y_{j_\ast};r,s,v).
 \end{split}
\end{align}

To bound $\hyperlink{JN3}{\mathcal J_{N,3}(s,t)}$ according to \eqref{JN3-intintint}, we consider the following two sets of integrals separately in the next two steps:
\begin{align}\label{JN3-set}
\left\{\int_{-N^{1/2}}^\infty \d x_{j}\phi(x)\int_{-N^{1/2}}^\infty \d y_j\phi(y)\mathfrak a_{N,3,j,k} (x_j,y_j;r,s,v);   j=1,2\right\},\quad k=1,2.
\end{align}

\noindent {\bf Step \hypertarget{J2-1}{3}.} For the set in \eqref{JN3-intintint} with $k=2$, applying Corollary~\ref{cor:uniformcovar} as in the proof of Lemma~\ref{lem:JN1} yields that
\begin{align}
\begin{split}
\int_{-N^{1/2}}^\infty \d x_j\phi(x)  \int_{-N^{1/2}}^\infty \d y_j\phi(y)\hyperlink{JN3j2}{\mathfrak a_{N,3,j,2}}
(x_j,y_j;r,s,v)=\mathcal R_{N,j},\quad \forall\;j\in \{1,2\}.\label{der:analyze}
\end{split}
\end{align}

In more detail for \eqref{der:analyze}, the binomial probabilities in $\hyperlink{JN3j2}{\mathfrak a_{N,3,j,2}}$ differ slightly from those in Corollary~\ref{cor:uniformcovar} since those in $\hyperlink{JN3j2}{\mathfrak a_{N,3,j,2}}$ are for events of the form $\{S_m+1=S'_{m'}\}$ rather than $\{S_m=S'_{m'}\}$. However, the proof of the corollary  relies on the bounds obtained in Sections~\ref{sec:poisson} and~\ref{sec:normal}, and these bounds carry over to the present case after slight modifications: See the algebra of Step \hyperlink{P2-3-STEP3}{3} of the proof of Proposition~\ref{prop:poisson2} up to \eqref{PO2}. The other changes are straightforward.\smallskip

\noindent {\bf Step \hypertarget{J2-2}{4}.}
In this step, we consider the set of integrals in \eqref{JN3-set} with $k=1$. In contrast to $\hyperlink{JN3j2}{\mathfrak a_{N,3,j,2}}$, each of the two terms defining $\hyperlink{JN3j1}{\mathfrak a_{N,3,j,1}}$ as a difference is of a larger order in $N$. We aim to let $\hyperlink{JN3j1}{\mathfrak a_{N,3,j,1}}$ (after appropriate integrations according to \eqref{JN3-intintint}) and $\hyperlink{JN22}{\mathcal J_{N,2,2}}$ (still unsettled) cancel each other. Then the task is to get an estimate of $\hyperlink{JN3j1}{\mathfrak a_{N,3,j,1}}$ that can be used to match $\hyperlink{JN22}{\mathcal J_{N,2,2}}$.

We begin by simplifying the integrals in \eqref{JN3-set} with $k=1$, in view of the property that there are additional $\pm1$ in the numbers of summands in the random sums defining $\hyperlink{JN3j1}{\mathfrak a_{N,3,j,1}}$. These unwanted integers can be removed  by using Lemma~\ref{lem:shift} (3$^\circ$) to translate variables:
\begin{align}
&\quad\, \int_{-N^{1/2}}^\infty \d x_j\phi(x)\int_{-N^{1/2}}^\infty \d y_j\phi(y)\hyperlink{JN3j1}{\mathfrak a_{N,3,j,1}}
(x_j,y_j;r,s,v)\notag\\
\begin{split}
&=N\int_{-N^{1/2}-\frac{1}{sN^{1/2}}}^\infty \d x_j\frac{M(x_j,s)+1}{Ns}\times \phi\left(x+\frac{e_j}{sN^{1/2}}\right)\\
&\quad \times\int_{-N^{1/2}}^\infty \d y_j\phi(y)\times N^{1/2}\mathbf P\left(S_{{M(x_j,s)} }\left(\frac{r}{s}\right)+2=S'_{M(y_j,v)}\left(\frac{r}{v}\right)\right)\\
&\quad -N\int_{-N^{1/2}-\frac{1}{sN^{1/2}}}^\infty \d x_j\frac{M(x_j,s)+1}{Ns}\times \phi\left(x+\frac{e_j}{sN^{1/2}}\right)\\
&\quad \times\int_{-N^{1/2}+\frac{1}{vN^{1/2}}}^\infty \!\!\!\d y_j \phi\left(y-\frac{e_j}{vN^{1/2}}\right)\times N^{1/2} \mathbf P\left(S_{{M(x_j,s)} }\left(\frac{r}{s}\right)+2=S'_{M(y_j,v)}\left(\frac{r}{v}\right)\right)\label{RN32-0}
\end{split}\\
\begin{split}
&=\frac{N^{1/2}}{v}\int_{-N^{1/2}-\frac{1}{sN^{1/2}}}^\infty \d x_j\frac{M(x_j,s)+1}{Ns}\times \phi\left(x+\frac{e_j}{sN^{1/2}}\right)\\
&\quad\times  \int_{-N^{1/2}}^\infty \d y_j\left[\phi(y)-\phi\left(y-\frac{e_j}{vN^{1/2}}\right)\right] vN^{1/2}\\
&\quad  \times  N^{1/2}\mathbf P\left(S_{{M(x_j,s)} }\left(\frac{r}{s}\right)+2=S'_{M(y_j,v)}\left(\frac{r}{v}\right)\right)+ \mathcal R_{N,j}.\label{RN32}
\end{split}
\end{align}

Let us explain \eqref{RN32}. The two terms defining the right-hand side of \eqref{RN32-0} as a difference are not the same only for the integrations with respect to $y_j$, and they share the same binomial probability. So we extract a discrete partial derivative $[\phi(y)-\phi(y-\frac{e_j}{vN^{1/2}})]vN^{1/2}$ from this difference, as shown in \eqref{RN32}. Additionally, in \eqref{RN32}, $\mathcal R_{N,j}$ arises from taking the difference of the ranges of $y_j$ in \eqref{RN32-0} and using  the fast decay property of $\phi$.

In the rest of Step~\hyperlink{J2-2}{4}, we show that the first term in \eqref{RN32} can be used to cancel $\hyperlink{JN22}{\mathcal J_{N,2,2}}$ left unsettled in Step~\hyperlink{J1}{1}. The plan is to reduce this term  in \eqref{RN32} to an integral showing only binomial probabilities of the form $\mathbf P(S_m=S'_{m'})$, now that these probabilities enter $\hyperlink{JN22}{\mathcal J_{N,2,2}}$.  To this end, we apply the higher-order expansion \eqref{bin:obs2}, and then we remove $+2$ and $+1$ in the numbers of summands on the right-hand side of \eqref{bin:obs2}  by changing variables according to Lemma~\ref{lem:shift} (3$^\circ$).  (The latter is similar to how we obtain \eqref{RN32-0}.)

In Steps~\hyperlink{J2-1-1}{4-1} and \hyperlink{J2-1-2}{4-2} below, we consider \eqref{RN32} for $1>r/s\geq 1/2$ and $r/s<1/2$ separately.
\smallskip

\noindent {\bf Step \hypertarget{J2-1-1}{4-1}.} We restrict our attention to $r$ such that $1>r/s\geq 1/2$. This condition on $r$ is used to bound away from zero the denominators $(r/s)^2$ in $\mathfrak a_{N,3,j,1,k}$ defined below.

To lighten notation for the application of \eqref{RN32}, we write
\begin{align}\label{def:tildeF-end}
\hypertarget{tildeF}{\widetilde{F}}(n)&=\int_{-N^{1/2}}^\infty \d y_j\left[\phi(y)-\phi\left(y-\frac{e_j}{vN^{1/2}}\right)\right] vN^{1/2} \times  N^{1/2}\mathbf P\left(n=S'_{M(y_j,v)}\left(\frac{r}{v}\right)\right).
\end{align}
Second, we apply the method outlined before Step \hyperlink{J2-1-1}{4-1}, with $p=r/s$, $m=M(x_j,s)$, and the notation $\mathfrak a_{N,3,j,1,k}=\mathfrak a_{N,3,j,1,k}(x,y_{j_{\ast}};r,s,v)$ defined by 
\begin{align*}
\hypertarget{JN2-2-1}{\mathfrak a_{N,3,j,1,1}}(x,y_{j_{\ast}};r,s,v)&=\frac{M(x_j,s)-1}{Ns} \times \phi\left(x-\frac{e_j}{sN^{1/2}}\right) \frac{1}{(r/s)^2}\mathbf E\left[\widetilde{F}\left(S_{M(x_j,s)}\left(\frac{r}{s}\right)\right)\right],\\
\hypertarget{JN2-2-2}{\mathfrak a_{N,3,j,1,2}}(x,y_{j_{\ast}};r,s,v)&=(-1)\times 
\frac{M(x_j,s)}{Ns}\times \phi(x) \frac{2(1-r/s)}{(r/s)^2}\mathbf E\left[\widetilde{F}\left(S_{M(x_j,s)}\left(\frac{r}{s}\right)\right)\right],\\
\hypertarget{JN2-2-3}{\mathfrak a_{N,3,j,1,3}}(x,y_{j_{\ast}};r,s,v)&=\frac{M(x_j,s)+1}{Ns}\times \phi\left(x+\frac{e_j}{sN^{1/2}}\right) \frac{(1-r/s)^2}{(r/s)^2} \mathbf E\left[\widetilde{F}\left(S_{M(x_j,s)}\left(\frac{r}{s}\right)\right)\right].
\end{align*}
These two considerations yield the next two equalities:
\begin{align}
&\quad  \int_{-N^{1/2}}^\infty \d x_j\phi(x)\int_{-N^{1/2}}^\infty \d y_j\phi(y)\hyperlink{JN3j1}{\mathfrak a_{N,3,j,1}}
(x_j,y_j;r,s,v)-\mathcal R_{N,j}\notag\\
&=\frac{N^{1/2}}{v}\int_{-N^{1/2}-\frac{1}{sN^{1/2}}}^\infty\!\!\! \d x_j\frac{M(x_j,s)+1}{Ns}\times \phi\left(x+\frac{e_j}{sN^{1/2}}\right) \mathbf E\left[\hyperlink{tildeF}{\widetilde{F}}\left(S_{M(x_j,s)}\left(\frac{r}{s}\right)+2\right)\right]\label{shift-origin000}\\
\begin{split}
&=\frac{N^{1/2}}{v}\int_{-N^{1/2}+\frac{1}{sN^{1/2}}}^\infty \d x_j\hyperlink{JN2-2-1}{\mathfrak a_{N,3,j,1,1}} +\frac{N^{1/2}}{v}\int_{-N^{1/2}}^\infty \!\!\d x_j \hyperlink{JN2-2-2}{\mathfrak a_{N,3,j,1,2}}
\\
&\quad +\frac{N^{1/2}}{v}\int_{-N^{1/2}-\frac{1}{sN^{1/2}}}^\infty \d x_j\hyperlink{JN2-2-3}{\mathfrak a_{N,3,j,1,3}}
\label{shift}
\end{split}\\
\begin{split}\label{shift1}
&=\frac{N^{1/2}}{v}\int_{-N^{1/2}}^\infty \d x_j\hypertarget{phibarN}{\overline{\phi}_N}  (x;r,s,v)\mathbf E\left[\hyperlink{tildeF}{\widetilde{F}}\left(S_{M(x_j,s)}\left(\frac{r}{s}\right)\right)\right] \\
& \quad -\frac{N^{1/2}}{v}\int_{-N^{1/2}}^{-N^{1/2}+\frac{1}{sN^{1/2}}}\d x_j\hyperlink{JN2-2-1}{\mathfrak a_{N,3,j,1,1}}
 +\frac{N^{1/2}}{v}\int_{-N^{1/2}-\frac{1}{sN^{1/2}}}^{-N^{1/2}}\d x_j\hyperlink{JN2-2-3}{\mathfrak a_{N,3,j,1,3}}.
\end{split}
\end{align}
For \eqref{shift1}, to sum the integrands in \eqref{shift} with the expectations excluded, we write 
\begin{align}
\begin{split}
\hyperlink{phibarN}{\overline{\phi}_N}  (x;r,s)\;&\defeq\; \frac{M(x_j,s)-1}{Ns}\times \phi\left(x-\frac{e_j}{sN^{1/2}}\right) \frac{1}{(r/s)^2}-\frac{M(x_j,s)}{Ns}\times \phi(x) \frac{2(1-r/s)}{(r/s)^2}\\
&\quad +\frac{M(x_j,s)+1}{Ns}\times \phi\left(x+\frac{e_j}{sN^{1/2}}\right) \frac{(1-r/s)^2}{(r/s)^2}.\label{def:phiN}
\end{split}
\end{align}
Recalling \eqref{def:tildeF-end}, we obtain from \eqref{shift1} that
\begin{align}
&\quad \int_{-N^{1/2}}^\infty \d x_j\phi(x)\int_{-N^{1/2}}^\infty \d y_j\phi(y) \hyperlink{JN3j1}{\mathfrak a_{N,3,j,1}}
(x_j,y_j;r,s,v)\notag\\
\begin{split}
&=\frac{N^{1/2}}{v}\int_{-N^{1/2}}^\infty \d x_j 
\hyperlink{phibarN}{\overline{\phi}_N} (x;r,s)
\int_{-N^{1/2}}^\infty \d y_j \left[\phi(y)-\phi\left(y-\frac{e_j}{vN^{1/2}}\right)\right] vN^{1/2}\\
&\quad \times  N^{1/2}\mathbf P\left(S_{{M(x_j,s)} }\left(\frac{r}{s}\right)=S'_{M(y_j,v)}\left(\frac{r}{v}\right)\right)+\mathcal R_{N,j}\label{RN:final0}.
\end{split}
\end{align}
Here, the new contribution to the $\mathcal R_{N,j}$ in \eqref{RN:final0} comes from the last two terms in \eqref{shift1}: We use the condition $r/s\geq 1/2$ to bound the denominators $(r/s)^2$ in $\hyperlink{JN2-2-1}{\mathfrak a_{N,3,j,1,1}}$ and $\hyperlink{JN2-2-1}{\mathfrak a_{N,3,j,1,3}}$ away from zero. The fast decay property of $\phi$ is also used.

Let us approximate $\hyperlink{phibarN}{\overline{\phi}_N}$ and show two more $\mathcal R_{N,j}$-terms in \eqref{RN:final0}. First, note that $\partial_{x_j}\phi\in \mathcal S(\R^2)$, and the sum of the coefficients of the expectations in \eqref{bin:obs2}  is $1$. By these properties, \eqref{MN-1} and $r/s\geq 1/2$, an approximation of $\hyperlink{phibarN}{\overline{\phi}_N }$ by $\phi$ with an $\mathcal O(1/N^{1/2})$-error holds:
\begin{align}\label{def:phi0}
\hyperlink{phibarN}{\overline{\phi}_N }(x;r,s)=\phi(x)+\hypertarget{phibar0N}{\overline{\phi}_N^{0}} (x;r,s)
,\mbox{ where }|\hyperlink{phibar0N}{\overline{\phi}_N^{0} }(x;r,s)|\leq \frac{C(T_0,T_1,\phi,n)}{N^{1/2}(1+|x|^n)},\; \forall\;n\in \Bbb N. 
\end{align}
Next, to get the two $\mathcal R_{N,j}$-terms,  we first use the mean-zero property, $\int \d y [\phi(y)-\phi(y-\frac{e_j}{vN^{1/2}})]=0$. So the first term in \eqref{RN:final0} is not changed after we subtract $\1_{[1,\infty)}(r)/(4\pi r)$. Then we apply Corollary~\ref{cor:uniformcovar} to this re-centered term: Replacing $\hyperlink{phibarN}{\overline{\phi}_N}$   by $\hyperlink{phibar0N}{\overline{\phi}_N^0}$ results in an $\mathcal R_{N,j}$-term, thanks to the bounds for $\hyperlink{phibar0N}{\overline{\phi}_N^{0} }$ in \eqref{def:phi0}. We get another $\mathcal R_{N,j}$-term from replacing $[\phi(y)-\phi(y-\tfrac{e_j}{v N^{1/2}})]vN^{1/2}$ with $\partial_{y_j}\phi(y)$, since  their difference can be bounded  uniformly in $v\in [s,t]$ and $T_0\leq s\leq t\leq T_1$  in the same way as  $\hyperlink{phibar0N}{\overline{\phi}_N^0}$. 

Our conclusion of Step \hyperlink{J2-1-1}{4-1} is as follows. By the observations in the last paragraph, \eqref{RN:final0}, and \eqref{def:phi0}, the following estimate holds for all $r$ and $v$ such that $r/s\geq 1/2$ and $s\leq v\leq t$:
\begin{align}
\begin{split}
&\quad\int_{-N^{1/2}}^\infty \d x_j\phi(x)\int_{-N^{1/2}}^\infty \d y_j\phi(y)\hyperlink{JN3j1}{\mathfrak a_{N,3,j,1}}
(x_j,y_j;r,s,v)\\
&=\frac{N^{1/2}}{v}\int_{-N^{1/2}}^\infty \d x_j
\phi (x)
\int_{-N^{1/2}}^\infty \d y_j \partial_{y_j}\phi(y)  \mathfrak b_{N,j}(x_j,y_j;r,s,v)+\mathcal R_{N,j}\label{RN:final00}.
\end{split}
\end{align}

\noindent {\bf Step \hypertarget{J2-1-2}{4-2}.} 
For the case $r/s< 1/2$ complementary to that considered in Step \hyperlink{J2-1-1}{4-1}, we turn to \eqref{bin:comp} to change parameters of $S_m(p)$ and $S_{m'}'(p')$ and consider probabilities of $S_m(1-p)$ and $S'_{m'}(1-p')$ alternatively. This replacement is reversible. In this way, the argument from \eqref{shift-origin000} to \eqref{RN:final0} applies similarly.  

In more detail, the modification of  Step \hyperlink{J2-1-1}{4-1} begins with the following replacement for \eqref{RN32} from using \eqref{bin:comp}:
\begin{align}
\begin{split}
&\quad\, \mathbf P\left(S_{{M(x_j,s)} }\left(\frac{r}{s}\right)+2=S'_{M(y_j,v)}\left(\frac{r}{v}\right)\right)\\
&=\mathbf P\left(S_{{M(x_j,s)} }\left(1-\frac{r}{s}\right)+M(y_j,v)-M(x_j,s)=S'_{M(y_j,v)}\left(1-\frac{r}{v}\right)+2\right).\label{P2SM}
\end{split}
\end{align}
Then we modify $\widetilde{F}$ in \eqref{def:tildeF-end} such that the binomial probability is replaced by 
\[
\mathbf P\left(S_{{M(x_j,s)} }\left(1-\frac{r}{s}\right)+M(y_j,v)-M(x_j,s)=n\right).
\]
The condition $r/s<1/2$ ensures that the denominators in the analogues of $a_{N,3,j,1,k}$ satisfy $(1-r/v)^2\geq (1/2)^2$ for all $s\leq v\leq t$. In the corresponding analogue of \eqref{RN:final0}, we reverse the replacements by writing
\begin{align}
\begin{split}\label{P2SM-2}
&\quad\; \mathbf P\left(S_{{M(x_j,s)} }\left(1-\frac{r}{s}\right)+M(y_j,v)-M(x_j,s)=S'_{M(y_j,v)}\left(1-\frac{r}{v}\right)\right)\\
&=\mathbf P\left(S_{{M(x_j,s)} }\left(\frac{r}{s}\right)=S'_{M(y_j,v)}\left(\frac{r}{v}\right)\right).
\end{split}
\end{align}
Hence,  for the present case, we have the same estimate in \eqref{RN:final00}, except that the applicable $r$ and $v$ now have to satisfy $r/s<1/2$ and $s\leq v\leq t$.\smallskip

\noindent {\bf Step \hypertarget{J5}{5}.}
Finally, we apply \eqref{der:analyze}, \eqref{RN:final00}, and the analogue of \eqref{RN:final00} from Step~\hyperlink{J2-1-2}{4-2} to the sum over $j\in \{1,2\}$ in \eqref{JN3-intintint}. Since $\div \phi(y)= \partial_{y_1}\phi(y)+\partial_{y_2}\phi(y)$, it follows that 
\begin{align}\label{JN23-final}
\hyperlink{JN3}{\mathcal J_{N,3}}(s,t)= \hyperlink{JN22}{\mathcal J_{N,2,2}}(s,t)+\widetilde{\mathcal R}_N(s,t),
 \end{align}
where $|\widetilde{\mathcal R}_N(s,t)|\leq C(T_0,T_1,\phi)(t-s)$ for all $T_0\leq s\leq t\leq T_1$ by the definition of $\mathcal R_{N,j}$'s in the last paragraph before Step \hyperlink{J2}{2}. By \eqref{main1} and the validity of \eqref{IJK:bdd} for $\mathcal L_N=\hyperlink{JN21}{\mathcal J_{N,2,1}}$ and $\alpha=1$ from Step \hyperlink{J1}{1}, \eqref{JN23-final} is enough for the proof of the lemma.
\end{proof}

By \eqref{main:JN}, Lemmas~\ref{lem:JN1} and \ref{lem:JN2,3} can be summarized as the following proposition.

\begin{prop}\label{prop:JN}
\eqref{IJK:bdd} for $\mathcal L_N=\hyperlink{JN}{\mathcal J_{N}}$ and $\alpha=1$ holds. 
\end{prop}

The same result holds for $\hyperlink{KN}{\mathcal K_{N}}$ by an almost identical argument if we restart from \eqref{interpolate}. We omit the details.

\begin{prop}\label{prop:KN}
\eqref{IJK:bdd} for $\mathcal L_N=\hyperlink{KN}{\mathcal K_{N}}$ and $\alpha=1$ holds. 
\end{prop}

\section{Stationary additive stochastic heat equation}\label{sec:stat}
In this section, we collect some basic results for the stationary additive stochastic heat equation which are mentioned in Section~\ref{sec:ASHE}, since they seem difficult to find in the literature.

Let $W$ be a two-sided space-time white noise. For any $\phi\in L_2(\R^2,\d x)$, $W(\phi)$ is a two-sided Brownian motion with $W_0(\phi)=0$ and $\E[W_{ 1}(\phi)^2]=\E[W_{- 1}(\phi)^2]=\|\phi\|^2_{L_2(\R^2,\d x)}$. An inspection of the proof of Proposition~\ref{prop:normal2}, especially \eqref{kernel1}, shows that, without the re-centering function,  we would obtain the following divergent integral instead of the first integral in \eqref{covar:zeta}:
\[
\int_{\R^2}\int_{\R^2}Q_{t^{-1}}(y',x)\left(\int_0^\infty Q_{2r}(y',y'')\d r\right)Q_{s^{-1}}(y'',y)\d y'\d y''.
\]
Then at least formally, the initial condition $X_0$ below \eqref{covar:X0} needs to be replaced by
$\widetilde{X}_0(\phi)= \int_{-\infty}^0 \int_{\R^2}Q_{-r}\phi(x)W(\d x,\d r)$ so that the process corresponding to \eqref{main:SPDE} is
\begin{align}\label{Xt:stat}
\widetilde{X}_t(\phi)=\int_{-\infty}^t\int_{\R^2}Q_{t-r}\phi(x)W(\d x,\d r).
\end{align}  
The process in \eqref{Xt:stat} is an analogue of the classical stationary solution of the Ornstein--Uhlenbeck process. The next proposition shows that it is well-defined whenever $\phi\in \mathcal S_0(\R^2)$.  

\begin{prop}\label{prop:stat}
For every $\phi\in \mathcal S_0(\R^2)$, the improper stochastic integral in \eqref{Xt:stat} converges a.s. for any fixed $t\geq 0$ and, as a process, has the same finite-dimensional marginals as $X(\phi)$ defined by \eqref{def:lim}. In particular, $(X_t;t\geq 0)$ as an $\mathcal S'(\R^2)$-valued process is stationary. 
\end{prop}

\begin{proof}
To get the first property of $\widetilde{X}$, note that, for $-\infty<-S\leq -T\leq 0\leq t<\infty$, 
\begin{align}
&\quad \Cov\left[\int_{-S}^t\int_{\R^2}Q_{t-r}\phi(x)W(\d x,\d r);\int_{-T}^t\int_{\R^2}Q_{t-r}\phi(x)W(\d x,\d r)\right]\notag\\
&=\int_{\R^2}\int_{\R^2}\phi(x)\phi(y)\left(\int_0^{T+t} Q_{2r}(x,y)\d r\right)\d x\d y\notag\\
&\xrightarrow[T\to\infty]{}\int_{\R^2}\int_{\R^2}\phi(x)\phi(y)\left[\int_0^{\infty} \left(Q_{2r}(x,y)-\1_{[1,\infty)}(r)\frac{1}{4\pi r}\right)\d r\right]\d x\d y\label{Gaussian-stat1}
\end{align}
by dominated convergence since $\phi\in \mathcal S_0(\R^2)$ and $1-\e^{-u}\leq u$ for all $u\geq 0$. The improper stochastic integral in \eqref{Xt:stat} converges in $L_2(\P)$ by \eqref{Gaussian-stat1},  and so, almost surely by the martingale convergence theorem as in the proof of Proposition~\ref{prop:SDE}.

To see that $X(\phi)$ and $\widetilde{X}(\phi)$ have the same finite-dimensional marginals, we use Lemma~\ref{lem:Tglue} to rewrite the log kernel in the definition \eqref{covar:X0} of  $X_0$. For all $0\leq s\leq t<\infty$,  
\begin{align}
&\quad \Cov[X_s(\phi);X_t(\phi)]\notag\\
&=\frac{1}{2\pi}\int_{\R^2}\int_{\R^2}\phi(x)\phi(y)\left(\int_{\R^2}\int_{\R^2}Q_{s}(y',x)\big(-\ln |y'-y''|\big)Q_{t}(y'',y)\d y'\d y''\right)\d x\d y\notag\\
&\quad +\int_{\R^2}\int_{\R^2}\phi(x)\phi(y)\left(\int_0^{s} \int_{\R^2} Q_{s-r}(z,x)Q_{t-r}(z,y)\d z\d r\right)\d x\d y\notag\\
&=\lim_{T\to\infty}\int_{\R^2}\int_{\R^2}\phi(x)\phi(y)\left(\int_{\R^2}\int_{\R^2}Q_{s}(y',x)\int_0^T Q_{2r}(y',y'')\d rQ_{t}(y'',y)\d y'\d y''\right)\d x\d y\notag\\
&\quad +\int_{\R^2}\int_{\R^2}\phi(x)\phi(y)\left(\int_0^{s} \int_{\R^2} Q_{r}(z,x)Q_{t-s+r}(z,y)\d z\d r\right)\d x\d y\notag\\
&= \lim_{T\to\infty}\int_{\R^2}\int_{\R^2}\phi(x)\phi(y)\left(\int_0^{T+s} Q_{t-s+2r}(x,y)\d r\right)\d x\d y\notag\\
&= \lim_{T\to\infty}\Cov\left[\int_{-T}^s\int_{\R^2}Q_{s-r}\phi(x)W(\d x,\d r);\int_{-T}^t\int_{\R^2}Q_{t-r}\phi(x)W(\d x,\d r)\right],\notag
\end{align}
where in the second equality, we use Lemma~\ref{lem:Tglue}, the assumption $\phi\in \mathcal S_0(\R^2)$, and the dominated convergence theorem. The last equality shows that $ \Cov[X_s(\phi);X_t(\phi)]= \Cov[\widetilde{X}_s(\phi);\widetilde{X}_t(\phi)]$, which is enough for the required identity in finite-dimensional marginals.  
 
Finally, given $\phi_1,\phi_2\in \mathcal S_0(\R^2)$, $\Cov[X_t(\phi_1);X_t(\phi_2)]$ is given by \eqref{Gaussian-stat1} with $\phi(x)\phi(y)$ replaced by $\phi_1(x)\phi_2(y)$. Hence, $\Cov[X_t(\phi_1);X_t(\phi_2)]$ does not depend on $t$. This proves the stationarity of $X$.  
\end{proof}

\end{document}